\documentclass[11pt,reqno]{amsart}

\usepackage{amsmath,amssymb,mathrsfs}
\usepackage{graphicx,cite,
times}
\usepackage{cases}
 \usepackage{dsfont}
\newtheorem{theo}{Theorem}[section]

\newtheorem{lemm}[theo]{Lemma}

\newtheorem{rema}[theo]{Remark}

\newtheorem{assu}{Hypothesis}
\numberwithin{equation}{section}

\begin{document}

\title[Periodic solutions to nonlinear wave equation]{
Periodic Solutions to Nonlinear  Wave equation with $X$-dependent  Coefficients under the General Boundary Conditions }

\author{Bochao Chen}
\address{School of
Mathematics and Statistics, Center for Mathematics and
Interdisciplinary Sciences, Northeast Normal University, Changchun, Jilin 130024, P.R.China}
\email{chenbc758@nenu.edu.cn}


\author{Yong Li}
\address{School of
Mathematics and Statistics, Center for Mathematics and
Interdisciplinary Sciences, Northeast Normal University, Changchun, Jilin 130024, P.R.China. }
\address{College of Mathematics, Jilin University, Changchun 130012, P.R.China}
\email{yongli@nenu.edu.cn}

\author{Xue Yang}
\address{School of
Mathematics and Statistics, Center for Mathematics and
Interdisciplinary Sciences, Northeast Normal University, Changchun, Jilin 130024, P.R.China. }
\email{yangx100@nenu.edu.cn}

\thanks{ The research  of YL was supported in part by NSFC grant 11571065, 11171132  and  National Research Program of China Grant 2013CB834100}


\begin{abstract}
In this paper we consider a class of nonlinear  wave equation with  $x$-dependent  coefficients  and prove  existence of families of time-periodic solutions
under the general boundary conditions.
 Such a model arises from the forced vibrations of a nonhomogeneous string and the propagation of seismic waves in nonisotropic media.
The proofs are based on
 a Lyapunov-Schmidt reduction together with a  differentiable   Nash-Moser iteration scheme.
\end{abstract}

\maketitle

\section{Introduction}
This paper is devoted to the study  of time-periodic solutions to
 nonlinear wave equation subject to  the general boundary conditions
\begin{align}\label{E1.1}
\begin{cases}
\rho(x)u_{tt}-(p(x)u_{x})_{x}+m(x)u=\epsilon f(\omega t,x,u),\\
\alpha_1u(t,0)-\beta_1u_{x}(t,0)=0,\\
\alpha_2u(t,\pi)+\beta_2u_{x}(t,\pi)=0,
\end{cases}
\end{align}
where $\alpha^2_i+\beta^2_i\neq0,i=1,2$, $\rho(x)>0$, $p(x)>0$, $\epsilon$ is a small parameter, the potential $m(x)>0$,
 and the nonlinear forcing term $f(\omega t,x, u)$ is $(2\pi/\omega)$-periodic in time, i.e. $f(\cdot,x, u)$ is $2\pi$-periodic.

Equation \eqref{E1.1} with $\rho(x),p(x)$ depending on $x$ is a  more realistic model, which
describes the forced vibrations of a bounded nonhomogeneous string and the propagation of seismic waves in non-isotropic media, see \cite{Bamberger1979about,Barbu1996Periodic,Barbu1997Periodic,Barbu1997determining,ji2006periodic,
ji2007periodic,ji2008time,ji2009peridic,ji2011time,baldi2008forced}.
More precisely, the vertical displacement $u(t,x)$ of a plane seismic waves at time $t$ and depth $x$  is described by the following equation
\begin{equation*}
\rho(x)u_{tt}-(p(x)u_{x})_{x}=0,
\end{equation*}
where $\rho(x)$ is the rock density, $p(x)$ is the elasticity coefficient.

The search for periodic solutions to nonlinear wave equations has a long standing tradition.
If the coefficients $\rho(x),p(x)$ are nonzero constants,  Equation \eqref{E1.1} corresponds to the classical wave equation.
The problem of finding time-periodic solutions to the classical nonlinear wave equation has received wide attention due to the first pioneering work
of Rabinowitz \cite{Rabinowitz1967periodic,rabinowitz1968periodic,Rabinowitz1971,Rabinowitz1978}.
Provided that the nonlinearity $f$ is monotonic in $u$, he \cite{Rabinowitz1967periodic} rephrased the problem as a variational problem and verified the existence of periodic solutions whenever the time period $T$ is a rational
multiple of the length of spatial interval.
 Subsequently, Rabinowitz's
variational methods was developed by Bahri, Br\'{e}zis, Corn, Nirenberg etc., and
  many related results was obtained, see \cite{bahri1980periodic,brezis1983periodic,brezis1981periodic,brezis1978forced}.
  In these papers, time
period $T$ is required to be a rational multiple of $\pi$, i.e. the  frequency $\omega$ has to be rational. When the  forced frequency $\omega$ is irrational,
the spectrum of the wave operator $\omega^2\partial_{tt}-\partial_{xx}$ approaches to zero for almost every $\omega$, which will leads to  a ``small denominators problem''. The case in which   $\omega$ is irrational
  has been investigated by Fe\v{c}kan \cite{mickan1995periodic}
and McKenna\cite{mckenna1985Osolutions}, where the frequencies are essentially the numbers whose continued fraction expansion is bounded.
At the end of the 1980s, a quite different approach
which used the Kolmogorov-Arnold-Moser (KAM) theory was developed from the
viewpoint of infinite dimensional dynamical systems by Kuksin \cite{kuksin1987hamiltonian}, Eliasson\cite{eliasson1988perturbations} and Wayne
\cite{wayne1990periodic}. This method allows one to obtain solutions whose periods are irrational multiples
of the length of the spatial interval.
 In addition, this method is easily
extended to construct quasi-periodic solutions, see \cite{chierchia2000kam,liu2010spectrum,yuan2006quasi-periodic,Gao2009,gengyou2006kam,Geng2013}. Later, in order to overcome some limitations inherent to the usual KAM procedures, Craig,
Wayne \cite{craig1993newton} applied a novel method based on a Lyapunov-Schmidt decomposition and the Nash-Moser technique to  construct the periodic
of the wave equation with the Dirichlet boundary conditions
 and periodic boundary conditions. Bourgain successfully constructed the periodic or quasi-periodic solutions of the wave equation using the similar method,
 see \cite{bourgain1994construction,bourgain1995construction}. The advantage of this approach is to require only the ``first order Melnikov'' non-resonance conditions, which are essentially the minimal assumptions.
Some recent results on  Nash-Moser theorems can be found in \cite{berti2006cantor,berti2008cantor,berti2010sobolev,Berti2012nonlinearity,berti2015abstract} and the references there in.

On the other hand, the problem of finding periodic solutions to
equation \eqref{E1.1} with $\rho(x)=p(x)$ depending on $x$
was firstly considered by Barbu and Pavel in \cite{Barbu1996Periodic,Barbu1997Periodic,Barbu1997determining}. In \cite{Barbu1997Periodic}, if $\omega$ is rational, then the spectrum of the linear operator has the following form
\begin{equation*}
-\omega^2l^2+\lambda_j=-\omega^2l^2+j^2+b+O({1}/{j}), \quad l\in\mathbf{Z},~j\geq1.
\end{equation*}
Under the assumption  $b\neq0$, the linear spectrum possesses at most finite many zero eigenvalues and the other eigenvalues are far away from zero. For $b=0$, $\omega\in\mathbf{Q}$,  the infinite eigenvalues tend to zero for $|l|=\frac{\mathfrak{n}}{\mathfrak{m}}j$.
Under assuming $b\neq0$, Ji
and Li obtained a series of results on looking for periodic solutions to
equation \eqref{E1.1} with $\rho(x)=p(x)$  under  the general boundary conditions and  periodic  boundary conditions, see \cite{ji2006periodic,
ji2007periodic,ji2008time,ji2009peridic}.
The case of $b=0$ was also posed as an open problem by Barbu and Pavel in \cite{Barbu1997Periodic}. The difficulty arising from $b=0$ have been actually overcame
  by Ji and Li in \cite{ji2011time} when the forced frequency $\omega$ is rational.
 For the forced frequency $\omega$ is irrational, the small denominators phenomenon arises.
 For the forced frequency $\omega$
is irrational with $b\neq0$,  Baldi and Berti \cite{baldi2008forced} investigated equation \eqref{E1.1} under Dirichlet boundary conditions,
 where the nonlinearity $f$ was assumed to be  analytic in $(t,u)$ and $H^1$ in $x$.
  They proved the existence of time-periodic solutions of the equation via a Lyapunov-Schmidt reduction together with a analytic Nash-Moser iteration scheme.
In this paper,  we will consider the case of the forced frequency $\omega$ is irrational  and $b\neq0$. There are two main difficulties in this work:
 (i)  the finite differentiable  regularities of the nonlinearity. All above results are carried out in analytic nonlinearities cases.
  However, there is no existence result of nonlinear wave equation (with
 $x$ dependent coefficients) with perturbations having only finitely differentiable regularities presently, which is the main motivation for this paper.
 (ii) the more general boundary conditions,
 which contain Dirichlet boundary conditions, Neumann boundary conditions, Dirichlet-Neumann boundary conditions and the general boundary condition. We have to give the
 asymptotic properties of the eigenvalues for different boundary conditions.
 Applying a differentiable  Nash-Moser method \cite{berti2008cantor,berti2010sobolev,Berti2012nonlinearity,berti2015abstract}, under  the ``first order Melnikov'' non-resonance conditions,
 we  obtain the  existence of time-periodic solutions to equation \eqref{E1.1}.

The rest of the paper is organized  as follows: we decompose equation \eqref{E1.2} as the bifurcation equation $(Q)$ and the range equation $(P)$
by a Lyapunov-Schmidt reduction and state the main result in subsection \ref{sec:2.2}.
In subsection \ref{sec:2.3}, the $(Q)$-equation is solved  by the classical implicit function theorem under Hypothesis \ref{hy1}.
Based on the type of boundary conditions, we give the asymptotic formulae for  the eigenvalues of Sturm-Liouville problem \eqref{E6.1} in section \ref{sec:5}.
 Relayed on a differentiable  Nash-Moser iteration scheme, we devote section \ref{sec:6} to solve the $(P)$-equation  under the ``first order Melnikov'' non-resonance conditions.
  In subsection \ref{sec:6.0}, we give the properties $(P1)$-$(P5)$. The inversion of the linearized operators is the core of  the differentiable  Nash-Moser iteration.
 The aim of subsection \ref{sec:6.1} is to verify inversion of the linearized operators (see $(P5)$). In subsection \ref{sec:6.2},
 we give the inductive lemma (see Lemma \ref{lemma4.1}).
 At the end of the construction, we obtain a large measure Cantor-like set $B_\gamma$ in subsection \ref{sec:6.3}.
  Finally, in section \ref{sec:8}, we list the  the proof of some related results for the sake of completeness.

\section{Main results}\label{sec:2}
In this section, we decompose equation \eqref{E1.2} into the bifurcation equation $(Q)$ and the range equation $(P)$ by a Lyapunov-Schmidt reduction.
\subsection{Notations}
Rescaling the  time  $t\rightarrow\frac{t}{\omega}$,  we consider  the existence of $2\pi$-periodic solutions in time of
\begin{align}\label{E1.2}
\omega^2\rho(x)u_{tt}-(p(x)u_{x})_{x}+m(x)u=\epsilon f(t,x,u)
\end{align}
with the corresponding boundary condition
\begin{align*}
\alpha_1u(t,0)-\beta_1u_{x}(t,0)=0,\quad\alpha_2u(t,\pi)+\beta_2u_{x}(t,\pi)=0,
\end{align*}
where $\epsilon\in[0,\epsilon_0]$ with  given positive  constant $\epsilon_0$. The $2\pi$-periodic forcing term $f$ is $H^1$ in $x$ and $C^k$ in $(t,u)$ for some $k\in\mathbf{N}$ large enough, more precisely $f\in\mathcal{C}_k$, where
\begin{align*}
\mathcal{C}_k:=\left\{f\in C(\mathbf{T}\times [0,\pi]\times\mathbf{R};\mathbf{R})~~:~~(t,u)\mapsto f(t,\cdot,u)\in C^k(\mathbf{T}\times\mathbf{R};H^1(0,\pi))\right\}
\end{align*}
with $\mathbf{T}:=\mathbf{R}/(2\pi\mathbf{Z})$. It follows from  the  continuously embedding  of  ${H}^1(0,\pi)$ into   $C([0,\pi];\mathbf{R})$ and the definition of $\mathcal{C}_k$    that $\partial_{t}^{i}\partial_{u}^{j}f\in C(\mathbf{T}\times[0,\pi]\times\mathbf{R};\mathbf{R}),~\forall ~0\leq i,j\leq k$.
 In addition, if $f(t,x,u)=\sum_{l\in\mathbf{Z}}f_{l}(x,u)e^{{\rm i}lt}$, then $u\mapsto f_{l}(\cdot,u)\in C^{k}(\mathbf{R};H^{1}(0,\pi))$ with $f_{-l}=f_{l}^*$.
\begin{rema}
It is obvious that
if $f(t,x,0)\neq0$, then $u=0$ is not the solution of equation \eqref{E1.2}.
\end{rema}
Denote by $H^{s}$ with $s>\frac{1}{2}$ the following Sobolev space
\begin{align*}
H^{s}:=\bigg\{u\mid~&u(t,x)=\sum\limits_{l\in\mathbf{Z}}u_{l}(x)e^{{\rm i}lt},\quad u_l\in \mathcal{H}^1_g,\quad u_{-l}=u^*_{l},\\
&\|u\|^{2}_{s}:=\sum\limits_{l\in\mathbf{Z}}\|u_l\|^2_{H^1}(1+ l^{2s})<+\infty \bigg\},
\end{align*}
where
\begin{align*}
\mathcal{H}^1_g:=\mathcal{H}^1_g((0,\pi);\mathbf{C}):=\bigg\{u\mid~&u\in H^1((0,\pi);\mathbf{C}):~\alpha_1u(t,0)-\beta_1u_{x}(t,0)=0,\\
&\alpha_2u(t,\pi)+\beta_2u_{x}(t,\pi)=0~\text{with}~\alpha^2_i+\beta^2_i\neq0,i=1,2\bigg\}.
\end{align*}
Our aim is to look for solutions defined on $\mathbf{T}\times[0,\pi]$ of equation \eqref{E1.2} in $H^s$.
\begin{rema}\label{remark1}
Let $C(s)$ denote a constant  depending on $s$. The Sobolev space ${H}^s$ with $s>\frac{1}{2}$ has the following properties:
\begin{align*}
&(\mathrm{i})~\|uv\|_{{s}}\leq C(s)\|u\|_{s}\|v\|_{s}, \quad\forall u,v\in{H}^s;\\
&(\mathrm{ii})~\|u\|_{L^{\infty}(\mathbf{T};{H}^1(0,\pi))}\leq C(s)\|u\|_{s},\quad\forall u\in H^s.
\end{align*}
\end{rema}
\begin{proof}
The proof  is postponed to  Appendix \ref{prm1}.
\end{proof}

\subsection{The Lyapunov-Schmidt reduction}\label{sec:2.2}
For any $u\in H^s$, $u$ can be  written as the sum of $u_0(x)+\bar{u}(t,x)$, where $\bar{u}(t,x)=\sum_{l\neq0}u_{l}(x)e^{\mathrm{i}lt}$.
Then we perform the Lyapunov-Schmidt reduction with respect to the following decomposition
\begin{equation*}
H^s=(V\cap H^s)\oplus (W\cap H^s)=V\oplus (W\cap H^s),
\end{equation*}
where
\begin{equation*}
V:=\mathcal{H}^1_g,\quad W:=\left\{ w=\sum\limits_{l\neq 0}w_{l}(x)e^{\mathrm{i}lt}\in H^0\right\}.
\end{equation*}
Denote  by $\Pi_{V}$, $\Pi_{W}$  the projectors onto $V$ and $W$ respectively. Letting  $u=v+w$ with $v\in V$, $w\in W$,  equation \eqref{E1.2} is equivalent to the bifurcation equation $(Q)$ and the range equation $(P)$:
\begin{align}\label{E1.4}
\begin{cases}
-(pv')'+m(x)v=\epsilon\Pi_{V}F(v+w)\quad &(Q),\\
L_{\omega}w=\epsilon\Pi_{W}F(v+w)\quad &(P),
\end{cases}
\end{align}
where
\begin{equation*}
L_{\omega}w:=\omega^2\rho(x)w_{tt}-(p(x)w_{x})_{x}+m(x)w,\quad F:u\rightarrow f(t,x,u).
\end{equation*}
In the same way, the nonlinearity $f$ can be written into
\begin{equation*}
f(t,x,u)=f_{0}(x,u)+\bar{f}(t,x,u),
\end{equation*}
where $\bar{f}(t,x,u)=\sum_{l\neq0}f_{l}(x,u)e^{{\rm i}lt}$. This indicates that  for $w=0$
 \begin{equation*}
\Pi_{V}F(v)=\Pi_{V}f(t,x,v)=\Pi_{V}f_{0}(x,v(x))+\Pi_{V}\bar{f}(t,x,v(x))=f_{0}(x,v(x)).
\end{equation*}
 Then the $(Q)$-equation is the time-independent equation
\begin{equation}\label{E1.5}
-(pv')'+m(x)v=\epsilon f_{0}(x,v)\quad \text{as}~w\rightarrow 0.
\end{equation}
We call \eqref{E1.5} the infinite-dimensional ``zeroth-order bifurcation equation''. This is a second order ODE with the corresponding  general boundary conditions.
The following hypothesis is required to make.
\begin{assu}\label{hy1}
The following system
\begin{align}\label{E1.6}
\begin{cases}
-(p(x)v'(x))'+m(x)v(x)=\epsilon f_{0}(x,v(x)),\\
\alpha_1v(0)-\beta_1v'(0)=0,\\
\alpha_2v(\pi)+\beta_2v'(\pi)=0
\end{cases}
\end{align}
admits a nondegenerate solution $\hat{{v}}\in \mathcal{H}^1_g$, i.e. the linearized equation
\begin{equation}\label{E1.7}
-(ph')'+mh=\epsilon f'_0(\hat{{v}})h
\end{equation}
possesses only the trivial solution $h=0$ in $\mathcal{H}^1_g$.
\end{assu}
Let us explain the rationality of Hypothesis \ref{hy1}. The linearized equation \eqref{E1.7} possesses only the trivial solution $h=0$ for $\epsilon=0$.
Thus the trivial solution $\hat{v}$ of \eqref{E1.6} with $\epsilon=0$ is nondegenerate. Due to the implicit function theorem, Hypothesis \ref{hy1} is satisfied for $\epsilon$ small enough.
This  fact implies that there exists a constant $\epsilon_0>0$ small enough, such that, for all $\epsilon\in[0,\epsilon_0]$, Hypothesis \ref{hy1} holds.
\begin{rema}
Under Hypothesis \ref{hy1}, the $(Q)$-equation may be solved by the classical implicit function theorem.
This idea is based on the technique of Craig and Wayne in \cite{craig1993newton}.
 By means of assuming the existence of a nondegenerate solution of the ``zeroth-order bifurcation equation'',
 the finite dimensional bifurcation equation can be solved. The main reason that we apply the technique is the presence of the Cantor set of “good” parameters caused by solving the range equation $(P)$ via a Nash-Moser iteration. It's very difficult to guarantee that the critical points lie in the Cantor set of “good” parameters if we use variational methods in the case of PDEs.
\end{rema}
Let us state our main theorem as follows.
\begin{theo}\label{Th1}
Assume  that Hypotheses \ref{hy1} and \ref{hy2} (see \eqref{E3.2}) hold for $\hat{\epsilon}\in[0,{\epsilon}_0]$. Set $m(x)\in H^1(0,\pi)$, $p(x),\rho(x)\in H^3(0,\pi)\hookrightarrow C^2[0,\pi]$
with $p(x)>0,\rho(x)>0$, $f\in \mathcal{C}_k$, and fix $\tau\in(1,2),\gamma\in(0,1)$.
If $\frac{\epsilon}{\gamma^5\omega}$ small enough,
 then there exist a constant $K>0$ depending on $\rho, p, m, f, \epsilon_0, \hat{v}, \gamma,\tau, s,\beta, \gamma_0$, a neighborhood $[\epsilon_1,\epsilon_2]$ of $\hat{\epsilon}$, $0<r\leq1$ and a $C^2$ map $v(\epsilon,w)$ defined on $[\epsilon_1,\epsilon_2]\times\left\{w\in W\cap H^s:~\|w\|_{s}\leq r\right\}$, with values in $H^1$, a map
$
\tilde{w}\in C^1(A_0;W\cap H^s)
$
with
\begin{equation*}
\|\tilde{w}\|_{s}\leq\frac{K\epsilon}{\gamma\omega},
\quad\|\partial_{\omega}\tilde{w}\|_{s}\leq\frac{K\epsilon}{\gamma^5\omega},
\quad\|\partial_{\epsilon}\tilde{w}\|_{s}\leq\frac{K}{\gamma^5\omega},
\end{equation*}
and a Cantor-like set $B_{\gamma}\subseteq A_0$ of positive measure satisfying
\begin{equation*}
|B_{\gamma}\cap\Omega|\geq(1-K\gamma)|\Omega|,
\end{equation*}
 such that, for all $(\epsilon,\omega)\in B_\gamma$,
\begin{equation*}
\tilde{u}:=v(\epsilon,\tilde{w}(\epsilon,\omega))+\tilde{w}(\epsilon,\omega)\in \mathcal{H}^1_g\oplus(W\cap H^s)
\end{equation*}
is a solution of equation \eqref{E1.4}, where $\Omega:=(\epsilon',\epsilon'')\times(\omega',\omega'')$ stands for a rectangle contained in $(\epsilon_1,\epsilon_2)\times(2\omega_0,+\infty)$, $A_0$ and $B_\gamma$ are defined by \eqref{E4.30} and \eqref{E4.31} respectively.
\end{theo}

\subsection{Solution of the bifurcation equation}\label{sec:2.3}
In this subsection, under Hypothesis \ref{hy1}, we will solve the ($Q$)-equation relaying on the classical implicit function theorem.
\begin{lemm}\label{lemma3.1}
Fix $\hat{\epsilon}\in[0,\epsilon_0]$. There exist a neighborhood $[\epsilon_1,\epsilon_2]$ of $\hat{\epsilon}$, $0<r\leq1$ and a $C^2$ map $v(\epsilon,w)$ defined on $[\epsilon_1,\epsilon_2]\times\left\{w\in W\cap H^s:~\|w\|_{s}\leq r\right\}$, with values in $V$, such that $v(\epsilon,w)$ solves the ($Q$)-equation in \eqref{E1.4}.
\end{lemm}
\begin{proof}
According to Hypothesis \ref{hy1}, the linear operator
\begin{equation*}
h\mapsto -(ph')'+mh-\hat{\epsilon}\mathrm{D}_v\Pi_{V}F(\hat{v})[h]=-(ph')'+mh-\hat{\epsilon}F'_0(\hat{v})h
\end{equation*}
is invertible from $V$ to $V$. In addition, Lemma \ref{lemma2.4} implies that the following map
\begin{equation*}
(\epsilon,w,v)\mapsto -(pv')'+mh-\epsilon \Pi_{V}F(v+w)
\end{equation*}
belongs to  $C^2([\epsilon_1,\epsilon_2]\times (W\cap H^s)\times V;V)$. Therefore, by the implicit function theorem, there is a $C^2$-path
\begin{equation*}
(\epsilon,w)\mapsto v(\epsilon,w)
\end{equation*}
such that $v(\epsilon,w)$ is a solution of the ($Q$)-equation in \eqref{E1.4} with $v(\hat{\epsilon},0)=\hat{v}$.
\end{proof}
\begin{rema}\label{Remark2}
Lemma \ref{lemma3.1} indicates that  $v, {\rm D}_{w}v, \partial_{\epsilon}v, \partial_{\epsilon}\mathrm{D}_wv, \partial^2_{\epsilon}v, {\rm D}^2_{w}v$ belong to $\mathcal{H}^1_g$,
where ${\rm D}_{w} $ denotes the Fr\'{e}chet derivative with respect to $w$.
\end{rema}

\section{The general  boundary value}\label{sec:5}
Before solving  the $(P)$-equation, we first propose  the asymptotic formulae of the eigenvalues to  the following Sturm-Liouville problem
\begin{align}\label{E6.1}
\begin{cases}
-(py')'+my-\epsilon\Pi_{V}f'(v(\epsilon,w)+w)y=\lambda\rho y,\\
\alpha_1y(0)-\beta_1y_{x}(0)=0,\\
\alpha_2y(\pi)+\beta_2y_{x}(\pi)=0.
\end{cases}
\end{align}
Denote by ${\lambda}_j(\epsilon,w),j\geq0$ the eigenvalues of \eqref{E6.1}.
\begin{rema}
Based on the invertibility of the linearized operators $\mathcal{L}_{N}(\epsilon,\omega,w)$ (see \eqref{E4.32}) on the  $(P)$-equation, we have to consider the Sturm-Liouville problem \eqref{E6.1}.
The related application can be seen in subsection \ref{sec:6.1}.
\end{rema}

\subsection{The Liouville substitution}
Let
\begin{equation}\label{E6.36}
d(x):=m(x)-\epsilon\Pi_{V}f'(t,x,v(\epsilon,w(t,x))+w(t,x))\in L^2(0,\pi).
\end{equation}
Make the Liouville substitution
\begin{equation}\label{E6.26}
x=\psi(\xi)\Leftrightarrow \xi=g(x)\quad\text{with}~ g(x):=\frac{1}{c}\int^{x}_{0}\left(\frac{\rho(s)}{p(s)}\right)^{\frac{1}{2}}\mathrm{d}s,
\end{equation}
where
\begin{align}\label{E6.23}
c:=\frac{1}{\pi}\int_{0}^{\pi}\left({\rho/p}\right)^{\frac{1}{2}}~\mathrm{d}x, \quad \xi\in[0,\pi],
\quad  g(0)=0, \quad g(\pi)=\pi,
\end{align}
$\lambda$ and $y(\psi(\xi))$ satisfy
\begin{align}\label{E6.2}
\begin{cases}
-y_{\xi\xi}-c\frac{(\sqrt{p\rho})'}{\rho}y_{\xi}+c^2\frac{d}{\rho}y=c^2\lambda y,\\
\alpha_1y(0)-\frac{\beta_1}{c}\sqrt{\rho(0)/p(0)}y_{\xi}(0)=0,\\
\alpha_2y(\pi)+\frac{\beta_2}{c}\sqrt{\rho(\pi)/p(\pi)}y_{\xi}(\pi)=0.
\end{cases}
\end{align}
Moreover we make the Liouville change
\begin{equation}\label{E6.35}
y=z/\mathfrak{s},
\end{equation}
where $\mathfrak{s}(x)=(p(x)\rho(x))^{\frac{1}{4}}$. Note that  $\mathrm{d}x=c(\frac{p(x)}{\rho(x)})^{\frac{1}{2}}~\mathrm{d}\xi$, which leads to
\begin{equation*}
y_{\xi}=\frac{z_{\xi} \mathfrak{s}-z\mathfrak{s}_{\xi}}{\mathfrak{s}^2}=\frac{z_{\xi} \mathfrak{s}-cz\mathfrak{s}_{x}(\frac{p}{\rho})^{\frac{1}{2}}}{\mathfrak{s}^2}.
\end{equation*}
Therefore the system \eqref{E6.2} can be reduced into
\begin{align}\label{E6.42}
\begin{cases}
-z_{\xi\xi}(\xi)+\varrho(\xi)z(\xi)=\mu z(\xi),\\
a_1z(0)-b_1z_{\xi}(0)=0,\\
a_2z(\pi)+b_2z_{\xi}(\pi)=0,
\end{cases}
\end{align}
where
\begin{align}
&\varrho(\xi)=q(\xi)+\vartheta(\xi),\label{E7.1}\\
&q(\xi)=c^2Q(\psi(\xi))\quad\mathrm{with}\quad Q(x)=\frac{p(x)\mathfrak{s}_{xx}(x)}{\rho(x)\mathfrak{s}(x)}+\frac{1}{2}\left(\frac{p(x)}{\rho(x)}\right)_{x}\frac{\mathfrak{s}_{x}(x)}{\mathfrak{s}(x)},\label{E7.2}\\
&\vartheta(\xi)=c^2\frac{d(\psi(\xi))}{\rho(\psi(\xi))},\quad\mu=c^2\lambda,\label{E6.6}\\
&a_1=\alpha_1+\frac{\beta_1}{4}\frac{(p\rho)_{x}(0)}{(p\rho)(0)},~b_1=\frac{\beta_1}{c}(\rho(0)/p(0))^{\frac{1}{2}},\quad a_2=\alpha_2-\frac{\beta_2}{4}\frac{(p\rho)_{x}(\pi)}{(p\rho)(\pi)},~ b_2=\frac{\beta_2}{c}(\rho(\pi)/p(\pi))^{\frac{1}{2}}.\label{E7.3}
\end{align}
We have to make an additional hypothesis:
\begin{assu}\label{hy2}
\begin{align}\label{E3.2}
\mathrm{(i)}~\varrho(\xi)>0,\quad\forall \xi\in[0,\pi];\quad\mathrm{(ii)}~a_i\geq0,b_i\geq0\quad\text{with}~a^2_i+b^2_i\geq0,\quad i=1,2.
\end{align}
\end{assu}
Now consider the following Sturm-Liouville problems
\begin{align}
&\varphi''_{n}(\xi)+(\mu_n-\varrho(\xi))\varphi_n(\xi)=0,\label{E6.7}\\
&a_1\varphi_n(0)-b_1\varphi'_n(0)=0,\quad
a_2\varphi_n(\pi)+b_2\varphi'_n(\pi)=0, \label{E6.8}
\end{align}
where $a_i\geq0,b_i\geq0,i=1,2$ are defined by \eqref{E7.3}, $\varphi'_{n}(\xi)=\frac{\mathrm{d}}{\mathrm{d}\xi}\varphi_n(\xi)$, $n\in\mathbf{N}$. By \eqref{E6.42}-\eqref{E7.1}, it shows that $\mu_n$ depends on $\vartheta(\cdot)$. However, for brevity, we do not write $\vartheta(\cdot)$.
The eigenvalues of the Sturm-Liouville problem \eqref{E6.7}-\eqref{E6.8} have the following properties.
\begin{theo}{\cite[Theorem 2.1]{coddington1972theory}}
The eigenvalues of \eqref{E6.7}-\eqref{E6.8} form an infinite number sequence with $\mu_0<\mu_1<\cdots< \mu_n < \cdots$, and $\mu_n\rightarrow+\infty$ as $n\rightarrow+\infty$.
 In addition, the eigenfunctions $\varphi_n,n\in\mathbf{N}$ with respect to $\mu_n $ have exactly $n $ zeros on $(0,\pi)$.
\end{theo}

\begin{lemm}\label{lemma6.2}
If Hypothesis \ref{hy2} holds, then $\mu_n\geq\varrho_0$ for $n\in\mathbf{N}$, where
\begin{equation}\label{E6.16}
\varrho_0:=\inf_{\xi\in[0,\pi]}\varrho(\xi)>0.
\end{equation}
In particular, $\mu_n>\varrho_0$  if  $a^2_1+a^2_2>0$.
\end{lemm}
\begin{proof}
Multiplying both sides of \eqref{E6.7} by $\varphi_n(\xi)$ and   integrating over $[0,\pi]$, it  yields that
\begin{align}\label{E6.9}
\mu_n\int^\pi_{0}\varphi^2_n(\xi)\mathrm{d}\xi&=\int^\pi_0\varrho(\xi)\varphi^2_n(\xi)\mathrm{d}\xi-\int^\pi_0\varphi''_{n}(\xi)\varphi_{n}(\xi)\mathrm{d}\xi \nonumber\\
&=\int^\pi_0\varrho(\xi)\varphi^2_n(\xi)\mathrm{d}\xi+\int^\pi_0(\varphi'_{n}(\xi))^2\mathrm{d}\xi-
\left(\varphi_n(\pi)\varphi'_n(\pi)-\varphi_n(0)\varphi'_n(0)\right).
\end{align}
Multiplying the first term of \eqref{E6.8} by $\varphi'_n(0)-\varphi_n(0)$ gives
\begin{equation*}
\varphi_n(0)\varphi'_n(0)=\frac{1}{a_1+b_1}(a_1\varphi^2_n(0)+b_1(\varphi'_n(0))^2)\geq0.
\end{equation*}
On the other hand, we obtain
\begin{equation*}
\varphi_n(\pi)\varphi'_n(\pi)=-\frac{1}{a_2+b_2}(a_2\varphi^2_n(\pi)+b_2(\varphi'_n(\pi))^2)\leq0
\end{equation*}
via multiplying the second equality of \eqref{E6.8} by $\varphi'_n(\pi)+\varphi_n(\pi)$. Then
$\varphi_n(\pi)\varphi'_n(\pi)-\varphi_n(0)\varphi'_n(0)\leq0$.
The combination of above estimates establishes
\begin{equation*}
\mu_n\int^\pi_{0}\varphi^2_n(\xi)\mathrm{d}\xi\geq\int^\pi_0\varrho(\xi)\varphi^2_n(\xi)\mathrm{d}\xi
\geq\varrho_0\int^\pi_0\varphi^2_n(\xi)\mathrm{d}\xi,
\end{equation*}
which implies $\mu_n\geq\varrho_0$.

In what follows, we further prove that $\mu_n > \rho_0$ if  $a^2_1+a^2_2>0$. Supposed by contrary that $\mu_n=\varrho_0$, formula \eqref{E6.9} leads to
\begin{equation*}
\begin{aligned}
&\int^\pi_0(\varrho(\xi)-\varrho_0)\varphi^2_n(\xi)\mathrm{d}\xi+\int^\pi_0(\varphi'_{n}(\xi))^2\mathrm{d}\xi=
\varphi_n(\pi)\varphi'_n(\pi)-\varphi_n(0)\varphi'_n(0)\\
\Longrightarrow&\int^\pi_0(\varrho(\xi)-\varrho_0)\varphi^2_n(\xi)\mathrm{d}\xi+\int^\pi_0(\varphi'_{n}(\xi))^2\mathrm{d}\xi=0\\
\Longrightarrow&\int^\pi_0(\varphi'_{n}(\xi))^2\mathrm{d}\xi=0\\
\Longrightarrow&\varphi'_{n}(\xi)=0.
\end{aligned}
\end{equation*}
Since $\varphi_{n}(\xi)$ is the eigenfunction, it checks that $\varphi_{n}(\xi)=c\neq0$. Plugging it back into the boundary conditions \eqref{E6.8}, we derive $a_1=a_2=0$.
 This leads to a contradiction to $a^2_1+a^2_2>0$. Hence $\mu_n>\varrho_0$.
\end{proof}
Since the eigenvalues $\mu_n,n\in\mathbf{N}$ of \eqref{E6.7}-\eqref{E6.8} are different when the type of boundary conditions is different, we restrict our attention to the following four cases:

Case 1: $a_1=0$, $b_1>0$, $a_2=0$, $b_2>0$.

Case 2: $a_1>0$, $b_1=0$ $a_2=0$, $b_2>0$.

Case 3: $a_1=0$, $b_1>0$ $a_2>0$, $b_2=0$.

Case 4: $a_1>0$, $b_1>0$ $a_2>0$, $b_2>0$.

\begin{rema}
Case 1 is  called the Neumann boundary conditions,   Case 2 and 3  are called Dirichlet-Neumann boundary conditions, and  we call case 4 the general boundary conditions.
In fact, the result about the Dirichlet boundary condition with finite differentiable nonlinearities can also be obtained.
 However, the Dirichlet boundary condition with  analytical nonlinearities  have been investigated in
 \cite{baldi2008forced}. Thus, we only consider above four cases.
\end{rema}

\subsection{Neumann boundary value problem}\label{sec:4.1}

In Case 1, the Sturm-Liouville problem \eqref{E6.7}-\eqref{E6.8} may be written as
\begin{equation}\label{E6.10}
\left\{
\begin{aligned}
&\varphi''_{n}(\xi)+(\mu_n-\varrho(\xi))\varphi_n(\xi)=0,\\
&\varphi'_n(0)=\varphi'_n(\pi)=0.
\end{aligned}
\right.
\end{equation}
\begin{lemm}\label{lemma6.3}
Denote by $\mu_0<\mu_1<\cdots$ and $\varphi_0,\varphi_1,\cdots$ the eigenvalues and orthonormal eigenfunctions of the Sturm-Liouville problem (\ref{E6.10}) respectively. Then, for $n\in\mathbf{N}^{+}$, we have the following asymptotic formulae
\begin{equation}\label{E6.11}
\mu_n=n^2+c_0+O\left(\frac{1}{n^2}\right)
\end{equation}
with $\varrho_0\leq c_0\leq\varrho_1$, where $\varrho_0$ is defined in \eqref{E6.16} and
\begin{equation}\label{rh1}
\varrho_1:=\frac{2}{\pi}\int^{\pi}_0\varrho(\xi)~\mathrm{d}\xi.
\end{equation}
\end{lemm}
\begin{proof}
Before approching the lemma, we first claim for $n\in\mathbf{N}^{+}$
\begin{equation}\label{E6.31}
n^2+\varrho_0\leq\mu_n\leq n^2+\varrho_1.
\end{equation}
In fact, Lemma \ref{lemma6.2} shows $\mu_n>\varrho_0$ for $n\in\mathbf{N}^{+}$. On the one hand, applying the first equality in \eqref{E6.10} and the Pr\"{u}fer transformation
\begin{equation*}
\varphi_n=r\sin\theta, \quad \varphi'_n=(\sqrt{\mu_n-\varrho_0})r\cos \theta
\end{equation*}
with $r(\xi)>0$, we derive for $n\in\mathbf{N}^+$
\begin{align}\label{E6.12}
\frac{\mathrm{d}\theta}{\mathrm{d}\xi}&=(\sqrt{\mu_n-\varrho_0})\cos^2\theta+\frac{\mu_n-\varrho(\xi)}{\sqrt{\mu_n-\varrho_0}}\sin^2\theta \nonumber\\
&=(\sqrt{\mu_n-\varrho_0})\cos^2\theta+\frac{\mu_n-\varrho_0+\varrho_0-\varrho(\xi)}{\sqrt{\mu_n-\varrho_0}}\sin^2\theta \nonumber\\
&=\sqrt{\mu_n-\varrho_0}+\frac{\varrho_0-\varrho(\xi)}{\sqrt{\mu_n-\varrho_0}}\sin^2\theta \nonumber\\
&\leq\sqrt{\mu_n-\varrho_0}.
\end{align}
Moreover we may choose $\theta(0)=\frac{\pi}{2}$, $\theta(\pi)=(n+\frac{1}{2})\pi$ thanks to that $\varphi_n$ has exactly $n$ zeros in $(0,\pi)$ and $\varphi'_n(0)=\varphi'_n(\pi)=0$. Integrating \eqref{E6.12} over $[0,\pi]$, we get for $n\in\mathbf{N}^{+}$
\begin{equation*}
n\pi\leq(\sqrt{\mu_n-\varrho_0})\pi\Longrightarrow\mu_n\geq n^2+\varrho_0.
\end{equation*}
And on the other hand, from the first term of \eqref{E6.10} and the Pr\"{u}fer transformation
\begin{equation*}
\varphi_n=r\sin\theta,~\varphi'_n=\sqrt{\mu_n}r\cos\theta
\end{equation*}
with $r(\xi)>0$, it yields that
\begin{equation}\label{E6.30}
\frac{\mathrm{d}\theta}{\mathrm{d}\xi}=\sqrt{\mu_n}-\frac{\varrho(\xi)}{\sqrt{\mu_n}}
\sin^2\theta\geq\sqrt{\mu_n}-\frac{\varrho(\xi)}{\sqrt{\mu_n}},\quad\forall n\in\mathbf{N}^{+}.
\end{equation}
Denoting $X=\mu_n$, $B=n$, $C=\frac{1}{\pi}\int^\pi_0\varrho(\xi)~\mathrm{d}\xi$, by integrating \eqref{E6.30} over $[0,\pi]$, the quadratic formula and the elementary inequality $\sqrt{1+x}\leq1+\frac{x}{2}$, we obtain
\begin{equation*}
\begin{aligned}
(\sqrt{X})^2-B\sqrt{X}-C\leq0
&\Longrightarrow\sqrt{X}\leq(B+\sqrt{B^2+4C})/2\\
&\Longrightarrow X\leq(2B^2+4C+2B\sqrt{B^2+4C})/4\\
&\Longrightarrow X\leq(2B^2+4C+2B^2\sqrt{1+4{C}/{B^2}})/4\\
&\Longrightarrow X\leq(2B^2+4C+2B^2(1+2{C}/{B^2}))/4=B^2+2C\\
&\Longrightarrow\mu_n\leq n^2+\varrho_1.
\end{aligned}
\end{equation*}
 The above analysis verifies the claim \eqref{E6.31},
which implies
\begin{equation*}
\sqrt{n^2+\varrho_0}-n\leq\sqrt{\mu_n}-n\leq\sqrt{n^2+\varrho_1}-n.
\end{equation*}
Let $g_1(\xi):=\sqrt{n^2+\xi}-n$. By Taylor expansion at $\xi=0$, we have
\begin{equation*}
g_1(\xi)=\frac{\xi}{2n}+O\left(\frac{1}{n^3}\right).
\end{equation*}
In view of the inequality: $g_1(\varrho_0)\leq\sqrt{\mu_n}-n\leq g_1(\varrho_1)$, there exists a constant $c_0$ with $\varrho_0\leq c_0\leq\varrho_1$ such that
\begin{equation*}
\mu_n=n^2+c_0+O\left(\frac{1}{n^2}\right).
\end{equation*}
This completes the proof.
\end{proof}

\subsection{Dirichlet-Newmann boundary value problem}\label{sec:4.3}

In Case 2, the Sturm-Liouville problem \eqref{E6.7}-\eqref{E6.8} becomes
\begin{equation}\label{E6.14}
\left\{
\begin{aligned}
&\varphi''_{n}(\xi)+(\mu_n-\varrho(\xi))\varphi_n(\xi)=0,\\
&\varphi_n(0)=\varphi'_n(\pi)=0.
\end{aligned}
\right.
\end{equation}
Notice that Case 3 is reduced to Case 2 if the transform $\bar{x}=\pi-x$ is made. Consequently, we just consider Case 2 in the section.
\begin{lemm}\label{lemma6.9}
Let $\mu_0<\mu_1<\cdots$ and $\varphi_0,\varphi_1,\cdots$ denote the eigenvalues and orthonormal eigenfunctions of the Sturm-Liouville problem \eqref{E6.14} respectively. Then, for $n\in\mathbf{N}$, the following asymptotic formulae hold:
\begin{align}\label{E6.13}
\mu_n=\left(n+{1}/{2}\right)^2+c_1+O\left(\frac{1}{(n+{1}/{2})^2}\right)
\end{align}
with $\varrho_0\leq c_1\leq\varrho_2$, where $\varrho_0,\varrho_1$ are given by \eqref{E6.16} and  \eqref{rh1} respectively.
\end{lemm}
\begin{proof}
Applying the similar technique as in the proof of Lemma \ref{lemma6.3},  we  prove
\begin{equation}\label{E6.37}
(n+1/2)^2+\varrho_0\leq\mu_n\leq (n+1/2)^2+\varrho_1,\quad\forall n\in \mathbf{N}.
\end{equation}
It follows from
Lemma \ref{lemma6.2} that $\mu_n>\varrho_0$ for $n\in\mathbf{N}$.
Denote by $a_2,a_4,\cdots,a_{2n}$ with $0<a_2<a_4<\cdots<a_{2n}<\pi$ the $n$ zeros of $\varphi_n$ in $(0,\pi)$. Letting $a_0=0$, $a_{2n+1}=\pi$, we take $a_{2i-1}$ satisfying $\phi'_n(a_{2i-1})=0$ for $a_{2i-1}\in(a_{2i-2},a_{2i}),i=1,2,\cdots,n$. This infers that for $j=1,2,\cdots,2n+1$
\begin{equation*}
\int^{a_j}_{a_{j-1}}\varphi^2_n(x)\mathrm{d}x\leq 4\pi^{-2}(a_j-a_{j-1})^2\int^{a_j}_{a_{j-1}}(\varphi'_n(x))^2\mathrm{d}x
\end{equation*}
by Sobolev inequality. On the other hand, integrating by parts yields
\begin{equation*}
\int^{a_j}_{a_{j-1}}\varphi^2_n(x)\mathrm{d}x=-\int^{a_j}_{a_{j-1}}\varphi''_n(x)\varphi_{n}(x)\mathrm{d}x, \quad j=1,2,\cdots,2n+1.
\end{equation*}
From multiplying both sides of the equation in system \eqref{E6.7} by $\varphi_n(\xi)$ and integrating over $[a_{j-1},a_{j}]$, it derives that
\begin{equation*}
\int^{a_j}_{a_{j-1}}(\mu_n-\varrho(x))\varphi^2_n(x)\mathrm{d}x=\int^{a_j}_{a_{j-1}}(\varphi'_n(x))^2\mathrm{d}x.
\end{equation*}
Then
\begin{equation*}
(\mu_n-\varrho_0)\int^{a_j}_{a_{j-1}}\varphi^2_n(x)\mathrm{d}x=\int^{a_j}_{a_{j-1}}((\varphi'_n(x))^2+(\varrho(x)-\varrho_0)\varphi^2_n)~\mathrm{d}x.
\end{equation*}
As a result
\begin{equation*}
\int^{a_j}_{a_{j-1}}(\varphi'_n(x))\mathrm{d}x\leq(\mu_n-\varrho_0)\int^{a_j}_{a_{j-1}}\varphi^2_n(x)\mathrm{d}x
\leq 4\pi^{-2}(\mu_n-\varrho_0)(a_j-a_{j-1})^2\int^{a_j}_{a_{j-1}}(\varphi'_n(x))^2\mathrm{d}x.
\end{equation*}
This leads to
\begin{equation*}
\sum^{2n+1}_{j=1}\frac{1}{(a_{j}-a_{j-1})^2}\leq4\pi^{-2}(2n+1)(\mu_n-\varrho_0).
\end{equation*}
Moreover $\min\left\{\sum^{2n+1}_{j=1}1/(x^2_j)|x_j>0,\sum^{2n+1}_{j=1}x_i=\pi\right\}=(2n+1)(\frac{2n+1}{\pi})^2$ for $x_1=x_2=\cdots=x_{2n+1}=\frac{\pi}{2n+1}$. The fact shows
\begin{equation*}
\mu_n\geq(n+1/2)^2+\varrho_0.
\end{equation*}
Let us check the upper bound in \eqref{E6.37}. We introduce the Pr\"{u}fer transformation
\begin{equation*}
\varphi_n=r\sin\theta,~\varphi'_n=\sqrt{\mu_n}r\cos\theta
\end{equation*}
with $r(\xi)>0$. The calculation similar to the one used in Lemma \ref{lemma6.3} shows
\begin{equation*}
\frac{\mathrm{d}\theta}{\mathrm{d}\xi}=\sqrt{\mu_n}-\frac{\varrho(\xi)}{\sqrt{\mu_n}}
\sin^2\theta\geq\sqrt{\mu_n}-\frac{\varrho(\xi)}{\sqrt{\mu_n}}\quad \forall n\in\mathbf{N}.
\end{equation*}
Since $\varphi_n$ has exactly $n$ zeros in $(0,\pi)$ and $\varphi_n(0)=\varphi'_n(\pi)=0$, we may choose $\theta(0)=0$, which then gives $\theta(\pi)=(n+\frac{1}{2})\pi$. Integrating the above inequality over $[0,\pi]$, we have
\begin{equation*}
(\sqrt{X})^2-B\sqrt{X}-C\leq0,
\end{equation*}
where $X=\mu_n$, $B=n+\frac{1}{2}$, $C=\frac{1}{\pi}\int^\pi_0\varrho(\xi)~\mathrm{d}\xi$. Furthermore the quadratic formula together with the elementary inequality $\sqrt{1+x}\leq1+\frac{x}{2}$ may give rise to
\begin{equation*}
\begin{aligned}
\sqrt{X}\leq(B+\sqrt{B^2+4C})/2
&\Longrightarrow X\leq(2B^2+4C+2B\sqrt{B^2+4C})/4\\
&\Longrightarrow X\leq(2B^2+4C+2B^2(1+2{C}/{B^2}))/4=B^2+2C\\
&\Longrightarrow\mu_n\leq n^2+\varrho_1.
\end{aligned}
\end{equation*}
Therefore the inequality in \eqref{E6.37} is established. This reads
\begin{equation*}
\sqrt{(n+{1}/{2})^2+\varrho_0}-(n+{1}/{2})\leq\sqrt{\mu_n}
-(n+{1}/{2})\leq\sqrt{(n+{1}/{2})^2+\varrho_1}-(n+{1}/{2}).
\end{equation*}
Set $g_2(\xi):=\sqrt{(n+{1}/{2})^2+\xi}-(n+{1}/{2})$. By Taylor expansion at $\xi=0$, we have
\begin{equation*}
g_2(\xi)=\frac{\xi}{2(n+{1}/{2})}+O\left(\frac{1}{(n+{1}/{2})^3}\right).
\end{equation*}
With the help of the fact $g_2(\varrho_0)\leq\sqrt{\mu_n}-(n+{1}/{2})\leq g_2(\varrho_1)$, we have
\begin{equation*}
\mu_n=(n+{1}/{2})^2+c_1+O\left(\frac{1}{(n+{1}/{2})^2}\right)
\end{equation*}
for some constant $c_1$ with $\varrho_0\leq c_1\leq\varrho_1$. This completes the proof.
\end{proof}

\subsection{General boundary value problem}\label{sec:4.2}

In Case 4, we write the Sturm-Liouville problem (\ref{E6.7})-(\ref{E6.8}) as
\begin{equation}\label{E6.15}
\left\{
\begin{aligned}
&\varphi''_{n}(\xi)+(\mu_n-\varrho(\xi))\varphi_n(\xi)=0,\\
&a_1\varphi_n(0)-b_1\varphi'_n(0)=0,\\
&a_2\varphi(\pi)+b_2\varphi'_n(\pi)=0.
\end{aligned}
\right.
\end{equation}
\begin{lemm}\label{lemma6.8}
Denote by $\mu_0<\mu_1<\cdots$ and $\varphi_0,\varphi_1,\cdots$ the eigenvalues and orthonormal eigenfunctions of the Sturm-Liouville problem \eqref{E6.15} respectively
. Then there exists an integer  $\widehat{N}>0$ such that, for $n\geq \widehat{N}$, the following asymptotic formulae hold:
\begin{equation}\label{E6.17}
\mu_n=n^2+c_2+O\left(\frac{1}{n^2}\right)
\end{equation}
with $\varrho_0\leq c_2\leq\varrho_2$, where $\varrho_0$ is defined by \eqref{E6.16} and
$\varrho_2:=\frac{2}{\pi}(\frac{a_1}{b_1}+\frac{a_2}{b_2}+1+\int^{\pi}_0\varrho(\xi)\mathrm{d}\xi)$.
\end{lemm}
\begin{proof}
We should adopt the similar technique as in the proof of Lemma \ref{lemma6.3}. Let us assert that there exists an $\widehat{N}>0$ such that, for $n\geq\widehat{N}$, the following holds:
\begin{equation}\label{E6.33}
n^2+\varrho_0\leq\mu_n\leq n^2+\varrho_2.
\end{equation}
Since $a^2_1+a^2_2>0$, by Lemma \ref{lemma6.2}, we arrive at $\mu_n>\varrho_0$ for $n\in\mathbf{N}$. First, we introduce the Pr\"{u}fer transformation for $r(\xi)>0$
\begin{equation*}
\varphi_n=r\sin\theta,~\varphi'_n=(\sqrt{\mu_n-\varrho_0})r\cos \theta.
\end{equation*}
The calculation similar to the one used in Lemma \ref{lemma6.3} shows for $n\in\mathbf{N}$
\begin{equation}\label{E6.18}
\frac{\mathrm{d}\theta}{\mathrm{d}\xi}
\leq\sqrt{\mu_n-\varrho_0}.
\end{equation}
Denote by $\tau_1<\tau_2<\cdots<\tau_n$ the $n$ zeros of $\varphi_n$ in $(0,\pi)$. Let $\tau_0=0$, $\tau_{n+1}=\pi$ and $\theta_i=\theta(\tau_i),i=0,\cdots,n+1$.
Hence we may choose $\theta_i=i\pi,i=1,\cdots,n$, which  gives $\theta_0=\arctan(\frac{b_1}{a_1}\sqrt{\mu_n-\varrho_0})$ and $\theta_{n+1}=(n+1)\pi-\arctan(\frac{b_2}{a_2}\sqrt{\mu_n-\varrho_0})$.
It follows from $\mu_n>\varrho_0$ for $n\in\mathbf{N}$ and $a_i>0,b_i>0,i=1,2$ that $\theta_0\in(0,\frac{\pi}{2})$, $\theta_{n+1}\in((n+\frac{1}{2})\pi,(n+1)\pi)$.
 As a consequence $\theta_{n+1}-\theta_0\geq n\pi$.  Integrating \eqref{E6.18} over $[0,\pi]$ yields for $n\in\mathbf{N}$
\begin{equation*}
n\pi\leq\theta_{n+1}-\theta_0\leq(\sqrt{\mu_n-\varrho_0})\pi\Longrightarrow\mu_n\geq n^2+\varrho_0 .
\end{equation*}
In order to get the upper bounded  in \eqref{E6.33}, we take  the Pr\"{u}fer transformation
\begin{equation*}
\varphi_n=r\sin\theta,~\varphi'_n=\sqrt{\mu_n}r\cos \theta
\end{equation*}
with $r(\xi)>0$. Proceeding as in the proof of Lemma \ref{lemma6.3}, we have
\begin{equation}\label{E6.20}
\frac{\mathrm{d}\theta}{\mathrm{d}\xi}=\sqrt{\mu_n}-\frac{\varrho(\xi)}{\sqrt{\mu_n}}\sin^2\theta\geq\sqrt{\mu_n}-\frac{\varrho(\xi)}{\sqrt{\mu_n}}.
\end{equation}
Let $\tau_1<\tau_2<\cdots<\tau_n$ denote the $n$ zeros of $\varphi_n$ in $(0,\pi)$. Setting $\tau_0=0$, $\tau_{n+1}=\pi$ and $\theta_i=\theta(\tau_i),i=0,\cdots,n+1$,
 we take $\theta_i=i\pi,i=1,\cdots,n$. Then $\theta_0=\arctan(\frac{b_1}{a_1}\sqrt{\mu_n})$, $\theta_{n+1}=(n+1)\pi-\arctan(\frac{b_2}{a_2}\sqrt{\mu_n})$.
  With the help of $\mu_n>0$ and $a_i>0,b_i>0,i=1,2$, we derive $\theta_0\in(0,\frac{\pi}{2})$ and $\theta_{n+1}\in((n+\frac{1}{2})\pi,(n+1)\pi)$.
 Moreover the fact $\mu_n\rightarrow+\infty$ as $n\rightarrow+\infty$ implies that $\theta_0\rightarrow\frac{\pi}{2}$ as $n\rightarrow+\infty$,
 and that $\theta_{n+1}\rightarrow(n+\frac{1}{2})\pi$ as $n\rightarrow+\infty$.
Integrating \eqref{E6.20} over $[0,\pi]$ gives
\begin{equation*}
\begin{aligned}
&\theta_{n+1}-\theta_0\geq\sqrt{\mu_n}\pi-\frac{1}{\sqrt{\mu_n}}\int^{\pi}_{0}\varrho(\xi)~\mathrm{d}\xi\\
\Longrightarrow&(n+1)\pi-\left(\arctan\left(\frac{b_1}{a_1}\sqrt{\mu_n}\right)+\arctan\left(\frac{b_2}{a_2}\sqrt{\mu_n}\right)\right)
\geq\sqrt{\mu_n}\pi-\frac{1}{\sqrt{\mu_n}}\int^{\pi}_{0}\varrho(\xi)~\mathrm{d}\xi\\
\Longrightarrow&\mu_n-n\sqrt{\mu_n}+\frac{\sqrt{\mu_n}}{\pi}\left(\arctan\left(\frac{b_1}{a_1}\sqrt{\mu_n}\right)-\frac{\pi}{2}\right)\\
&+\frac{\sqrt{\mu_n}}{\pi}\left(\arctan\left(\frac{b_2}{a_2}\sqrt{\mu_n}\right)-\frac{\pi}{2}\right)
-\frac{1}{\pi}\int^{\pi}_{0}\varrho(\xi)~\mathrm{d}\xi\leq0.
\end{aligned}
\end{equation*}
Furthermore for $x>0,a>0,b>0$, it is obvious that
\begin{equation*}
\begin{aligned}
\lim_{x\rightarrow+\infty}x\left(\arctan\left(\frac{b}{a}x\right)-\frac{\pi}{2}\right)
&=\lim_{x\rightarrow+\infty}\frac{\arctan\left(\frac{b}{a}x\right)-\frac{\pi}{2}}{1/x}{=}\lim_{x\rightarrow+\infty}\frac{\frac{1}{1+((b/a)x)^2\frac{b}{a}}}{-1/x^2}=-\frac{a}{b}.
\end{aligned}
\end{equation*}
Since $\sqrt{\mu_n}\rightarrow+\infty$ as $n\rightarrow+\infty$, we get
\begin{equation}\label{E6.34}
\begin{aligned}
&\lim_{n\rightarrow+\infty}\sqrt{\mu_n}\left(\arctan\left(\frac{b_1}{a_1}\sqrt{\mu_n}\right)-\frac{\pi}{2}\right)=-\frac{a_1}{b_1},\\
&\lim_{n\rightarrow+\infty}\sqrt{\mu_n}\left(\arctan\left(\frac{b_2}{a_2}\sqrt{\mu_n}\right)-\frac{\pi}{2}\right)=-\frac{a_2}{b_2}.
\end{aligned}
\end{equation}
It follows from  \eqref{E6.34} that  there exists an integer  $\widehat{N}>0$ such that, for $n\geq\widehat{N}$, the following inequalities hold:
\begin{equation*}
\sqrt{\mu_n}\left(\arctan\left(\frac{b_1}{a_1}\sqrt{\mu_n}\right)-\frac{\pi}{2}\right)\geq-\frac{a_1}{b_1}-\frac{1}{2},\quad
\sqrt{\mu_n}\left(\arctan\left(\frac{b_2}{a_2}\sqrt{\mu_n}\right)-\frac{\pi}{2}\right)\geq-\frac{a_2}{b_2}-\frac{1}{2}.
\end{equation*}
Therefore for $n\geq \widehat{N}$
\begin{equation*}
\mu_n-n\sqrt{\mu_n}-\frac{1}{\pi}\left(\frac{a_1}{b_1}+\frac{a_2}{b_2}+1+\int^\pi_0\varrho(\xi)~\mathrm{d}\xi\right)\leq0.
\end{equation*}
Setting $X=\mu_n$, $B=n$, $C=\frac{1}{\pi}(\frac{a_1}{b_1}+\frac{a_2}{b_2}+1+\int^\pi_0\varrho(\xi))~\mathrm{d}\xi$, the elementary inequality $\sqrt{1+x}\leq1+1/x$ verifies for $n\geq \widehat{N}$
\begin{equation*}
\begin{aligned}
(\sqrt{X})^2-B\sqrt{X}-C\leq0
&\Longrightarrow X\leq(2B^2+4C+2B^2(1+2{C}/{B^2}))/4=B^2+2C\\
&\Longrightarrow\mu_n\leq n^2+\varrho_2.
\end{aligned}
\end{equation*}
The above discussion carries out the assertion. Thus, by \eqref{E6.33}, we obtain that  for $n\geq \widehat{N}$
\begin{equation*}
\sqrt{n^2+\varrho_0}-n\leq\sqrt{\mu_n}-n\leq\sqrt{n^2+\varrho_2}-n.
\end{equation*}
Denote $g_3(\xi):=\sqrt{n^2+\xi}-n$. It is clear to see that
$g_3(\xi)=\frac{\xi}{2n}+O\left(\frac{1}{n^3}\right)$ by Taylor expansion at $\xi=0$.
Since $g_3(\varrho_0)\leq\sqrt{\mu_n}-n\leq g_3(\varrho_2)$, there exists a constant $c_2$ with $\varrho_0\leq c_2\leq\varrho_2$ such that
\begin{equation*}
\mu_n=n^2+c_2+O\left(\frac{1}{n^2}\right),\quad\forall n\geq \widehat{N}.
\end{equation*}
We complete the proof of the lemma.
\end{proof}

\subsection{Summary}

Summarize what we have  obtained  in Lemmas \ref{lemma6.3}-\ref{lemma6.8} as the following lemma. 

\begin{lemm}\label{lemma6.1}
Denote by $\lambda_{n}(\epsilon,w),\psi_{n}(\epsilon,w),n\in\mathbf{N}$ the eigenvalues and the eigenfunctions of the sturm-Liouville problem \eqref{E6.1} respectively.
Let $c$ (see \eqref{E6.23}) be fixed constant and Hypothesis \ref{hy2} hold. Then
\[
0<\lambda_{0}(\epsilon,w)<\lambda_{1}(\epsilon,w)<\cdots<\lambda_{n}(\epsilon,w) < \cdots \]
with $\lambda_{n}(\epsilon,w)\rightarrow+\infty$ as $n\rightarrow+\infty$.
Moreover each eigenvalue  $\lambda_{n}(\epsilon,w)$ is simple and, for all $\epsilon\in[\epsilon_1,\epsilon_2]$, $w\in\{W\cap H^s:\|w\|_{s}\leq r\}$, it has  the following asymptotic formulae:

$\mathrm{(i)}$ If $\alpha_1=-\frac{\beta_1}{4}\frac{(p\rho)_{x}(0)}{(p\rho)(0)},\beta_1>0, \alpha_2=\frac{\beta_2}{4}\frac{(p\rho)_{x}(0)}{(p\rho)(0)},\beta_2>0$, then
for $n\in\mathbf{N}^{+}$,
\begin{equation}\label{E6.24}
\lambda_n(\epsilon,w)=\frac{n^2}{c^2}+\mathfrak{c}_0(\epsilon,w)+O\left(\frac{1}{n^2}\right)
\end{equation}
with ${\eta_0}\leq \mathfrak{c}_0(\epsilon,w)\leq\frac{2}{c\pi}\int^\pi_0{\sqrt{\rho/p}}~\eta(\epsilon,w)~\mathrm{d}x$, where
\begin{align*}
&\eta(x,\epsilon,w)=Q(x)+
\frac{1}{{\rho(x)}}\left (m(x)-\epsilon\Pi_{V}f'(t,x,v(\epsilon,w(t,x))+w(t,x))\right),\\
&\eta_0=\min_{\stackrel{x\in[0,\pi],\epsilon\in[\epsilon_1,\epsilon_2]}{w\in \{W\cap H^s:\|w\|_{s}\leq r\}}}{\eta(x,\epsilon,w)},
\end{align*}
and $Q(x)$ is given by  \eqref{E7.2};

$\mathrm{(ii)}$ If either $\alpha_1>0,\beta_1=0,\alpha_2=\frac{\beta_2}{4}\frac{(p\rho)_{x}(\pi)}{(p\rho)(\pi)},\beta_2>0$ or $\alpha_1=-\frac{\beta_1}{4}\frac{(p\rho)_{x}(0)}{(p\rho)(0)},\beta_1>0,\alpha_2>0,\beta_2=0$,
then for $n\in\mathbf{N}$,
\begin{equation}\label{E6.38}
\lambda_n(\epsilon,w)=\frac{(n+1/2)^2}{c^2}+\mathfrak{c}_1(\epsilon,w)+O\left(\frac{1}{(n+1/2)^2}\right)
\end{equation}
with ${\eta_0}\leq \mathfrak{c}_1(\epsilon,w)\leq\frac{2}{c\pi}\int^\pi_0{\sqrt{\rho/p}}~\eta(\epsilon,w)~\mathrm{d}x$;

$\mathrm{(iii)}$ If $\alpha_1>-\frac{\beta_1}{4}\frac{(p\rho)_{x}(0)}{(p\rho)(0)}, \beta_1>0, \alpha_2>\frac{\beta_2}{4}\frac{(p\rho)_{x}(0)}{(p\rho)(0)},\beta_2>0$, then there exists $\widehat{N}>0$ (see Lemma \ref{lemma6.8}) such that for $n\geq \widehat{N}$
\begin{equation}\label{E6.25}
\lambda_n(\epsilon,w)=\frac{n^2}{c^2}+\mathfrak{c}_2(\epsilon,w)+O\left(\frac{1}{n^2}\right)
\end{equation}
with
\[{\eta_0}\leq \mathfrak{c}_2(\epsilon,w)\leq\frac{2}{c^2\pi}(\frac{a_1}{b_1}+\frac{a_2}{b_2}+1)+
\frac{2}{c\pi}\int^\pi_0\frac{\sqrt{\rho}}{\sqrt{p}}{\eta(\epsilon,w)}~\mathrm{d}x.\]
And $\psi_{n}(\epsilon,w),n\in\mathbf{N}$ form an orthogonal basis of $L^2(0,\pi)$ with the scalar product $(y,z)_{L^2_{\rho}}:=c^{-1}\int_{0}^{\pi}yz\rho~\mathrm{d}x$.
 In addition we define an equivalent scalar product on $\mathcal{H}^1_g$
\begin{equation*}
(y,z)_{\epsilon,w}:=\frac{1}{c}\int_{0}^{\pi}py'z'+(m-\epsilon\Pi_{V}f'(v(\epsilon,w)+w))yz~\mathrm{d}x,
\end{equation*}
which establishes that for all $y\in \mathcal{H}^1_g$
\begin{equation}\label{E6.3}
L_1\|y\|_{H^1}\leq\|y\|_{\epsilon,w}\leq L_2\|y\|_{H^1}
\end{equation}
for some constants $L_1$, $L_2$. The eigenfunctions $\psi_{n}(\epsilon,w),n\geq0$ are also an orthogonal basis of $\mathcal{H}^1_g$
 with respect to the scalar product $(\cdot,\cdot)_{\epsilon,w}$. For $y=\sum_{n\geq0}\hat{y}_n\psi_{n}(\epsilon,w)$, it has
\begin{equation}\label{E6.4}
\|y\|^2_{L^2_{\rho}}=\sum\limits_{n\geq0}\hat{y}^2_n,\quad
\|y\|^2_{\epsilon,w}=\sum\limits_{n\geq0}\lambda_n(\epsilon,w)\hat{y}^2_n.
\end{equation}
\end{lemm}
\begin{proof}
Formula \eqref{E6.6} gives $\mu_n(\vartheta)=c^2\lambda_n(d)$. Then it follows from Hypothesis \ref{hy2} and Lemma \ref{lemma6.2} that $\lambda_0(\epsilon,w)>0$ for all $\epsilon\in[\epsilon_1,\epsilon_2]$ and $w\in\{W\cap H^s:\|w\|_{s}\leq r\}$. By dividing by $c^2$, \eqref{E6.11}, \eqref{E6.13}, \eqref{E6.17} and the inverse Liouville substitution of \eqref{E6.26}, eigenvalues $\lambda_n(d),n\in\mathbf{N}$ of \eqref{E6.1} have the asymptotic formulae
\eqref{E6.24}-\eqref{E6.25}. Moreover the eigenfunctions of \eqref{E6.10} or \eqref{E6.14} or \eqref{E6.15} form an orthonormal basis for $L^2$.  The Liouville substitution \eqref{E6.35} yields
\begin{equation*}
\varphi_{n}=\psi_n(\epsilon,w)\mathfrak{s},
\end{equation*}
where $\mathfrak{s}(x)=(p(x)\rho(x))^{\frac{1}{4}}$. It follows from the inverse Liouville substitution of \eqref{E6.26} that
\begin{equation*}
\begin{aligned}
\int^\pi_0\varphi_{n}(\xi)\varphi_{n}(\xi)~\mathrm{d}\xi&=\int^\pi_0\psi_n(\epsilon,w)(x)\mathfrak{s}(x)
\psi_n(\epsilon,w)(x)\mathfrak{s}(x)\frac{1}{c}\left(\frac{\rho(x)}{p(x)}\right)^{1/2}\mathrm{d}x\\
&=\frac{1}{c}\int^\pi_0\psi_n(\epsilon,w)(x)\psi_n(\epsilon,w)(x)\rho(x)~\mathrm{d}x.
\end{aligned}
\end{equation*}
 Hence the eigenfunctions $\psi_{n}(\epsilon,w),n\in\mathbf{N}$ of \eqref{E6.1}
 form an orthogonal basis for $L^2$ with respect to the scalar product $(\cdot,\cdot)_{L^2_\rho}$. A simple calculation  gives \eqref{E6.3}.
 Furthermore, it is obvious that
\begin{equation*}
-(p\psi'_n(\epsilon,w))'+(m-\epsilon\Pi_{V}f'(t,x,v(\epsilon,w)+w))\psi_n(\epsilon,w)=\lambda_n(\epsilon,w)\psi_n(\epsilon,w).
\end{equation*}
 Multiplying above equality  by $\psi_{n'}(\epsilon,w)$ and integrating by parts yield
\begin{equation*}
(\psi_n,\psi_{n'})_{\epsilon,w}=\delta_{n,n'}\lambda_n(\epsilon,w),
\end{equation*}
which implies \eqref{E6.4}.
\end{proof}

\section{Solution of the $(P)$-equation}\label{sec:6}

In this section, our aim now is to solve the following $(P)$-equation
\begin{equation}\label{E4.1}
L_{\omega}w=\epsilon \Pi_{W}\mathcal{F}(\epsilon,w),
\end{equation}
where $\mathcal{F}(\epsilon,w):=F(v(\epsilon,w),w)$.
\begin{theo}\label{Th2}
There exists $\tilde{w}\in C^1(A_0;W\cap H^s)$ with
\begin{equation*}
\|\tilde{w}\|_{s}\leq\frac{K\epsilon}{\gamma\omega},
\quad\|\partial_{\omega}\tilde{w}\|_s\leq\frac{K(\gamma_0)\epsilon}{\gamma^5\omega},
\quad\|\partial_{\epsilon}\tilde{w}\|_s\leq\frac{K(\gamma_0)}{\gamma^5\omega},
\end{equation*}
and the large Cantor set $B_\gamma$, where $\gamma_0,B_\gamma$ are defined by \eqref{E4.45} (or see \eqref{E4.46}) and \eqref{E4.31} respectively, such that,
for $(\epsilon,\omega)\in B_\gamma$,  $\tilde{w}(\epsilon,\omega)$ solves equation \eqref{E4.1}.
\end{theo}
Let us denote by the symbol $[~\cdot~]$ the integer part. Define
\begin{align}\label{E4.2}
W_{N_n}:=\Bigg\{w\in W~|~w=\sum_{|l|\leq N_n}w_{l}(x)e^{\mathrm{i}lt}\Bigg\},\quad W_{N_n}^{\bot}:=\Bigg\{w\in W~|~w=\sum_{|l|> N_n}w_{l}(x)e^{\mathrm{i}lt}\Bigg\},
\end{align}
where $N_n:=[e^{\mathfrak{d}\chi^{n}}]$ with $\mathfrak{d}=\ln N_0$ and $1<\chi\leq2$. Evidently there is a direct sum decomposition
\begin{equation*}
W=W_{N_n}\oplus W_{N_n}^{\bot}.
\end{equation*}
Furthermore $P_{N_n},P^{\bot}_{N_n}$ denote the orthogonal projectors onto $W_{N_n}$ and $W_{N_n}^{\bot}$ respectively, namely
\begin{equation*}
P_{N_n}:W\rightarrow W_{N_n},~~P^{\bot}_{N_n}:W\rightarrow W_{N_n}^{\bot}.
\end{equation*}
The existence result of the solutions of  ($P$)-equation is based on a differentiable Nash-Moser iteration scheme. Denote by $\Delta^{\gamma,\tau}_{N}(w)$ for $\tau\in(1,2),\gamma\in(0,1)$ the set of $(\epsilon,\omega)$ satisfying the Melnikov non-resonance conditions, i.e.
\begin{equation}\label{E4.3}
\begin{aligned}
\Delta^{\gamma,\tau}_{N}(w):=\Bigg\{&(\epsilon,\omega)\in(\epsilon_1,\epsilon_2)\times(\omega_0,+\infty): \left|\omega l-\sqrt{\lambda_{j}(\epsilon,w)}\right|>\frac{\gamma}{l^{\tau}},\\
 &\left|\omega l-\frac{j}{c}\right|>\frac{\gamma}{l^{\tau}},~\forall l=1,2,\cdots,N,j\geq0\Bigg\},
\end{aligned}
\end{equation}
where $\epsilon_i,i=1,2$ are given by Lemma \ref{lemma3.1}, $\omega_{0}$ is given by \eqref{E6.39} and $c$ is defined by \eqref{E6.23}.
\subsection{Some properties on $\mathcal{F}(\epsilon,w)$}\label{sec:6.0}
To guarantee the convergence of the iteration scheme, we need the following properties $(P1)$-$(P5)$, see also \cite{berti2008cantor,berti2010sobolev}.

$\bullet~(P1)$(Regularity) $\mathcal{F}\in C^2(H^s;H^s)$ and $\mathcal{F},\mathrm{D}_{w}\mathcal{F},\mathrm{D}^2_{w}\mathcal{F}$ are bounded on $\{\|w\|_{s}\leq 1\}$.

$\bullet~(P2)$(Tame) $\forall s\leq s'\leq k-3$, $\forall w\in H^{s'}$ and $\|w\|_{s}\leq1$,
\begin{align*}
\|\mathcal{F}(\epsilon,w)\|_{{s'}}&\leq C(s')(1+\|w\|_{{s'}}),\\
\|\mathrm{D}_{w}\mathcal{F}(\epsilon,w)h\|_{{s'}}&{\leq}C(s')(\|w\|_{{s'}}\|h\|_{s}+\|h\|_{{s'}}),\\
\|\mathrm{D}^2_{w}\mathcal{F}(\epsilon,w)[h,h]\|_{{s'}}&{\leq}C(s')(\|w\|_{{s'}}\|h\|^2_{s}+\|h\|_{s}\|h\|_{{s'}}).
\end{align*}

$\bullet~(P3)$(Taylor Tame) $\forall s\leq s'\leq k-3$, $\forall w\in H^{s'}$, $\forall h\in H^{s'}$ and $\|w\|_{s}\leq1$,
\begin{equation*}
\|\mathcal{F}(\epsilon,w+h)-\mathcal{F}(\epsilon,w)-\mathrm{D}_{w}\mathcal{F}(\epsilon,w)[h]\|_{{s'}}\leq C(s')(\|w\|_{{s'}}\|h\|^2_{s}+\|h\|_{s}\|h\|_{{s'}}).
\end{equation*}
In particular, for $s'=s$,
\begin{equation*}
\|\mathcal{F}(\epsilon,w+h)-\mathcal{F}(\epsilon,w)-\mathrm{D}_{w}\mathcal{F}(\epsilon,w)[h]\|_{s}\leq C\|h\|^2_{s}.
\end{equation*}

$\bullet~(P4)$(Smoothing) For all $N\in\mathbf{N}\backslash \{0\}$, we have
\begin{align*}
&\|P_{N}u\|_{{s+r}}\leq N^r\|u\|_{s},\quad\forall u\in H^s,\\
&\|P^{\bot}_{N}u\|_{s}\leq N^{-r}\|u\|_{{s+r}},\quad\forall u\in H^{s+r}.
\end{align*}

$\bullet~(P5)$(Invertibility of $\mathcal{L}_{N}$) Let $(\epsilon,\omega)\in\Delta^{\gamma,\tau}_{N}(w)$ for fixed $\tau,\gamma$ with $1<\tau<2,0<\gamma<1$. The linear operator is defined as
\begin{align}\label{E4.32}
\mathcal{L}_{N}(\epsilon,\omega,w)[h]:=-L_{\omega}h+\epsilon P_{N}\Pi_{W}\mathrm{D}_{w}\mathcal{F}(\epsilon,w)[h],\quad\forall h\in W^{(N)},
\end{align}
where $L_{\omega}h:=\omega^2\rho(x)h_{tt}-(p(x)h_{x})_{x}$, $h=\sum_{1\leq|l|\leq N}h_l(x)e^{\mathrm{i}lt}$.  There exist $K, K(s')>0$ such that if
\begin{equation*}
\|w\|_{{s+\sigma}}\leq1\quad\text{with}~\sigma=\frac{\tau(\tau-1)}{2-\tau},
\end{equation*}
 then $\mathcal{L}_{N}(\epsilon,\omega,w)$ is invertible with
\begin{equation}\label{E4.7}
\left\|\mathcal{L}^{-1}_{N}(\epsilon,\omega,w)h\right\|_{{s'}}\leq \frac{K(s')}{\gamma\omega}N^{\tau-1}\left(\|h\|_{{s'}}+\|w\|_{{s'+\sigma}}\|h\|_{s}\right),\quad\forall s'\geq s>{1}/{2}.
\end{equation}
In particular, for $s'=s$,
\begin{equation}\label{E4.8}
\left\|\mathcal{L}^{-1}_{N}(\epsilon,\omega,w)h\right\|_{s}\leq \frac{K}{\gamma\omega}N^{\tau-1}\|h\|_{s}.
\end{equation}
\begin{proof}
($\mathrm{i}$) It follows from formula \eqref{E2.5} and Lemma \ref{lemma3.1} that $\mathcal{F}\in C^2(H^s;H^s)$. Furthermore, according to \eqref{E2.4}, Lemma \ref{lemma2.4} and Remark \ref{Remark2}, it yields that
\begin{align*}
\|\mathcal{F}(\epsilon,w)\|_{s}&=\|F(v(\epsilon,w)+w)\|_{s}{\leq}
C(1+\|v\|_{H^1}+\|w\|_{s})\leq C,\\
\|\mathrm{D}_{w}\mathcal{F}(\epsilon,w)\|_{s}&=\|\mathrm{D}_{u}F(v(\epsilon,w)+w)(1+\mathrm{D}_{w}v(\epsilon,w))\|_{s}\stackrel{\eqref{E2.5}}{\leq}C,\\
\|\mathrm{D}^2_{w}\mathcal{F}(\epsilon,w)\|_{s}&=\|\mathrm{D}^2_{u}F(v(\epsilon,w)+w)(1+\mathrm{D}_{w}v(\epsilon,w))^2+
\mathrm{D}_{u}F(v(\epsilon,w)+w)\mathrm{D}^2_{w}v(\epsilon,w)\|\stackrel{\eqref{E2.5}}{\leq} C.
\end{align*}

($\mathrm{ii}$) It is easy to obtain that
\begin{align*}
\mathrm{D}^2_{w}\mathcal{F}(\epsilon,w)[h,h]=&\mathrm{D}^2_{u}F(v(\epsilon,w)+w)(h+\mathrm{D}_{w}v(\epsilon,w)[h])^2\\
&+\mathrm{D}_{u}F(v(\epsilon,w)+w)\mathrm{D}^2_{w}v(\epsilon,w)[h,h].
\end{align*}
With the help of \eqref{E2.1}, \eqref{E2.4}, Lemma \ref{lemma2.4} and Remark \ref{Remark2}, we get
\begin{align*}
&\|\mathcal{F}(\epsilon,w)\|_{s'}=\|F(v(\epsilon,w)+w)\|_{s'}{\leq}
C(s')(1+\|v\|_{H^1}+\|w\|_{s'})\leq C(s')(1+\|w\|_{s'}),\\
&\|\mathrm{D}_{w}\mathcal{F}(\epsilon,w)h\|_{s'}=\|\mathrm{D}_{u}F(v(\epsilon,w)+w)(1+\mathrm{D}_{w}v(\epsilon,w))\|_{s'}\stackrel{\eqref{E2.5}}{\leq}C(s')(\|w\|_{s'}\|h\|_{s}+\|h\|_{s'}),\\
&\|\mathrm{D}^2_{u}F(v(\epsilon,w)+w)(h+\mathrm{D}_{w}v(\epsilon,w)[h])^2\|_{s'}\stackrel{\eqref{E2.5}}{\leq} C(s')(\|w\|_{s'}\|h\|^2_s+\|h\|_{s'}\|h\|_{s}),\\
&\|\mathrm{D}_{u}F(v(\epsilon,w)+w)\mathrm{D}^2_{w}v(\epsilon,w)[h,h]\|_{s'}\stackrel{\eqref{E2.5}}{\leq} C(s')(\|w\|_{s'}\|h\|^2_s+\|h\|^2_{s}).
\end{align*}

($\mathrm{iii}$) According to $(P2)$, it is clear to see  that $(P3)$ holds.

($\mathrm{iv}$) Obviously, we can  obtain $(P4)$ owing to \eqref{E4.2} and the definition of $H^s$-norm.
\end{proof}

\subsection{Invertibility of the linearized operator}\label{sec:6.1}
The invertibility of the linearized operators is the core of any Nash-Moser iteration.
Let us complete the proof of the property $(P5)$.

Let $(\epsilon,\omega)\in\Delta^{\gamma,\tau}_{N}(w)$ for fixed $\tau,\gamma$ with $1<\tau<2,0<\gamma<1$, where $\Delta^{\gamma,\tau}_{N}(w)$ is defined by \eqref{E4.3}.
The linear operator may be written as, for all $h\in W_{N}$,
\begin{align*}
\mathcal{L}_{N}(\epsilon,\omega,w)[h]:&=-L_{\omega}h+\epsilon P_N\Pi_{W}\mathrm{D}_{w}F(v(\epsilon,w)+w)[h]
=\mathscr{L}_1(\epsilon,\omega,w)[h]+\mathscr{L}_2(\epsilon,\omega,w)[h],
\end{align*}
where
\begin{align}
&\mathscr{L}_1(\epsilon,\omega,w)[h]:=-L_{\omega}h+\epsilon P_N\Pi_Wf'(t,x,v(\epsilon,w)+w)h,\nonumber\\
&\mathscr{L}_2(\epsilon,\omega,w)[h]:=\epsilon P_N\Pi_Wf'(t,x,v(\epsilon,w)+w)\mathrm{D}_{w}v(\epsilon,w)[h].\label{E7.4}
\end{align}
Lemma \ref{lemma6.1} gives that, for all $u=\sum_{l\in\mathbf{Z}\setminus \{0\},j\in\mathbf{N}}\hat{u}^2_{l,j}\psi_{j}(\epsilon,w)e^{\mathrm{i}lt}$,
the $H^s$-norm
\[\|u\|^{2}_{{s}}=\sum\limits_{l\in\mathbf{Z}\setminus \{0\}}\|u_l\|^2_{H^1}(1+ l^{2s})\]
is equivalent to the norm $\sum_{l\in\mathbf{Z}\setminus \{0\},j\in\mathbf{N}}\lambda_{j}(\epsilon,w)\hat{u}^2_{l,j}(1+l^{2s})$, i.e.
\begin{align}\label{E5.26}
L^2_1\|u\|^2_{s}\leq\sum\limits_{l\in\mathbf{Z}\setminus \{0\},j\in\mathbf{N}}\lambda_{j}(\epsilon,w)\hat{u}^2_{l,j}(1+l^{2s})\leq L^2_2\|u\|^2_{s},
\end{align}
where $L_i,i=1,2$ are seen in \eqref{E6.3}.  Denote  $a(t,x):=f'(t,x,v(\epsilon,w(t,x))+w(t,x))$. It follows from  \eqref{E2.5} and the fact $\|w\|_{{s+\sigma}}\leq1$
that  $\forall v\in\mathcal{H}^1_g$,  $\forall s'\geq s>\frac{1}{2}$,
\begin{align}
&\|a\|_{{s'}}\leq C(s')(1+\|w\|_{{s'}}),\label{E5.14}\\
&\|a\|_{s}\leq\|a\|_{{s+\sigma}}\leq C.\label{E5.15}
\end{align}
Moreover, with the help of decomposing
$a(t,x)=\sum_{k\in\mathbf{Z}}a_{k}(x)e^{\mathrm{i}kt}$, $h=\sum_{1\leq |l|\leq N}h_l(x)e^{\mathrm{i}lt}$, the operator $\mathscr{L}_1(\epsilon,\omega,w)$ can be written as
\begin{align*}
\mathscr{L}_1(\epsilon,\omega,w)[h]=&\sum\limits_{{1\leq|l|\leq N}}\left(\omega^2l^2\rho h_{l}+\partial_{x}(p\partial_{x}h_l-mh_l)\right)e^{\mathrm{i}lt}+\epsilon P_{N}\left(\sum\limits_{k\in\mathbf{Z},{1\leq|l|\leq N}}a_{k-l}h_le^{\mathrm{i}kt}\right)\\
=&\sum\limits_{1\leq|l|\leq N}\left(\omega^2l^2\rho h_{l}+\partial_{x}(p\partial_{x}h_l)-mh_l+\epsilon a_0h_l\right)e^{\mathrm{i}lt}+\epsilon\sum\limits_{|l|,|k|\in\{1,\cdots,N\},k\neq l}a_{k-l}h_le^{\mathrm{i}kt}\\
=&\rho\mathscr{L}_{{1,\mathrm{D}}}(\epsilon,\omega,w)[h]-\rho\mathscr{L}_{{1,\mathrm{ND}}}(\epsilon,\omega,w)[h],
\end{align*}
where
\begin{align}
&\mathscr{L}_{{1,\mathrm{D}}}h:=\sum\limits_{1\leq|l|\leq N}\left(\omega^2l^2 h_{l}+\frac{1}{\rho}\left((ph'_l)'-mh_l+\epsilon a_0h_l\right)\right)e^{\mathrm{i}lt},\nonumber\\
&\mathscr{L}_{{1,\mathrm{ND}}}h:=-\frac{\epsilon}{\rho}\sum\limits_{|l|,|k|\in\{1,\cdots,N\},k\neq l}a_{k-l}h_le^{\mathrm{i}kt}.\label{E5.27}
\end{align}
It is evident that $a_0=\Pi_{V}f'(t,x,v(\epsilon,w)+w)$ and $h_l\in \mathcal{H}^1_g$. Hence, by
Lemma \ref{lemma6.1}, we obtain that for $h_{l}=\sum^{+\infty}_{j=0}\hat{h}_{l,j}\psi_j(\epsilon,w)$
\[\omega^2l^2 h_{l}+\frac{1}{\rho}(ph'_l)'-mh_l+\epsilon a_0h_l=\sum\limits^{+\infty}_{j=0}(\omega^2l^2-\lambda_j(\epsilon,w))\hat{h}_{l,j}\psi_j(\epsilon,w).\]
Then $\mathscr{L}_{{1,\mathrm{D}}}$  is a diagonal operator on $W_{N}$.
In addition we define $|\mathscr{L}_{{1,\mathrm{D}}}|^{\frac{1}{2}}$ as
\begin{align*}
|\mathscr{L}_{{1,\mathrm{D}}}|^{\frac{1}{2}}h=\sum\limits_{1\leq|l|\leq N}\sum\limits^{+\infty}_{j=0}{|\omega^2l^2-\lambda_j(\epsilon,w)|}^{\frac{1}{2}} \hat{h}_{l,j}\psi_j(\epsilon,w)e^{\mathrm{i}lt}.
\end{align*}
If  $\omega^2l^2-\lambda_j(\epsilon,w)\neq0,~\forall1\leq|l|\leq N,~\forall j\geq0$, then its invertibility is
\begin{align*}
|\mathscr{L}_{{1,\mathrm{D}}}|^{-\frac{1}{2}}{h}:=\sum\limits_{1\leq|l|\leq N}\sum\limits^{+\infty}_{j=0}\frac{\hat{{h}}_{l,j}}{\sqrt{|\omega^2l^2-\lambda_j(\epsilon,w)|}}\psi_j(\epsilon,w)e^{\mathrm{i}lt}.
\end{align*}
Thus $\mathcal{L}_{N}(\epsilon,\omega,w)$  can be written as
\begin{align*}
\mathcal{L}_{N}(\epsilon,\omega,w)=\rho|\mathscr{L}_{1,{\mathrm{D}}}|^{\frac{1}{2}}\left(|\mathscr{L}_{1,{\mathrm{D}}}|^{-\frac{1}{2}}
\mathscr{L}_{1,{\mathrm{D}}}|\mathscr{L}_{1,{\mathrm{D}}}|^{-\frac{1}{2}}-{R}_1-{R}_2\right)|\mathscr{L}_{1,{\mathrm{D}}}|^{\frac{1}{2}},
\end{align*}
where
\begin{equation}\label{E5.28}
{R}_1=|\mathscr{L}_{1,{\mathrm{D}}}|^{-\frac{1}{2}}
\mathscr{L}_{1,{\mathrm{ND}}}|\mathscr{L}_{1,{\mathrm{D}}}|^{-\frac{1}{2}},\quad
{R}_2=-|\mathscr{L}_{1,{\mathrm{D}}}|^{-\frac{1}{2}}
\left(1/\rho\mathscr{L}_2\right)|\mathscr{L}_{1,{\mathrm{D}}}|^{-\frac{1}{2}}.
\end{equation}
The definitions of $\mathscr{L}_{1,\mathrm{D}}$, $|\mathscr{L}_{1,{\mathrm{D}}}|^{-\frac{1}{2}}$ give
\begin{align*}
\left(|\mathscr{L}_{1,{\mathrm{D}}}|^{-\frac{1}{2}}
\mathscr{L}_{1,{\mathrm{D}}}|\mathscr{L}_{1,{\mathrm{D}}}|^{-\frac{1}{2}}\right){h}=\sum\limits_{1\leq |l|\leq N}\sum\limits_{j=0}^{+\infty}\frac{\omega^2l^2-\lambda_j(\epsilon,w)}{|\omega^2l^2-\lambda_j(\epsilon,w)|}\hat{h}_{l,j}\psi_{j}(\epsilon,w)e^{\mathrm{i}lt},\quad \forall{h}\in{W}_{N}.
\end{align*}
Consequently, $|\mathscr{L}_{1,{\mathrm{D}}}|^{-\frac{1}{2}}\mathscr{L}_{1,{\mathrm{D}}}|\mathscr{L}_{1,{\mathrm{D}}}|^{-\frac{1}{2}}$ is invertible with
\begin{align*}
\left(|\mathscr{L}_{1,{\mathrm{D}}}|^{-\frac{1}{2}}
\mathscr{L}_{1,{\mathrm{D}}}|\mathscr{L}_{1,{\mathrm{D}}}|^{-\frac{1}{2}}\right)^{-1}{h}=
\sum\limits_{1\leq |l|\leq N}\sum\limits_{j=0}^{+\infty}\frac{|\omega^2l^2-\lambda_j(\epsilon,w)|}
{\omega^2l^2-\lambda_j(\epsilon,w)}\hat{h}_{l,j}\psi_{j}(\epsilon,w)e^{\mathrm{i}lt},\quad \forall{h}\in{W}_{N}.
\end{align*}
Hence $\mathcal{L}_{N}(\epsilon,\omega,w)$ is reduced to
\begin{align}\label{E5.34}
\mathcal{L}_{N}(\epsilon,\omega,w)=\rho|\mathscr{L}_{1,{\mathrm{D}}}|^{\frac{1}{2}}\left(|\mathscr{L}_{1,{\mathrm{D}}}|^{-\frac{1}{2}}
\mathscr{L}_{1,{\mathrm{D}}}|\mathscr{L}_{1,{\mathrm{D}}}|^{-\frac{1}{2}}\right)\left(\mathrm{Id}-\mathcal{R}\right)
|\mathscr{L}_{1,{\mathrm{D}}}|^{\frac{1}{2}},
\end{align}
where $\mathcal{R}=\mathcal{R}_1+\mathcal{R}_2$ with
\begin{align*}
\mathcal{R}_1=\left(|\mathscr{L}_{1,{\mathrm{D}}}|^{-\frac{1}{2}}
\mathscr{L}_{1,{\mathrm{D}}}|\mathscr{L}_{1,{\mathrm{D}}}|^{-\frac{1}{2}}\right)^{-1}R_1,\quad
\mathcal{R}_2=\left(|\mathscr{L}_{1,{\mathrm{D}}}|^{-\frac{1}{2}}
\mathscr{L}_{1,{\mathrm{D}}}|\mathscr{L}_{1,{\mathrm{D}}}|^{-\frac{1}{2}}\right)^{-1}R_2.
\end{align*}
To verify the invertibility of $\left(\mathrm{Id}-\mathcal{R}\right)$ for all $s'\geq s>\frac{1}{2}$ and $h\in W_{N}$, it is required to estimate the upper bounds of $\left\|\mathcal{R}_ih\right\|_{s'},i=1,2$. This indicates that, $\forall s>0$, $\forall h\in W_{N}$, the upper bounds of $\|~|\mathscr{L}_{{1,\mathrm{D}}}|^{-\frac{1}{2}}{h}\|_{{s}}$ and $\|(|\mathscr{L}_{1,{\mathrm{D}}}|^{-\frac{1}{2}}
\mathscr{L}_{1,{\mathrm{D}}}|\mathscr{L}_{1,{\mathrm{D}}}|^{-\frac{1}{2}})^{-1}{h}\|_{s}$ have to be given.

First, the equivalent norm \eqref{E5.26} shows for all $s\geq0$
\begin{align}\label{E5.30}
\left\|\left(|\mathscr{L}_{1,{\mathrm{D}}}|^{-\frac{1}{2}}
\mathscr{L}_{1,{\mathrm{D}}}|\mathscr{L}_{1,{\mathrm{D}}}|^{-\frac{1}{2}}\right)^{-1}{h}\right\|_{s}\leq&\frac{1}{L_1}\left(\sum\limits_{1\leq|l|\leq N,j\in\mathbf{N}}\lambda_j(\epsilon,w)\left(\frac{|\omega^2l^2-\lambda_j(\epsilon,w)|}
{\omega^2l^2-\lambda_j(\epsilon,w)}\hat{h}_{l,j}\right)^2(1+l^{2s})\right)^{\frac{1}{2}}\nonumber\\
=&\frac{1}{L_1}\left(\sum\limits_{1\leq|l|\leq N,j\in\mathbf{N}}\lambda_j(\epsilon,w)\hat{h}^2_{l,j}(1+l^{2s})\right)^{\frac{1}{2}}\nonumber\\
\leq& \frac{L_2}{L_1}\|h\|_{s}.
\end{align}
To establish a proper upper bound of $\|~|\mathscr{L}_{{1,\mathrm{D}}}|^{-\frac{1}{2}}{h}\|_{{s}}$, assume  the following ``Melnikov's'' non-resonance conditions:
\begin{equation}\label{E5.12}
|\omega^2l^2-\lambda_j(\epsilon,w)|>\frac{\gamma\omega}{l^{\tau-1}},\quad\forall l=1,\cdots,N,~j\geq 0.
\end{equation}
 Let
\begin{equation}\label{E6.27}
\varpi_l:=\min_{j\geq0}|\omega^2l^2-\lambda_j(\epsilon,w)|=|\omega^2l^2-\lambda_{j^*}(\epsilon,w)|
\end{equation}
satisfy $\varpi_l=\varpi_{-l},l=1,2,\cdots,N$. Based on assumption \eqref{E5.12}, we claim that, $\forall {h}\in{W}_{N}$, $\forall s\geq0$, the operator $|\mathscr{L}_{{1,\mathrm{D}}}|^{-\frac{1}{2}}$
satisfies
\begin{align}\label{E5.17}
\|~|\mathscr{L}_{{1,\mathrm{D}}}|^{-\frac{1}{2}}{h}\|_{{s}}\leq\frac{\sqrt{2}L_2}{\sqrt{\gamma\omega}L_1}\|{h}\|_{{s+\frac{\tau-1}{2}}},\\
\|~|\mathscr{L}_{{1,\mathrm{D}}}|^{-\frac{1}{2}}{h}\|_{{s}}\leq\frac{\sqrt{2}{L_2}N^{\frac{\tau-1}{2}}}{\sqrt{\gamma\omega}L_1}\|{h}\|_{{s}} \label{E5.17.2}.
\end{align}
It is clear that
\[
|l|^{\tau-1}(1+l^{2s})<2(1+|l|^{2s+\tau-1}),\quad  \forall |l|\geq1.
\]
Applying this together with \eqref{E5.26}, \eqref{E5.12}-\eqref{E6.27}, we derive
\begin{align*}
\|~|\mathscr{L}_{{1,\mathrm{D}}}|^{-\frac{1}{2}}{h}\|^2_{{s}}&\leq\frac{1}{L^2_1}\sum\limits_{1\leq|l|\leq N,j\in\mathbf{N}}\lambda_j(\epsilon,w)\left(\frac{\hat{{h}}_{l,j}}{\sqrt{|\omega^2l^2-\lambda_j(\epsilon,w)|}}\right)^2(1+l^{2s})\\
&\leq\frac{1}{L^2_1}\sum\limits_{1\leq|l|\leq N,j\in\mathbf{N}}\frac{\lambda_j(\epsilon,w)}{\varpi_l}\hat{{h}}^2_{l,j}(1+l^{2s})\\
&\leq\frac{1}{\gamma\omega L^2_1}\sum\limits_{
1\leq|l|\leq N,j\in\mathbf{N}}|l|^{\tau-1}\lambda_j(\epsilon,w)\hat{h}^2_{{l,j}}(1+l^{2s})\\
&\leq \frac{2L^2_2}{\gamma\omega L^2_1}\|h\|^2_{{s+\frac{\tau-1}{2}}}\leq \frac{2L^2_2N^{\tau-1}}{\gamma\omega L^2_1}\|h\|^2_{{s}}.
\end{align*}

The next step is to verify the upper bounds of $\left\|\mathcal{R}_ih\right\|_{s'},i=1,2$ for all $s'\geq s>\frac{1}{2}$ and $h\in W_{N}$. Assume that the other ``Melnikov'' non-resonance conditions holds:
\begin{equation}\label{E5.11}
\left|\omega l-\frac{j}{c}\right|>\frac{\gamma}{l^\tau},\quad\forall l=1,\cdots,N,~j\geq 0,
\end{equation}
where $c$ is defined by \eqref{E6.23}. In fact,
   condition \eqref{E5.11} will be  applied in the proof of the claim {\bf(F1)} (see \eqref{E5.16}).
   To prove the claim {\bf(F1)}, we have to  use the asymptotic formulae \eqref{E6.24}-\eqref{E6.25} of  $\lambda_{j}(\epsilon,w)$ and guarantee $j^*\neq0$.  This leads to  some restrictions on $\omega$.
   We discuss it in three cases:

$\mathrm{(i)}$ $\alpha_1=-\frac{\beta_1}{4}\frac{(p\rho)_{x}(0)}{(p\rho)(0)},\beta_1>0, \alpha_2=\frac{\beta_2}{4}\frac{(p\rho)_{x}(0)}{(p\rho)(0)},\beta_2>0$. For all $(\epsilon,w)\in[\epsilon_1,\epsilon_2]\times\{W\cap H^s:\|w\|_{s}\leq r\}$, there exists a constant $\omega_1>\gamma$ such that, $\forall l\geq1$, $\forall\omega\geq\omega_1$, the following  holds:
\begin{equation*}
\omega^2-\lambda_0(\epsilon,w)>\omega^2-\lambda_1(\epsilon,w)>0\Rightarrow j^*\geq1,
\end{equation*}
where $j^*$ is given by \eqref{E6.27}. Owing to \eqref{E6.24} and Taylor expansion,  there exists some constant $\mathfrak{M}_0>0$ such that
\begin{equation*}
\left|\sqrt{\lambda_j(\epsilon,w)}-{j}/{c}\right|\leq {\mathfrak{M}_0}/{j},\quad \forall j\geq1.
\end{equation*}
Moreover we obtain that  for some constant  $\mathfrak{N}_{0}>0$
\begin{equation*}
\varpi_l=\min_{j\geq0}|\omega^2l^2-\lambda_j(\epsilon,w)|=|\omega^2l^2-\lambda_{j^*}(\epsilon,w)|
\Rightarrow j^*\geq \mathfrak{N}_{0}\omega l.
\end{equation*}

$\mathrm{(ii)}$ Either $\alpha_1>0,\beta_1=0,\alpha_2=\frac{\beta_2}{4}\frac{(p\rho)_{x}(\pi)}{(p\rho)(\pi)},\beta_2>0$
or $\alpha_1=-\frac{\beta_1}{4}\frac{(p\rho)_{x}(0)}{(p\rho)(0)},\beta_1>0,\alpha_2>0,\beta_2=0$. The same discussion is adopted as in $\mathrm{(i)}$.
 There exists some constant $\omega_2>\gamma$ such that $\forall(\epsilon,w)\in[\epsilon_1,\epsilon_2]\times\{W\cap H^s:\|w\|_{s}\leq r\}$, $\forall l\geq1$, $\forall \omega\geq\omega_2$,
\begin{equation*}
\omega^2-\lambda_0(\epsilon,w)>\omega^2-\lambda_1(\epsilon,w)>0\Rightarrow j^*\geq1.
\end{equation*}
Formula \eqref{E6.38} together with Taylor expansion can derive
\begin{equation*}
\left|\sqrt{\lambda_j(\epsilon,w)}-{(j+1/2)}/{c}\right|\leq{\mathfrak{M}_1}/{(j+1/2)}\leq {\mathfrak{M}_1}/{j},\quad\forall j\geq1
\end{equation*}
for some constants $\mathfrak{M}_1>0$. Furthermore
\begin{equation*}
\varpi_l=\min_{j\geq0}|\omega^2l^2-\lambda_j(\epsilon,w)|=|\omega^2l^2-\lambda_{j^*}(\epsilon,w)|
\Rightarrow j^*\geq \mathfrak{N}_{1}\omega l\quad \text{for}~\text{some}~\mathfrak{N}_{1}>0.
\end{equation*}

$\mathrm{(iii)}$:  $\alpha_1>-\frac{\beta_1}{4}\frac{(p\rho)_{x}(0)}{(p\rho)(0)}, \beta_1>0, \alpha_2>\frac{\beta_2}{4}\frac{(p\rho)_{x}(0)}{(p\rho)(0)},\beta_2>0$.
The same discussion can be  taken as $\mathrm{(i)}$. There exists a  constant $\omega_3>\gamma$ satisfying $\forall(\epsilon,w)\in[\epsilon_1,\epsilon_2]\times\{W\cap H^s:\|w\|_{s}\leq r\}$,
$\forall l\geq1$, $\forall\omega\geq\omega_3$,
\begin{equation*}
\omega^2-\lambda_{\widehat{N}}(\epsilon,w)>0\Rightarrow j^*\geq \widehat{N},
\end{equation*}
where $\widehat{N}$ is seen in Lemma \ref{lemma6.8}. It follows from formula \eqref{E6.25} and Taylor expansion that for some $\mathfrak{M}_2>0$
\begin{equation*}
\left|\sqrt{\lambda_j(\epsilon,w)}-{j}/{c}\right|\leq {\mathfrak{M}_2}/{j},\quad\forall j\geq \widehat{N}.
\end{equation*}
In addition there exists some constant $\mathfrak{N}_{2}>0$ such that
\begin{equation*}
\varpi_l=\min_{j\geq0}|\omega^2l^2-\lambda_j(\epsilon,w)|=|\omega^2l^2-\lambda_{j^*}(\epsilon,w)|
\Rightarrow j^*\geq \mathfrak{N}_{2}\omega l.
\end{equation*}
Hence, $\forall l\geq1$, $\forall\omega\geq\omega_0>\gamma$, $\forall\epsilon\in[\epsilon_1,\epsilon_2]$, $\forall w\in\{W\cap H^s:\|w\|_{s}\leq r\}$, the eigenvalues $\lambda_{j^*}(\epsilon)$ (see \eqref{E6.27}) of \eqref{E6.1} satisfy
\begin{equation}\label{E6.28}
\left|\sqrt{\lambda_{j^*}(\epsilon,w)}-{j^*}/{c}\right|\leq \mathfrak{M}/{j^*}\quad\mathrm{and}\quad
 j^*\geq\mathfrak{N}\omega l,
\end{equation}
or
\begin{equation*}
\left|\sqrt{\lambda_{j^*}(\epsilon,w)}-{(j^*+1/2)}/{c}\right|\leq \mathfrak{M}/{j^*}\quad\mathrm{and}\quad
 j^*\geq\mathfrak{N}\omega l,
\end{equation*}
where $\mathfrak{M}:=\max\{\mathfrak{M}_0,\mathfrak{M}_1,\mathfrak{M}_2\}$, $\mathfrak{N}:=\max\{\mathfrak{N}_0,\mathfrak{N}_1,\mathfrak{N}_2\}$ and
\begin{align}\label{E6.39}
\omega_0:=\max\{\omega_1,\omega_2,\omega_3\}.
\end{align}

\begin{lemm}\label{lemma6.5}

For  $\|w\|_{s+\sigma}\leq1$, under the non-resonance conditions \eqref{E5.12} and \eqref{E5.11},
 there exists some constant  $M(s')>0$ such that
\begin{equation}\label{E5.32}
\left\|\mathcal{R}_1h\right\|_{s'}\leq\frac{\epsilon M(s')}{2\gamma^3\omega}\left(\|h\|_{s'}+\|w\|_{s'+\sigma}\|h\|_s\right),\quad\forall h\in W_{N}, \forall s'\geq s>{1}/{2}.
\end{equation}
\end{lemm}
\begin{proof}
By formulae \eqref{E5.27}-\eqref{E5.28} and the definition of $|\mathscr{L}_{{1,\mathrm{D}}}|^{-\frac{1}{2}}$, we obtain for all $h\in W_{N}$
\begin{align*}
R_1h&=|\mathscr{L}_{1,{\mathrm{D}}}|^{-\frac{1}{2}}
\mathscr{L}_{1,{\mathrm{ND}}}|\mathscr{L}_{1,{\mathrm{D}}}|^{-\frac{1}{2}}h\\
&=|\mathscr{L}_{1,{\mathrm{D}}}|^{-\frac{1}{2}}\mathscr{L}_{1,{\mathrm{ND}}}\left(\sum\limits_{1\leq|l|\leq N}\sum\limits^{+\infty}_{j=1}\frac{\hat{{h}}_{l,j}}{\sqrt{|\omega^2l^2-\lambda_j(\epsilon,w)|}}\psi_j(\epsilon,w)e^{\mathrm{i}lt}\right)\\
&=-\epsilon|\mathscr{L}_{1,{\mathrm{D}}}|^{-\frac{1}{2}}\left(\sum\limits_{\stackrel{1\leq|l|,|k|\leq N}{k\neq l}}\sum\limits_{j=1}^{+\infty}\frac{\hat{h}_{l,j}}{\sqrt{|\omega^2l^2-\lambda_j(\epsilon,w)|}}\left(\frac{a_{k-l}}{\rho}\right)\psi_j(\epsilon,w)e^{\mathrm{i}kt}\right)\\
&=-\epsilon\sum\limits_{\stackrel{1\leq|l|,|k|\leq N}{k\neq l}}\sum\limits_{j=1}^{+\infty}\frac{\hat{h}_{l,j}}{\sqrt{|\omega^2k^2-\lambda_j(\epsilon,w)|}
\sqrt{|\omega^2l^2-\lambda_j(\epsilon,w)|}}\left(\frac{a_{k-l}}{\rho}\right)\psi_j(\epsilon,w)e^{\mathrm{i}kt},
\end{align*}
which then leads to
\begin{align*}
(R_1h)_{k}=-\epsilon\sum\limits_{\stackrel{1\leq|l|\leq N}{k\neq l}}\sum\limits_{j=1}^{+\infty}\frac{\hat{h}_{l,j}}{\sqrt{|\omega^2k^2-\lambda_j(\epsilon,w)|}
\sqrt{|\omega^2l^2-\lambda_j(\epsilon,w)|}}\left(\frac{a_{k-l}}{\rho}\right)\psi_j(\epsilon,w).
\end{align*}
Formulae \eqref{E6.3}-\eqref{E6.4} verify that
\begin{equation}\label{E5.29}
\|(R_1h)_k\|_{H^1}\leq\epsilon\frac{L_2}{L_1}\sum\limits_{1\leq|l|\leq N,l\neq k}\frac{1}{\sqrt{\varpi_k\varpi_l}}\left\|{a_{k-l}}/{\rho}\right\|_{H^1}\|h_l\|_{H^1}.
\end{equation}

In addition we claim that the following fact holds:\\
{\bf (F1)}:
Supposed that the non-resonance conditions \eqref{E5.12} and \eqref{E5.11} hold, if $\omega\geq\omega_0>\gamma$, for all $|k|,|l|\in\{1,\cdots,N\}$ with $k\neq l$, then we have
\begin{equation}\label{E5.16}
\frac{1}{\sqrt{\varpi_l\varpi_k}}\leq\frac{|k-l|^{\sigma}}{L\gamma^3\omega}
\end{equation}
for some constant $L>0$, where $\varpi_l$ is defined by \eqref{E6.27}, $\sigma:=(\tau-1)/{\varsigma}$ with $\varsigma=(2-\tau)/\tau\in(0,1)$.

 It follows from formulae \eqref{E5.29}-\eqref{E5.16} that
\begin{align*}
\|(R_1h)_k\|_{H^1}\leq\frac{\epsilon L_2}{LL_1\gamma^3\omega}\sum\limits_{1\leq|l|\leq N,l\neq k}\left\|{a_{k-l}}/{\rho}\right\|_{H^1}|k-l|^\sigma\|h_l\|_{H^1}.
\end{align*}
Let us define for $b_0=0$,
\begin{align*}
&{s}(t):=\sum\limits_{\stackrel{1\leq |l|,|k|\leq N}{l\neq k}}\left\|{a_{k-l}}/{\rho}\right\|_{H^1}|k-l|^\sigma\|h_l\|_{H^1}e^{\mathrm{i}kt},\\ &b(t):=\sum_{l\in\mathbf{Z}}\left\|{a_{l}}/{\rho}\right\|_{H^1}|l|^{\sigma} e^{\mathrm{i}lt},\quad
c(t):=\sum_{1\leq|l|\leq N}\|h_l\|_{H^1}e^{\mathrm{i}lt}.
\end{align*}
It is straightforward that ${s}=P_{N}(bc)$. Moreover
\begin{equation*}
\|b\|_{s'}\leq \widehat{M}\|a\|_{s'+\sigma}\stackrel{\eqref{E5.14}}{\leq}C(s')(1+\|w\|_{s'+\sigma})
\quad\text{and}\quad
\|c\|_{s'}=\|h\|_{s'},
\end{equation*}
where $\widehat{M}=\|1/\rho\|_{L^{\infty}}$. Hence, $\forall s'\geq s>\frac{1}{2}$, we get for $\|w\|_{s+\sigma}\leq1$
\begin{align*}
\|R_1h\|_{s'}&\leq\frac{\epsilon L_2}{LL_1\gamma^3\omega}\|{s}\|_{s'}
=\frac{\epsilon L_2}{LL_1\gamma^3\omega}\|P_{N}(bc)\|_{s'}\stackrel{\eqref{E2.1}}{\leq}\frac{\epsilon L_2C(s')}{LL_1\gamma^3\omega}(\|b\|_{s'}\|c\|_{s}+\|b\|_{s}\|c\|_{s'})\\
&\leq\frac{\epsilon M_1(s')}{2\gamma^3\omega}(\|w\|_{s'+\sigma}\|h\|_s+\|h\|_{s'}).
\end{align*}
If $M(s')=\frac{L_2}{L_1}M_1(s')$, then combining above inequality with \eqref{E5.30} completes the proof of the lemma.
\end{proof}

\begin{lemm}\label{lemma6.7}
Given the non-resonance conditions \eqref{E5.12}, \eqref{E5.11} and $\|w\|_{s+\sigma}\leq1$, $\forall h\in W_{N}$, $\forall s'\geq s>\frac{1}{2}$, we have for some constant $M(s')>0$
\begin{equation}\label{E5.31}
\left\|\mathcal{R}_2h\right\|_{s'}\leq\frac{\epsilon M(s')}{2\gamma\omega}\left(\|h\|_{s'}+\|w\|_{s'+\sigma}\|h\|_s\right).
\end{equation}
\end{lemm}
\begin{proof}
The fact  $\varsigma\in(0,1)$ in the claim {\bf(F1)} gives rise to
\[\sigma=(\tau-1)/{\varsigma}>\tau-1.\]
Therefore, according to \eqref{E7.4}, \eqref{E5.14}-\eqref{E5.15}, \eqref{E5.28}, \eqref{E5.17}, \eqref{E2.1} and $\mathrm{D}_{w}v(\epsilon,w)[|\mathscr{L}_{1,{\mathrm{D}}}|^{-\frac{1}{2}}h]\in \mathcal{H}^1_g$, we derive for all $s'\geq s>\frac{1}{2}$
\begin{align*}
\|R_2h\|_{s'}&=\left\|\frac{\epsilon}{\rho}|\mathscr{L}_{1,{\mathrm{D}}}|^{-\frac{1}{2}}P_N\Pi_Wa(t,x)\mathrm{D}_wv(\epsilon,w)[|\mathscr{L}_{1,{\mathrm{D}}}|^{-\frac{1}{2}}h]\right\|_{s'}\\
&\leq\frac{\epsilon\sqrt{2}L_2\widehat{M}}{L_1\sqrt{\gamma\omega}}\left\|a(t,x)\mathrm{D}_wv(\epsilon,w)[|\mathscr{L}_{1,{\mathrm{D}}}|^{-\frac{1}{2}}h]\right\|_{s'+\frac{\tau-1}{2}}\\
&\leq\frac{\epsilon \sqrt{2} L_2\widehat{M}C(s')}{L_1\sqrt{\gamma\omega}}\left(\|a\|_{s'+\frac{\tau-1}{2}}\left\|\mathrm{D}_{w}v[|\mathscr{L}_{1,{\mathrm{D}}}|^{-\frac{1}{2}}h]\right\|_{H^1}
+\|a\|_s\left\|\mathrm{D}_{w}v[|\mathscr{L}_{1,{\mathrm{D}}}|^{-\frac{1}{2}}h]\right\|_{H^1}\right)\\
&\leq\frac{\epsilon \sqrt{2} L_2\widehat{M}C(s')}{L_1\sqrt{\gamma\omega}}\left(\|a\|_{s'+\sigma}\||\mathscr{L}_{1,{\mathrm{D}}}|^{-\frac{1}{2}}h\|_{s-\frac{\tau-1}{2}}+\|a\|_{s+\sigma}\||\mathscr{L}_{1,{\mathrm{D}}}|^{-\frac{1}{2}}h\|_{s-\frac{\tau-1}{2}}\right)\\
&\leq\frac{\epsilon M_1(s')}{2{\gamma\omega}}\left(\|w\|_{s'+\sigma}\|h\|_s+\|h\|_{s'}\right).
\end{align*}
Combining above inequality  with \eqref{E5.30}, if $M(s')=\frac{L_2}{L_1}M_1(s')$, then we give \eqref{E5.31}.
\end{proof}
It follows from  Lemmas \ref{lemma6.5}-\ref{lemma6.7} that
 the operator $(\mathrm{Id}-\mathcal{R})$ is invertible for $\frac{\epsilon M(s')}{\gamma^3\omega}$ small enough. Furthermore we have the following estimate for the invertible operator.

\begin{lemm}
For $\frac{\epsilon M(s')}{\gamma^3\omega}\leq {\tilde{\mathfrak{c}}}(s')$ small enough, we have
\begin{equation}\label{E5.21}
\|(\mathrm{Id}-\mathcal{R})^{-1}h\|_{s'}\leq 2(\|h\|_{s'}+\|w\|_{s'+\sigma}\|h\|_s).
\end{equation}
\end{lemm}
\begin{proof}
First,  we claim that:\\
{\bf (F2):} There exists $M_2(s')>0$ such that, $\forall h\in W_{N}$, $\forall s'\geq s>\frac{1}{2}$, $\forall p\in\mathbf{N}^{+}$ and $\|w\|_{s+\sigma}\leq1$
\begin{equation}\label{E5.19}
\|\mathcal{R}^ph\|_{s'}\leq\left(\frac{\epsilon M_2(s')}{\gamma^3\omega}\right)^p(\|h\|_{s'}+p\|w\|_{s'+\sigma}\|h\|_s).
\end{equation}

The claim {\bf (F2)} can be proved by induction. For $p=1$, $0<\gamma<1$, \eqref{E5.32} and \eqref{E5.31}  imply
\begin{equation}\label{E5.33}
\left\|\mathcal{R}h\right\|_{s'}\leq\frac{\epsilon M(s')}{\gamma^3\omega}\left(\|h\|_{s'}+\|w\|_{s'+\sigma}\|h\|_s\right).
\end{equation}
In particular, for $\|w\|_{s+\sigma}\leq1$, formula \eqref{E5.33} infers
\begin{equation}\label{E5.20}
\|\mathcal{R}h\|_{s}\leq\frac{\epsilon M}{\gamma^3\omega}\|h\|_{s}.
\end{equation}
Suppose that \eqref{E5.19} holds for $p=l$ with $l\in\mathbf{N}^{+}$. Based on the assumption and \eqref{E5.33}-\eqref{E5.20}, for $\frac{\epsilon M(s')}{\gamma^3\omega}\leq \tilde{\mathfrak{c}}(s')$ small enough, we get for $p=l+1$
\begin{align*}
\|\mathcal{R}^{l+1}h\|_{s'}=&\|\mathcal{R}^l(\mathcal{R}h)\|_{s'}{\leq}
\left(\frac{\epsilon M_2(s')}{\gamma^3\omega}\right)^l(\|\mathcal{R}h\|_{s'}+l\|w\|_{s'+\sigma}\|\mathcal{R}h\|_s)\\
\leq&\left(\frac{\epsilon M_2(s')}{\gamma^3\omega}\right)^l\left(\frac{\epsilon M(s')}{\gamma^3\omega}\|h\|_{s'}+\left(\frac{\epsilon M}{\gamma^3\omega}l+\frac{\epsilon M(s')}{\gamma^3\omega}\right)\|w\|_{s'+\sigma}\|h\|_s\right)\\
\leq&\left(\frac{\epsilon M_2(s')}{\gamma^3\omega}\right)^{l+1}(\|h\|_{s'}+(l+1)\|w\|_{s'+\sigma}\|h\|_s),
\end{align*}
where $M_2(s'):=M(s')$. As a consequence, for $\frac{\epsilon M(s')}{\gamma^3\omega}\leq \tilde{\mathfrak{c}}(s')$ small enough,  the claim {\bf (F2)} is proved.
Thus, for $\frac{\epsilon M(s')}{\gamma^3\omega}\leq \tilde{\mathfrak{c}}(s')$ small enough, the claim {\bf (F2)} establishes
\begin{align*}
\|(\mathrm{Id}-\mathcal{R})^{-1}h\|_{s'}=&\|(\mathrm{Id}+\sum\limits_{p\in\mathbf{N}^{+}}\mathcal{R}^p)h\|_{s'}\leq\|h\|_{s'}+\sum\limits_{p\in\mathbf{N}^{+}}\|\mathcal{R}^ph\|_{s'}\\
\leq&\|h\|_{s'}+\sum\limits_{p\in\mathbf{N}^{+}}\left(\frac{M_2(s')\epsilon}{\gamma^3\omega}\right)^p(\|h\|_{s'}+p\|w\|_{s'+\sigma}\|h\|_s)\\
\leq&2\|h\|_{s'}+2\|w\|_{s'+\sigma}\|h\|_s,
\end{align*}
which completes the proof.
\end{proof}

Now, let us show that  the claim {\bf(F1)} holds.

\begin{proof}[{\bf{The proof of}} {\bf(F1)}]
Recall that  $l,k\geq 1$ and the asymptotic formulae \eqref{E6.24} or \eqref{E6.25}.

Case 1: $2|k-l|>(\max{\{k,l\}})^{\varsigma}$ with $\varsigma\in(0,1)$. The condition \eqref{E5.12} shows
\begin{equation}\label{E7.5}
\varpi_l\varpi_k\geq\frac{(\gamma\omega)^2}{(kl)^{\tau-1}}\geq\frac{(\gamma\omega)^2}{(\max{\{k,l\}})^{2(\tau-1)}}
\geq\frac{(\gamma\omega)^2}{2^{1/\varsigma}|k-l|^{2(\tau-1)/\varsigma}}.
\end{equation}

Case 2: $0<2|k-l|\leq(\max{\{k,l\}})^{\varsigma}$. Since $\varsigma\in(0,1)$, if $k>l$, then
\begin{equation*}
2(k-l)\leq k^{\varsigma}~\Rightarrow~ 2l\geq2k-k^{\varsigma}\geq k.
\end{equation*}
In the same way, if $l>k$, then $2k\geq l$. As a result
$\frac{k}{2}\leq l\leq2k$.
Moreover denote
\begin{equation*}
\varpi_l:=|\omega^2l^2-\lambda_{j^*}(\epsilon,w)|,\quad\varpi_k:=|\omega^2k^2-\lambda_{i^*}(\epsilon,w)|,
\end{equation*}
 and $k_{\star}:=\left(\frac{6L_3}{\gamma\omega}\right)^{\frac{1}{1-\varsigma\tau}}$ with $L_3={\mathfrak{M}}/{\mathfrak{N}}$.

$\mathrm{(i)}$ Let us consider $\max{\{k,l\}}\geq k_\star$. Suppose $\max{\{k,l\}}=k$. It follows from \eqref{E5.11}-\eqref{E6.28}, $k\leq 2l$ and $\tau\in(1,2)$ that
\begin{align*}
\Big|\Big(\omega k-\sqrt{\lambda_{i^*}(\epsilon,w)}\Big)
-&\Big(\omega l-\sqrt{\lambda_{j^*}(\epsilon,w)}\Big)\Big|=\left|\omega (k-l)-\left(\sqrt{\lambda_{i^*}(\epsilon,w)}-\sqrt{\lambda_{j^*}(\epsilon,w)}\right)\right|\\
\geq&\left|\omega(k-l)-\frac{i^*-j^*}{c}\right|-\left|\sqrt{\lambda_{i^*}(\epsilon,w)}-(i^*/c)\right|-\left|\sqrt{\lambda_{j^*}(\epsilon,w)}-(j^*/c)\right|\\
\geq&\left|\omega(k-l)-\frac{i^*-j^*}{c}\right|-\frac{L_3}{\omega l}-\frac{L_3}{\omega k}>\frac{\gamma}{(k-l)^\tau}-\frac{3L_3}{\omega k}
\\
\geq &\frac{2^\tau \gamma}{k^{\varsigma\tau}}-\frac{3L_3}{\omega k}\geq\frac{\gamma}{2k^{\varsigma\tau}}+\frac{\gamma}{k^{\varsigma\tau}}+\frac{\gamma}{2k^{\varsigma\tau}}-\frac{3L_3}{\omega k}.
\end{align*}
The expression $\varsigma=(2-\tau)/\tau\in(0,1)$ leads to $\varsigma\tau<1$. Hence, from $k\leq2l$ and $\max{\{k,l\}}=k\geq k_\star$, we derive \begin{equation}\label{E5.10}
\left|\left(\omega k-\sqrt{\lambda_{i^*}(\epsilon,w)}\right)\right|+\left|\left(\omega l-\sqrt{\lambda_{j^*}(\epsilon,w)}\right)\right|
\geq\frac{1}{2}\left(\frac{\gamma}{k^{\varsigma\tau}}+\frac{\gamma}{l^{\varsigma\tau}}\right).
\end{equation}
The same conclusion is reached if $\max{\{k,l\}}=l$. In addition the inequality in \eqref{E5.10} implies that
\begin{equation*}
\mathrm{either}\quad\left|\omega k-\sqrt{\lambda_{i^*}(\epsilon,w)}\right|\geq\frac{\gamma}{2k^{\varsigma\tau}}\quad\mathrm{or}\quad\left|\omega l-\sqrt{\lambda_{j^*}(\epsilon,w)}\right|\geq\frac{\gamma}{2l^{\varsigma\tau}}
\end{equation*}
holds. Without loss of generality, we suppose $|\omega k-\sqrt{\lambda_{i^*}(\epsilon,w)}|\geq\frac{\gamma}{2k^{\varsigma\tau}}$. Then
\begin{equation*}
\varpi_k=|\omega^2 k^2-\lambda_{i^*}(\epsilon,w)|=\left|\left(\omega k+\sqrt{\lambda_{i^*}(\epsilon,w)}\right)\left(\omega k-\sqrt{\lambda_{i^*}(\epsilon,w)}\right)\right|\geq \frac{\gamma\omega}{2}k^{1-\varsigma\tau},
\end{equation*}
which leads to
\begin{equation}\label{E7.6}
\varpi_l\varpi_k\stackrel{\eqref{E5.12}}{\geq} \frac{\gamma\omega}{l^{\tau-1}} \frac{\gamma\omega}{2}k^{1-\varsigma\tau}
\geq\frac{(\gamma\omega)^2}{2^{\tau}}k^{2-\tau-\varsigma\tau}=\frac{(\gamma\omega)^2}{2^\tau} \quad  \text{ for} ~l\leq2k,
\end{equation}
where $\varsigma$ is taken as $(2-\tau)/\tau$ to ensure $2-\tau-\varsigma\tau=0$.

$\mathrm{(ii)}$ If $\max{\{k,l\}}=k<k_\star$, then
\begin{equation}\label{E7.7}
\varpi_l\varpi_k\stackrel{\eqref{E5.12}}{\geq}\frac{(\gamma\omega)^2}{(kl)^{\tau-1}}>\frac{(\gamma\omega)^2}{(k_\star)^{2(\tau-1)}}
=\gamma^2\omega^2\left(\frac{\gamma\omega}{6L_3}\right)^{\frac{1}{\tau-1}2(\tau-1)}\stackrel{\omega>\gamma}{>}\frac{\gamma^6\omega^2}{(6L_3)^2}.
\end{equation}
Formula \eqref{E5.16} is reached if we take the minimums of lower bounds in \eqref{E7.5}, \eqref{E7.6}-\eqref{E7.7}.
The next step is to consider $l,k\geq 1$ and the asymptotic formula \eqref{E6.38}. The only difference is that
\begin{align*}
\Big|\Big(\omega k-\sqrt{\lambda_{i^*}(\epsilon,w)}\Big)-\Big(\omega l-\sqrt{\lambda_{j^*}(\epsilon,w)}\Big)\Big|=&\left|\omega (k-l)-\left(\sqrt{\lambda_{i^*}(\epsilon,w)}-\sqrt{\lambda_{j^*}(\epsilon,w)}\right)\right|\\
\geq&\left|\omega(k-l)-\frac{i^*-j^*}{c}\right|-\left|\sqrt{\lambda_{i^*}(\epsilon,w)}-\frac{i^*+1/2}{c}\right|\\
&-\left|\sqrt{\lambda_{j^*}(\epsilon,w)}-\frac{j^*+1/2}{c}\right|.\\
\end{align*}
Since $\varpi_{l}=\varpi_{-l}$, the remainder of the lemma may be proved in the similar way as  above with $l\geq1,k\leq-1$, or $l\leq-1,k\geq1$, or $l,k\leq-1$.
\end{proof}

\begin{proof}[{\bf{The proof of}} {\bf(P5)}]
It follows from  \eqref{E5.34}-\eqref{E5.30},  \eqref{E5.17.2} and \eqref{E5.21} that
\begin{align*}
\|\mathcal{L}^{-1}_N(\epsilon,\omega,w)\|_{s'}&\leq\frac{\sqrt{2}\widehat{M}L_2N^{\frac{\tau-1}{2}}}{L_1\sqrt{\gamma\omega}}
\left\|\left(\mathrm{Id}-\mathcal{R}\right)^{-1}
\left(|\mathscr{L}_{1,{\mathrm{D}}}|^{-\frac{1}{2}}
\mathscr{L}_{1,{\mathrm{D}}}|\mathscr{L}_{1,{\mathrm{D}}}|^{-\frac{1}{2}}\right)^{-1}|\mathscr{L}_{1,{\mathrm{D}}}|^{-\frac{1}{2}}h\right\|_{s'}\\
&\leq\frac{4\widehat{M}L^3_2N^{{\tau-1}}}{{L^3_1\gamma\omega}}(
\|h\|_{s'}+\|w\|_{s'+\sigma}\|h\|_{s}),
\end{align*}
where $\widehat{M}=\|1/\rho\|_{L^{\infty}}$. By means of the fact $\|w\|_{s+\sigma}\leq1$, we get for $s'=s$
\begin{equation*}
\|\mathcal{L}^{-1}_N(\epsilon,\omega,w)\|_{s}\leq\frac{8\widehat{M}L^3_2N^{{\tau-1}}}{{L^3_1\gamma\omega}}\|h\|_{s}.
\end{equation*}
Thus the property $(P5) $ holds.
\end{proof}

\subsection{The Nash-Moser scheme}\label{sec:6.2}
Denote by $A_0$ the open set
\begin{equation}\label{E4.30}
A_{0}:=\Bigg\{(\epsilon,\omega)\in(\epsilon_1,\epsilon_2)\times(\omega_0,+\infty):~\left|\omega l-\sqrt{\lambda_{j}}\right|>\frac{\gamma}{l^{\tau}},~\forall l=1,\cdots,N_0, j\geq0\Bigg\},
\end{equation}
where $\lambda_j,j\geq0$ are the eigenvalues of the Sturm-Liouville problem
\begin{equation*}
\left\{
\begin{aligned}
&-(p(x)y'(x))'+m(x)y(x)=\lambda\rho(x)y(x),\\
&\alpha_1y(0)-\beta_1y_{x}(0)=0,\\
&\alpha_2y(\pi)+\beta_2y_{x}(\pi)=0.
\end{aligned}
\right.
\end{equation*}
In addition, define for $1<\chi\leq2$
\begin{equation}\label{E5.35}
\beta:=\chi(3\tau+\sigma)+\frac{\chi}{\chi-1}(\tau-1+\sigma)+1,
\end{equation}
where $\sigma=\frac{\tau(\tau-1)}{2-\tau}$, and denote  by
$ \mathfrak{B}(0,\rho)$ the open ball of center 0 and radius $\rho$.
\begin{lemm}(inductive scheme)\label{lemma4.1}
For all $n\in\mathbf{N}$, there exists a sequence of subsets $(\epsilon,\omega)\in A_{n}\subseteq A_{n-1}\subseteq\cdots\subseteq A_1\subseteq A_{0}$, where
\begin{equation*}
A_{n}:=\Bigg\{(\epsilon,\omega)\in A_{n-1} |~ (\epsilon,\omega)\in\Delta^{\gamma,\tau}_{N_{n}}(w_{n-1})\Bigg\},
\end{equation*}
and a sequence $w_{n}(\epsilon,\omega)\in W_{N_{n}}$ with
\begin{equation}\label{E4.10}
\|w_{n}\|_{s+\sigma}\leq1,
\end{equation}
 such that, if $(\epsilon,\omega)\in A_{n}$, for $\frac{\epsilon N^{\tau-1}_0}{\gamma\omega}\leq \mathfrak{c}$ small enough, then $w_{n}(\epsilon,\omega)$ is a solution of
\begin{equation*}
  L_{\omega}w-\epsilon P_{N_n}\Pi_{W}\mathcal{F}(\epsilon,w)=0.  \quad\quad (P_{N_{n}})
\end{equation*}
Furthermore $w_{n}=\sum^{n}_{i=0}h_i$, where $h_i\in W_{N_i}$ with
\begin{equation}\label{E4.11}
\|h_0\|_{s}\leq \frac{K_1\epsilon}{\gamma\omega},\quad\|h_i\|_{s}\leq\frac{ K_{2}\epsilon}{\gamma\omega}N^{-\sigma-3}_{i},\quad\forall1\leq i\leq n
\end{equation}
for $K_1>0,K_2>0$ depending on $\rho, p, m, f, \epsilon_0, \hat{v}, \gamma,\tau, s,\beta$ at most.
\end{lemm}
\begin{proof}
The lemma is verified by induction.

Step1: initialization. Let $N_{0}:=[e^{\mathfrak{d}}]$, where $\mathfrak{d}=\ln N_{0}$. By the  definition of $A_0$, it is easy to see that
\begin{equation*}
|\omega^2l^2-\lambda_j|=\left|\omega l-\sqrt{\lambda_j}\right|\left|\omega l+\sqrt{\lambda_j}\right|>\left|\omega l-\sqrt{\lambda_j}\right|\omega l>\frac{\gamma\omega}{l^{\tau-1}},\quad\forall l=1,2,\cdots,N_0,~\forall j\geq0,
\end{equation*}
where $\omega^2l^2-\lambda_j$ is the eigenvalue of $\frac{1}{\rho}L_{\omega}$. Then, for $\|w_0\|\leq1$, the operator $\frac{1}{\rho}L_{\omega}$ is invertible with
\begin{align}\label{E8.1}
\|w_{0}\|_{s}=\left\|\epsilon \left(({1}/{\rho})L_{\omega}\right)^{-1}\left(({1}/{\rho})P_{0}\Pi_{W}\mathcal{F}(\epsilon,w_0)\right)\right\|_{s}
\leq\frac{\epsilon \widehat{M}N^{\tau-1}_0}{\gamma\omega}\|P_0\Pi_{W}\mathcal{F}(\epsilon,w_0)\|_{s}\stackrel{(P1)}{\leq}\frac{K_1\epsilon}{\gamma\omega}.
\end{align}
Moreover, for $\frac{\epsilon N^{\tau-1}_0}{\gamma\omega}\leq \mathfrak{c}$ small enough, we also have
\begin{align}
&\|w_{0}\|_{s+\beta}\leq\frac{\epsilon \widehat{M}N^{\beta+\tau-1}_0}{\gamma\omega}\|P_0\Pi_{W}\mathcal{F}(\epsilon,w_0)\|_{s}\stackrel{(P1)}{\leq}\widehat{C},\label{E8.2}\\
&\left\|\epsilon \left( ({1}/{\rho})L_{\omega}\right)^{-1}\left(({1}/{\rho})P_{0}\Pi_{W}\mathrm{D}_{w}\mathcal{F}(\epsilon,w_0)\right)\right\|_{s}\leq\frac{\epsilon \widehat{M} N^{\tau-1}_0}{\gamma\omega}\|P_{0}\Pi_{W}\mathrm{D}_{w}\mathcal{F}(\epsilon,w_0)\|_{s}\stackrel{(P1)}{\leq}\frac{1}{2}.\label{E8.3}
\end{align}
Formulae \eqref{E8.1} and  \eqref{E8.3} establish that the map $w_0\mapsto\epsilon(\frac{1}{\rho}L_{\omega})^{-1}(\frac{1}{\rho}P_{0}\Pi_{W}\mathcal{F}(\epsilon,w_0))$  is a contraction in $\{w_0\in W_{N_{0}}:\|w_{0}\|_s\leq \rho_0\}$ for $\frac{\epsilon N^{\tau-1}_0}{\gamma\omega}\leq \mathfrak{c}$ small enough,  where $\rho_0=\frac{K_1\epsilon}{\gamma\omega}$.

Step 2: assumption. Assume that we have obtained a solution $w_{n}\in W_{N_{n}}$ of $(P_{N_{n}})$ satisfying conditions \eqref{E4.10}-\eqref{E4.11}.

Step 3: iteration. Our goal is to find a solution $w_{n+1}\in W_{N_{n+1}}$ of $(P_{N_{n+1}})$ with conditions \eqref{E4.10}-\eqref{E4.11} at ($n+1$)-th step. For $h_{n+1}\in W_{N_{n+1}}$, denote by
\begin{equation*}
w_{n+1}=w_{n}+h_{n+1}
\end{equation*}
a solution of
\begin{equation*}
L_{\omega}w-\epsilon P_{N_{n+1}}\Pi_{W}\mathcal{F}(\epsilon,w)=0. \quad\quad (P_{N_{n+1}})
\end{equation*}
From the fact 
$L_{\omega}w_{n}=\epsilon P_{N_{n}}\Pi_{W}\mathcal{F}(\epsilon,w_{n})$,
it yields that
\begin{align*}
L_{\omega}(w_{n}+h)-\epsilon P_{N_{n+1}}\Pi_{W}\mathcal{F}(\epsilon,w_{n}+h)=&L_{\omega}h+ L_{\omega}w_{n}-\epsilon P_{N_{n+1}}\Pi_{W}\mathcal{F}(\epsilon,w_{n}+h)\\
=&-\mathcal{L}_{N_{n+1}}(\epsilon,\omega,w_n)h+R_{n}(h)+r_{n},
\end{align*}
where
\begin{align*}
&R_{n}(h):=\epsilon P_{N_{n+1}}(-\Pi_{W}\mathcal{F}(\epsilon,w_n+h)
+\Pi_{W}\mathcal{F}(\epsilon,w_n)+\Pi_{W}\mathrm{D}_{w}\mathcal{F}(\epsilon,w_n)[h]),\\
&r_{n}:=\epsilon(P_{N_{n}}\Pi_{W}\mathcal{F}(\epsilon,w_n)-P_{N_{n+1}}\Pi_{W}\mathcal{F}(\epsilon,w_n))=
-\epsilon P^{\bot}_{N_{n}}P_{N_{n+1}}\Pi_{W}\mathcal{F}(\epsilon,w_n).
\end{align*}
By means of \eqref{E4.7}-\eqref{E4.8}, \eqref{E4.10} and $(\epsilon,\omega)\in A_{n+1}\subseteq A_{n}$, the linear operator $\mathcal{L}_{N_{n+1}}(\epsilon,\omega,w_n)$ is invertible with
\begin{equation}\label{E4.12}
\left\|\mathcal{L}^{-1}_{N_{n+1}}(\epsilon,\omega,w_{n})h\right\|_{s'}\leq \frac{K(s')}{\gamma\omega}N_{n+1}^{\tau-1}\left(\|h\|_{s'}+\|w\|_{s'+\sigma}\|h\|_{s}\right),\quad \forall s'\geq s,
\end{equation}
and
\begin{equation}\label{E4.13}
\left\|\mathcal{L}^{-1}_{N_{n+1}}(\epsilon,\omega,w_{n})h\right\|_{s}\leq \frac{K}{\gamma\omega}N_{n+1}^{\tau-1}\|h\|_{s}.
\end{equation}
Define a map
\begin{align*}
\mathcal{G}_{n+1}: &W_{N_{n+1}}\rightarrow W_{N_{n+1}}\\
&h\mapsto-\mathcal{L}_{N_{n+1}}(\epsilon,\omega,w_n)^{-1}(R_{n}(h)+r_{n}).
\end{align*}
Then solving $(P_{N_{n+1}})$ is reduced to find  the fixed point of  $h=\mathcal{G}_{n+1}(h)$. Let us show that there exists $K_2>0$ such that, for $(\epsilon,\omega)\in A_{n+1}$ and $\frac{\epsilon}{\gamma\omega}$ small enough, $\mathcal{G}_{n+1}$ is a contraction in
\[\mathfrak{B}_{n+1}(0,\rho_{n+1})=\{h~|~h\in W_{N_{n+1}}~\text{with}~\|h\|_{s}\leq \rho_{n+1}\},\] where
\begin{align}\label{E4.33}
\rho_{n+1}:=\frac{ K_2\epsilon}{\gamma\omega}N^{-\sigma-3}_{n+1}\quad\text {with}~\sigma=\frac{\tau(\tau-1)}{2-\tau}.
\end{align}
 In fact, the properties $(P2)$-$(P4)$ imply
\begin{equation*}
\|R_{n}(h)\|_{s}{\leq}\epsilon C\|h\|^2_{s},\quad
\|r_{n}\|_{s}{\leq}\epsilon C(\beta)N^{-\beta}_{n}S_{n},
\end{equation*}
where
\begin{align}\label{E4.34}
S_{n}:=1+\|w_n\|_{s+\beta}.
\end{align}
If $\frac{\epsilon}{\gamma\omega}$ is small enough, then we claim that the following
\begin{equation}\label{E4.35}
{\bf ( F3):}\quad S_{n}\leq C'S_{0}N^{\frac{1}{\chi-1}{(\tau-1+\sigma)}}_{n+1}
\end{equation}
holds for some constant $C':=C'(\chi,\mathfrak{d},\tau,\sigma)>0$. While the proof of  {\bf (F3)} will be given in the next lemma.

 Combining \eqref{E4.35} with definition \eqref{E5.35} on $\beta$ and formulae \eqref{E8.2}, \eqref{E4.13}-\eqref{E4.33}, for $\frac{\epsilon}{\gamma\omega}$ small enough, we have
\begin{align*}
\|\mathcal{G}_{n+1}(h)\|_{s}&{\leq}\frac{K}{\gamma\omega}N^{\tau-1}_{n+1}(\|R_{n}(h)\|_{s}+\|r_{n}\|_{s})\leq\frac{\epsilon K'}{\gamma\omega}N^{\tau-1}_{n+1}N^{-\beta}_{n}S_{n}+\frac{\epsilon K'}{\gamma\omega}N^{\tau-1}_{n+1}\rho^2_{n+1}\leq\rho_{n+1}.
\end{align*}
Moreover
\begin{equation}\label{E4.15}
\mathrm{D}_{h}\mathcal{G}_{n+1}(h)[\hat{w}]=\epsilon\mathcal{L}^{-1}_{N_{n+1}}(\epsilon,\omega,w_n)P_{N_{n+1}}(\Pi_{W}\mathrm{D}_{w}
\mathcal{F}(\epsilon,w_n+h)-\Pi_W\mathrm{D}_{w}\mathcal{F}(\epsilon,w_n))\hat{w}.
\end{equation}
Using definition \eqref{E5.35} on $\beta$, for $\frac{\epsilon}{\gamma\omega}$ small enough, we derive
\begin{equation}\label{E4.16}
\|\mathrm{D}_h\mathcal{G}_{n+1}(h)[\hat{w}]\|_{s}\stackrel{(P1),\eqref{E4.13}}{\leq}\frac{K''\epsilon}{\gamma\omega}N^{\tau-1}_{n+1}
\|h\|_{s}\|\hat{w}\|_{s}\leq\frac{K''\epsilon}{\gamma\omega}N^{\tau-1}_{n+1}\rho_{n+1}\|\hat{w}\|_{s}\leq\frac{1}{2}\|\hat{w}\|_{s}.
\end{equation}
Hence the map $\mathcal{G}_{n+1}$ is a  contraction, which gives rise to $h_{n+1}\in W_{N_{n+1}}$. In addition, for $\frac{\epsilon}{\gamma\omega}$ small enough, the following holds:
\begin{equation*}
\|w_{n+1}\|_{s+\sigma}\leq\sum\limits_{i=0}^{n+1}\|h_{i}\|_{s+\sigma}\stackrel{(P4)}{\leq}
\sum\limits_{i=0}^{n+1}N^{\sigma}_{i}\|h_{i}\|_{s}\leq\sum\limits_{i=0}^{n+1}N^{\sigma}_{i}K_2\epsilon(\gamma\omega)^{-1}N^{-\sigma-3}_{i}\leq1.
\end{equation*}
This completes the proof.
\end{proof}
\begin{rema}
In Lemma \ref{lemma4.1}, we construct the function $h_{n}$ depending on the parameters $(\epsilon,\omega)$.
\end{rema}

Let us give the proof of the fact {\bf (F3)}.
\begin{lemm}\label{lemma4.3}
Let $S_n$ be given by \eqref{E4.34}. Given conditions
\eqref{E4.10}-\eqref{E4.11} for all $i\leq n$, there exist $C':=C'(\chi,\mathfrak{d},\tau,\sigma)$ such that for $\frac{\epsilon}{\gamma\omega}$ small enough
\begin{equation*}
S_{n}\leq C'S_{0}N^{\frac{1}{\chi-1}{(\tau-1+\sigma)}}_{n+1}.
\end{equation*}

\end{lemm}
\begin{proof}
Firstly, let us  claim
\begin{align}\label{E4.36}
S_{n}\leq (1+N^{\tau-1+\sigma}_{n})S_{n-1}.
\end{align}
In fact, it follows from $(P2)$-$(P3)$, the definitions of $R_{n-1}(h),r_{n-1},S_{n}$ and the fact $\|h_n\|_{s}\leq\rho_{n}$ that
\begin{align*}
&\|R_{n-1}(h_{n})\|_{s}\leq\epsilon C,\quad\|r_{n-1}\|_{s}\leq\epsilon C,\quad\|r_{n-1}\|_{s+\beta}\leq \epsilon C(\beta)(1+\|w_{n-1}\|_{s+\beta})=\epsilon C(\beta)S_{n-1},\\
&\|R_{n-1}(h_{n})\|_{s+\beta}\leq \epsilon C(\beta)(\|w_{n-1}\|_{s+\beta}\|h_n\|^2_s+\|h_n\|_s\|h_n\|_{s+\beta})\leq \epsilon C(\beta)(S_{n-1}\rho^2_{n}+\rho_{n}\|h_n\|_{s+\beta}).
\end{align*}
In view of the equality $h_{n}=\mathcal{L}^{-1}_{N_{n}}(\epsilon,\omega,w_{n-1})(R_{n-1}(h_{n})+r_{n-1})$, formula \eqref{E4.12}, definition \eqref{E4.33} on $\rho$ and the definition of $\sigma$, for $ \frac{\epsilon}{\gamma\omega}$ small enough, it yields that
\begin{align*}
\|h_{n}\|_{s+\beta}{\leq}&\frac{K(\beta)}{\gamma\omega}N^{\tau-1}_{n}
\left(\|R_{n-1}(h_{n})\|_{s+\beta}+\|r_{n-1}\|_{s+\beta}+\|w_{n-1}\|_{s+\beta+\sigma}(\|R_{n-1}(h_{n})\|_{s}+\|r_{n-1}\|_{s})\right)\\
\leq&\frac{\epsilon K'(\beta)}{2\gamma\omega}N^{\tau-1}_{n}(S_{n-1}\rho^{2}_{n}+\rho_{n}\|h_{n}\|_{s+\beta}+S_{n-1}+N^{\sigma}_{n-1}S_{n-1})\\
\leq&\frac{\epsilon K'(\beta)}{\gamma\omega}N^{\tau-1+\sigma}_{n}S_{n-1}+\frac{\epsilon K'(\beta)}{2\gamma\omega}\|h_{n}\|_{s+\beta}\\
\leq&\frac{\epsilon K'(\beta)}{\gamma\omega}N^{\tau-1+\sigma}_{n}S_{n-1}+\frac{1}{2}\|h_{n}\|_{s+\beta},
\end{align*}
which leads to
\begin{equation}\label{E4.17}
\|h_{n}\|_{s+\beta}\leq\frac{2\epsilon K'(\beta)}{\gamma\omega}N^{\tau-1+\sigma}_{n}S_{n-1}\leq N^{\tau-1+\sigma}_{n}S_{n-1}.
\end{equation}
This indicates
\begin{equation*}
S_{n}=1+\|w_n\|_{s+\beta}\leq1+\|w_{n-1}\|_{s+\beta}+\|h_{n}\|_{s+\beta}\leq(1+N^{\tau-1+\sigma}_{n})S_{n-1}.
\end{equation*}
Consequently, with the help of the inequality: $N_{n+1}\leq e^{\mathfrak{d}\chi^{n+1}}<N_{n+1}+1<2N_{n+1}$, the definition of $N_n$ and the claim \eqref{E4.36}, we have
\begin{align*}
S_{n}{\leq}&S_0\prod^{n}_{i=1}(1+N^{\tau-1+\sigma})\leq S_0\prod^{n}_{i=1}(1+e^{\mathfrak{d}\chi^n(\tau-1+\sigma)}) \\
\leq& C''S_0e^{\mathfrak{d}\frac{\chi^{n+1}}{\chi-1}(\tau-1+\sigma)}\\
\leq& C'S_0N_{n+1}^{\frac{1}{\chi-1}(\tau-1+\sigma)},
\end{align*}
where $C'=C''2^{\frac{1}{\chi-1}(\tau-1+\sigma)}=2^{\frac{1}{\chi-1}
(\tau-1+\sigma)}\prod^{n}_{i=1}(1+e^{-\mathfrak{d}\chi^n(\tau-1+\sigma)})$.
\end{proof}

To give the measure estimates on $B_\gamma$ defined by \eqref{E4.31}, the estimates on the derivatives of $h_n$ with respect to $(\epsilon,\omega)$ have to be required.
\begin{lemm}\label{lemma4.5}
For $\frac{\epsilon}{\gamma^2\omega}$ small enough, the map $h_{i},i\in\mathbf{N}$ belong to $C^1(A_{i};W_{N_{i}})$ with
\begin{align*}
&\|\partial_{\omega}h_{0}\|_{s}\leq\frac{ K_{3}\epsilon}{\gamma^2\omega},\quad
\|\partial_{\omega}h_{i}\|_{s}\leq\frac{ K_{5}\epsilon}{\gamma^2\omega}N_{i}^{-1},\quad\forall i\geq1,\\
&\|\partial_{\epsilon}h_{0}\|_{s}\leq\frac{ K_{4}}{\gamma\omega},\quad\quad
\|\partial_{\epsilon}h_{i}\|_{s}\leq\frac{ K_{6}}{\gamma\omega}N_{i}^{-1},\quad\forall i\geq1,
\end{align*}
where $K_i>0,i=3,4,5,6$ depend on $\rho, p, m, f, \epsilon_0, \hat{v}, \gamma,\tau, s,\beta$ at most.
\end{lemm}
\begin{proof}
This lemma is verified by induction. For $w_{0}=h_{0}$, define
\begin{equation*}
U_{0}(\epsilon,\omega,h):=h-\epsilon L^{-1}_{\omega}P_0\Pi_{W}\mathcal{F}(\epsilon,h), \quad
\mathcal{G}_{0}(h_0):=-\epsilon L^{-1}_{\omega}P_{0}\Pi_{W}\mathcal{F}(\epsilon,h_0).
\end{equation*}
The definition of $A_0$ indicates for $\frac{\epsilon}{\gamma\omega}$ small enough
\begin{equation*}
\|\mathrm{D}_{h}\mathcal{G}_{0}(h_0)\|_{s}=\|-\epsilon L^{-1}_{\omega}P_{0}\Pi_{W}\mathrm{D}_w\mathcal{F}(\epsilon,h_0)\|_{s}\leq\frac{\epsilon \widehat{M} N^{\tau-1}_0}{\gamma\omega}\|P_{0}\Pi_{W}\mathcal{F}(\epsilon,h_0)\|_{s}\stackrel{(P1)}{\leq}\frac{1}{2}.
\end{equation*}
This implies that
$\mathrm{D}_hU_{0}(\epsilon,\omega,h_0)=\mathrm{Id}-\mathrm{D}_{h}\mathcal{G}_{0}(h_0)$
is invertible. Clearly, it can be seen that $h_0\in C^1(A_0;W_{N_0})$ by the implicit function theorem.
Taking the derivative of the identity $L_{\omega}(L^{-1}_{\omega}h)=h$ with respect to $\omega$ yields
\begin{equation*}
\partial_{\omega}L^{-1}_{\omega}h=-(L^{-1}_{\omega})(2\omega\rho\partial_{tt})L^{-1}_{\omega}h.
\end{equation*}
Then, due to the definition of $A_0$, we get
\begin{align}\label{E8.4}
\|\partial_{\omega}L^{-1}_{\omega}h_0\|_{s}\leq\frac{\widehat{K}}{\gamma^2\omega}\|h_0\|_{s}.
\end{align}
It follows from taking the derivative of  $h_{0}=\epsilon L^{-1}_{\omega}P_0\Pi_{W}\mathcal{F}(\epsilon,h_0)$ with respect to $\omega$ that
\begin{equation*}
\partial_{\omega}h_{0}=\epsilon(\mathrm{Id}-\mathrm{D}_{h}\mathcal{G}_{0}(h_0))^{-1}\partial_{\omega}L^{-1}_{\omega}P_0\Pi_{W}\mathcal{F}(\epsilon,h_0),
\end{equation*}
which carries out
$\|\partial_{\omega}h_0\|_{s}\leq\frac{K_3\epsilon}{\gamma^2\omega}$ in view of \eqref{E8.4}. Furthermore we have
\begin{equation*}
\partial_{\epsilon}h_{0}=(\mathrm{Id}-\mathrm{D}_{h}\mathcal{G}_{0}(h_0))^{-1}L^{-1}_{\omega}P_0\Pi_{W}\mathcal{F}(\epsilon,h_0)
\end{equation*}
by taking derivative of  $h_{0}=\epsilon L^{-1}_{\omega}P_0\Pi_{W}\mathcal{F}(\epsilon,h_0)$ with respect to  $\epsilon$.
A similar process yields $\|\partial_{\epsilon}h_0\|_{s}\leq\frac{K_4}{\gamma\omega}$.

Assume $h_{i}\in C^1(A_i;W_{N_i})$ with $\|\partial_{\omega}h_{i}\|_{s}\leq\frac{ K_{5}\epsilon}{\gamma^2\omega}N_{i}^{-1}$, $\|\partial_{\epsilon}h_{i}\|_{s}\leq\frac{ K_{6}\epsilon}{\gamma\omega}N_{i}^{-1}$,
 $\forall1\leq i\leq n$.  By  assumption, it is easy to see that
\begin{align}\label{E4.39}
\|\partial_{\omega}w_{n}\|_{s}\leq\frac{\widehat{K}'\epsilon}{\gamma^2\omega},\quad \|\partial_{\epsilon}w_{n}\|_{s}\leq\frac{\widehat{K}'}{\gamma\omega}.
\end{align}
Moreover denote
\begin{align}\label{E4.37}
S'_{n}:=1+\|\partial_{\omega}{w_{n}}\|_{s+\beta},\quad S''_{n}:=1+\|\partial_{\epsilon}{w_{n}}\|_{s+\beta}.
\end{align}
We also claim that\\
{\bf (F4):} For $\frac{\epsilon}{\gamma\omega}$ small enough, the following inequalities
\begin{align}\label{E4.38}
S'_{n}\leq C_1'\frac{S_{0}}{\gamma}N^{2\tau+\sigma+{\frac{1}{\chi-1}{(\tau-1+\sigma)}}}_{n+1},\quad S''_{n}\leq C_2'\frac{S_{0}}{\gamma\omega}N^{2\tau+\sigma+{\frac{1}{\chi-1}{(\tau-1+\sigma)}}}_{n+1}
\end{align}
hold for some constant $C_i':=C_i'(\chi,\mathfrak{d},\tau,\sigma),i=1,2$. While the proof of {\bf (F4)}  will be given in  Lemma \ref{lemma4.8}.

Let us verify the results of the lemma for $i=n+1$. Set
\begin{equation}\label{E4.18}
U_{n+1}(\epsilon,\omega,h):=-L_{\omega}(w_{n}+h)+\epsilon  P_{N_{n+1}}\Pi_{W}\mathcal{F}(\epsilon,w_n+h).
\end{equation}
Since $h_{n+1}(\epsilon,\omega)$ is a solution of \eqref{E4.18}, it is straightforward to give
\begin{equation}\label{E4.4}
U_{n+1}(\epsilon,\omega,h_{n+1})=0.
\end{equation}
Formula \eqref{E4.15} indicates 
\begin{align}\label{E8.5}
\mathrm{D}_{h}U_{n+1}(\epsilon,\omega,h_{n+1}){=}\mathcal{L}_{N_{n+1}}(\epsilon,\omega,w_{n+1})
{=}\mathcal{L}_{N_{n+1}}(\epsilon,\omega,w_{n})(\mathrm{Id}-\mathrm{D}_{h}\mathcal{G}(h_{n+1})).
\end{align}
Then estimate \eqref{E4.16} shows that the operator $\mathcal{L}_{N_{n+1}}(\epsilon,\omega,w_{n+1})$ is invertible with
\begin{equation}\label{E4.19}
\|\mathcal{L}^{-1}_{N_{n+1}}(\epsilon,\omega,w_{n+1})\|_{s}\leq\|(\mathrm{Id}-\mathrm{D}_{h}\mathcal{G}(h_{n+1}))^{-1}\mathcal{L}^{-1}_{N_{n+1}}(\epsilon,\omega,w_{n})\|_{s}
\leq\frac{2K}{\gamma\omega}N_{n+1}^{\tau-1}.
\end{equation}
The implicit function theorem establishes $h_{n+1}\in C^1(A_{n+1};W_{N_{n+1}})$, which then infers
\begin{equation*}
\partial_{\omega,\epsilon}U_{n+1}(\epsilon,\omega,h_{n+1})+\mathrm{D}_{h}U_{n+1}(\epsilon,\omega,h_{n+1})\partial_{\omega,\epsilon}h_{n+1}=0
\end{equation*}
by \eqref{E4.4}. Consequently, using $w_{n+1}=w_n+h_{n+1}$, we obtain
\begin{equation}\label{E4.20}
\partial_{\omega,\epsilon}h_{n+1}=-\mathcal{L}^{-1}_{N_{n+1}}(\epsilon,\omega,w_{n+1})\partial_{\omega,\epsilon}U_{n+1}(\epsilon,\omega,h_{n+1}),
\end{equation}
where
\begin{align*}
\partial_{\omega}U_{n+1}(\epsilon,\omega,h_{n+1})=&2\omega\rho(x)(h_{n+1})_{tt}-\epsilon P^{\bot}_{N_{n}}P_{N_{n+1}}\Pi_{W}\mathrm{D}_{w}\mathcal{F}(\epsilon,w_n)\partial_{\omega}w_{n}\\
&-\epsilon P_{N_{n+1}}(\Pi_{W}\mathrm{D}_{w}\mathcal{F}(\epsilon,w_{n+1})-\Pi_{W}\mathrm{D}_{w}\mathcal{F}(\epsilon,w_n))\partial_{\omega}w_{n},
\end{align*}
\begin{align*}
\partial_{\epsilon}U_{n+1}(\epsilon,\omega,h_{n+1})=&-P^{\bot}_{N_{n}}P_{N_{n+1}}\Pi_{W}\mathcal{F}(\epsilon,w_n)-P_{N_{n+1}}
(\Pi_{W}\mathcal{F}(\epsilon,w_{n+1})-\Pi_{W}\mathcal{F}(\epsilon,w_n))\\
&-\epsilon P^{\bot}_{N_{n}}P_{N_{n+1}}\Pi_{W}\partial_{\epsilon}\mathcal{F}(\epsilon,w_n)-\epsilon P_{N_{n+1}}
(\Pi_{W}\partial_{\epsilon}\mathcal{F}(\epsilon,w_{n+1})-\Pi_{W}\partial_{\epsilon}\mathcal{F}(\epsilon,w_n))\\
&-\epsilon P^{\bot}_{N_{n}}P_{N_{n+1}}\Pi_{W}\mathrm{D}_{w}\mathcal{F}(\epsilon,w_n)\partial_{\epsilon}w_{n}\\
&-\epsilon P_{N_{n+1}}(\Pi_{W}\mathrm{D}_{w}\mathcal{F}(\epsilon,w_{n+1})-\Pi_{W}\mathrm{D}_{w}\mathcal{F}(\epsilon,w_n))\partial_{\epsilon}w_{n}.
\end{align*}
Furthermore $(P1)$-$(P2)$ and Remark \ref{Remark2} imply
\begin{align*}
&\|\Pi_{W}\partial_{\epsilon}\mathcal{F}(\epsilon,w_n)\|_{s}\leq C(1+\|\partial_{\epsilon} w_n\|_{s}),\\
&\|\Pi_{W}\partial_{\epsilon}\mathcal{F}(\epsilon,w_n)\|_{s+\beta}\leq C(\beta)\|w_n\|_{s+\beta}(1+\|\partial_{\epsilon }w_n\|_{s})+C(\beta)(1+\|\partial_{\epsilon} w_n\|_{s+\beta}),\\
&\|\Pi_{W}\partial_{\epsilon}\mathcal{F}(\epsilon,w_{n+1})-\Pi_{W}\partial_{\epsilon}\mathcal{F}(\epsilon,w_n)\|_{s}\leq C(1+\|\partial_{\epsilon} w_n\|_{s})\|h_{n+1}\|_s
\end{align*}
and
\begin{align*}
\|\Pi_{W}\partial_{\epsilon}\mathcal{F}(\epsilon,w_{n+1})-\Pi_{W}\partial_{\epsilon}\mathcal{F}(\epsilon,w_n)\|_{s+\beta}\leq& C(\beta)(\|w_n\|_{s+\beta}+\|h_{n+1}\|_{s+\beta})(1+\|\partial_{\epsilon} w_n\|_{s})
\\&+C(\beta)(1+\|\partial_{\epsilon} w_n\|_{s+\beta}).
\end{align*}
By means of $(P1)$, $(P4)$, $\eqref{E4.19}$ and $\eqref{E2.1}$, some simple calculation leads to
\begin{align*}
\|\partial_{\omega}h_{n+1}\|_s\leq&\frac{\widetilde{K}}{\gamma}N^{\tau+1}_{n+1}\|h_{n+1}\|_s+\frac{\epsilon \widetilde{K}}{\gamma\omega}{N^{\tau-1}_{n+1}}{N^{-\beta}_{n}}(\|\partial_{\omega}{w_{n}}\|_{s+\beta}
+\|w_{n}\|_{s+\beta}\|\partial_{\omega}w_{n}\|_s)\\
&+\frac{\epsilon \widetilde{K}}{\gamma\omega}N^{\tau-1}_{n+1}\rho_{n+1}\|\partial_{\omega}w_{n}\|_{s}\\
\leq&\frac{\widetilde{K}}{\gamma}N^{\tau+1}_{n+1}\|h_{n+1}\|_s+\frac{\epsilon \widetilde{K}}{\gamma\omega}{N^{\tau-1}_{n+1}}{N^{-\beta}_{n}}(S'_{n}+S_{n}\|\partial_{\omega}w_n\|_{s})\\
&+\frac{\epsilon \widetilde{K}}{\gamma\omega}N^{\tau-1}_{n+1}\rho_{n+1}\|\partial_{\omega}w_{n}\|_{s},
\end{align*}
\begin{align*}
\|\partial_{\epsilon}h_{n+1}\|_s\leq&\frac{\widetilde{K}}{\gamma\omega}N^{\tau-1}_{n+1}N^{-\beta}_{n}S_n
+\frac{\widetilde{K}}{\gamma\omega}\|h_{n+1}\|_{s}+\frac{\epsilon \widetilde{K}}{\gamma\omega}N^{\tau-1}_{n+1}N^{-\beta}_{n}(S_n(1+\|\partial_{\epsilon}w_n\|_s)+S''_{n})\\
&+\frac{\epsilon \widetilde{K}}{\gamma\omega}N^{\tau-1}_{n+1}(1+\|\partial_{\epsilon}w_n\|_s)\|h_{n+1}\|_{s},
\end{align*}
where $S_n$ is given by \eqref{E4.34}, $S'_n,S''_n$ are given by \eqref{E4.37}.
For $\frac{\epsilon}{\gamma^2\omega}$ small enough, applying \eqref{E8.2}, \eqref{E4.35}, \eqref{E4.39}, \eqref{E4.38}, and the definitions of $\rho_{n+1}$ (see \eqref{E4.33}), $\beta$ (see \eqref{E5.35}), we  can obtain
\begin{equation*}
\|\partial_{\omega} h_{n+1}\|_{s}\leq\frac{ K_{5}\epsilon}{\gamma^2\omega}N_{n+1}^{-1},\quad\|\partial_{\epsilon} h_{n+1}\|_{s}\leq\frac{ K_{6}}{\gamma\omega}N_{n+1}^{-1}.
\end{equation*}
This completes the  proof of Lemma \ref{lemma4.5}.
\end{proof}

\begin{rema}\label{remark8}
Lemma \ref{lemma4.5} implies that, for all $\gamma\in(0,1)$, $\|\partial_{\omega}w_n\|_s\leq\frac{\epsilon E}{\gamma^2\omega}$ and $\|\partial_{\epsilon}w_n\|_s\leq\frac{E}{\gamma^2\omega}$ for some $E>0$.
\end{rema}

To prove {\bf (F4)}, we first  have to estimate the upper bound of $\|\mathcal{L}^{-1}_{N_{n+1}}(\epsilon,\omega,w_{n+1})\hat{w}\|_{s+\beta}$.
\begin{lemm}\label{lemma4.6}
 For $\frac{\epsilon}{\gamma\omega}$ small enough,  one  has
\begin{equation*}
\|\mathcal{L}^{-1}_{N_{n+1}}(\epsilon,\omega,w_{n+1})\hat{w}\|_{s+\beta}\leq\frac{ \widehat{E}}{\gamma\omega}N^{\tau-1}_{n+1}\|\hat{w}\|_{s+\beta}
+\frac{ \widehat{E}}{\gamma\omega}N^{2(\tau-1)}_{n+1}(\|w_{n}\|_{s+\beta+\sigma}+\|h_{n+1}\|_{s+\beta})\|\hat{w}\|_{s},
\end{equation*}
where  $\beta$ is  given by \eqref{E5.35}.
\end{lemm}
\begin{proof}
Denote $\mathcal{A}(h_{n+1}):=(\mathrm{Id}-\mathrm{D}_{h}\mathcal{G}(h_{n+1}))^{-1}\hat{w}$. It is obvious  that
\begin{align}\label{E8.6}
\mathcal{A}(h_{n+1})=\hat{w}+\mathrm{D}_{h}\mathcal{G}(h_{n+1})\mathcal{A}(h_{n+1}).
\end{align}
According to$(P2)$, \eqref{E4.12}-\eqref{E4.33}, \eqref{E4.15} and $\|h_{n+1}\|_s\leq\rho_{n+1}$, it leads to
\begin{align*}
\|\mathrm{D}_{h}\mathcal{G}(h_{n+1})\|_{s+\beta}&\leq\frac{\epsilon E'(\beta)}{2\gamma\omega}N^{\tau-1}_{n+1}(\|w_{n}\|_{s+\beta}\|h_{n+1}\|_{s}+\|h_{n+1}\|_{s+\beta}+\|w_{n}\|_{s+\beta+\sigma}\|h_{n+1}\|_{s})\\
&\leq\frac{\epsilon E'(\beta)}{\gamma\omega}N^{\tau-1}_{n+1}(\|w_{n}\|_{s+\beta+\sigma}\|h_{n+1}\|_{s}+\|h_{n+1}\|_{s+\beta})\\
&\leq\frac{\epsilon E'(\beta)}{\gamma\omega}N^{\tau-1}_{n+1}(\|w_{n}\|_{s+\beta}+\|h_{n+1}\|_{s+\beta}).
\end{align*}
Combining above inequality  with \eqref{E4.16}, \eqref{E8.6} and \eqref{E2.1} yields
\begin{align*}
\|\mathcal{A}(h_{n+1})\|_{s+\beta}\leq&\|\hat{w}\|_{s+\beta}+\|\mathrm{D}_{h}\mathcal{G}(h_{n+1})\mathcal{A}(h_{n+1})\|_{s+\beta}\\
{\leq}&\|\hat{w}\|_{s+\beta}+C(\beta)\|\mathrm{D}_{h}\mathcal{G}(h_{n+1})\|_{s+\beta}\|\mathcal{A}(h_{n+1})\|_{s}+C(\beta)\|\mathrm{D}_{h}\mathcal{G}(h_{n+1})\|_{s}\|\mathcal{A}(h_{n+1})\|_{s+\beta}\\
{\leq}&\|\hat{w}\|_{s+\beta}+\frac{\epsilon E''}{\gamma\omega}N^{\tau-1}_{n+1}(\|w_{n}\|_{s+\beta}+\|h_{n+1}\|_{s+\beta})\|\hat{w}\|_{s}+\frac{\epsilon E''}{\gamma\omega}N^{\tau-1}_{n+1}\rho_{n+1}\|\mathcal{A}(h_{n+1})\|_{s+\beta}.
\end{align*}
Then, for $\frac{\epsilon}{\gamma\omega}$ small enough, it shows that
\begin{equation*}
\|\mathcal{A}(h_{n+1})\|_{s+\beta}\leq\|\hat{w}\|_{s+\beta}+\frac{\epsilon E''}{\gamma\omega}N^{\tau-1}_{n+1}(\|w_{n}\|_{s+\beta}+\|h_{n+1}\|_{s+\beta})\|\hat{w}\|_{s}+\frac{1}{2}\|\mathcal{A}(h_{n+1})\|_{s+\beta},
\end{equation*}
which implies
\begin{equation*}
\|(\mathrm{Id}-\mathrm{D}_{h}\mathcal{G}(h_{n+1}))^{-1}\hat{w}\|_{s+\beta}\leq2\|\hat{w}\|_{s+\beta}+\frac{2\epsilon E''}{\gamma\omega}N^{\tau-1}_{n+1}(\|w_{n}\|_{s+\beta}+\|h_{n+1}\|_{s+\beta})\|\hat{w}\|_{s}.
\end{equation*}
Hence, for $\frac{\epsilon}{\gamma\omega}$ small enough, it follows from \eqref{E8.5} and \eqref{E4.12}-\eqref{E4.13} that
\begin{align*}
\|\mathcal{L}^{-1}_{N_{n+1}}(\epsilon,\omega,w_{n+1})\hat{w}\|_{s+\beta}{\leq}&
\frac{ \widehat{E}}{\gamma\omega}N^{\tau-1}_{n+1}\|\hat{w}\|_{s+\beta}
+\frac{ \widehat{E}}{\gamma\omega}N^{2(\tau-1)}_{n+1}(\|w_{n}\|_{s+\beta+\sigma}+\|h_{n+1}\|_{s+\beta})\|\hat{w}\|_{s}.
\end{align*}
\end{proof}
Now, let us verify that the claim {\bf (F4)}.
\begin{lemm}\label{lemma4.8}
Supposed  $\|\partial_{\omega}h_{i}\|_{s}\leq\frac{ K_{5}\epsilon}{\gamma^2\omega}N_{i}^{-1}$, $\|\partial_{\epsilon}h_{i}\|_{s}\leq\frac{ K_{6}}{\gamma\omega}N_{i}^{-1}$, $1\leq i\leq n-1$,
 for $\frac{\epsilon}{\gamma\omega}$ small enough, we have
\begin{equation*}
S'_{n}\leq C'_1\frac{S_{0}}{\gamma}N^{2\tau+\sigma+{\frac{1}{\chi-1}{(\tau-1+\sigma)}}}_{n+1},\quad S''_{n}\leq C'_2\frac{S_{0}}{\gamma\omega}N^{2\tau+\sigma+{\frac{1}{\chi-1}{(\tau-1+\sigma)}}}_{n+1}
\end{equation*}
with  the  constants $C'_i:=C'_i(\chi,\mathfrak{d},\tau,\sigma)>0,i=1,2$.
\end{lemm}
\begin{proof}
First, let us check that for $\frac{\epsilon}{\gamma\omega}$ small enough, there exists  some constant $E_1$ such that
\begin{align}
&S'_{n}\leq(1+\frac{\epsilon E_1}{\gamma\omega}N^{\tau-1}_{n})S'_{n-1}+\frac{E_1}{\gamma}N^{2\tau+\sigma}_{n}S_{n-1},\label{E4.40}\\
&S''_{n}\leq(1+\frac{\epsilon E_1}{\gamma\omega}N^{\tau-1}_{n})S''_{n-1}+\frac{E_1}{\gamma\omega}N^{2\tau+\sigma}_{n}S_{n-1}.\nonumber
\end{align}
In fact, it is obvious  that
\begin{equation*}
S'_{n}=1+\|\partial_{\omega}w_{n}\|_{s+\beta}\leq1+\|\partial_{\omega}w_{n-1}\|_{s+\beta}+\|\partial_{\omega}h_n\|_{s+\beta}
=S'_{n-1}+\|\partial_{\omega}h_n\|_{s+\beta}.
\end{equation*}
Formula \eqref{E4.20} and Lemma \ref{lemma4.6} yield $\|\partial_{\omega}h_n\|_{s+\beta}\leq \Sigma_1+\Sigma_2$, where
\begin{align*}
\Sigma_1:=&\frac{\widehat{E}}{\gamma\omega}N^{\tau-1}_{n}(2\omega\|\rho\|_{H^1}N^2_{n}\|h_n\|_{s+\beta}+
\epsilon\|\Pi_{W}\mathrm{D}_{w}\mathcal{F}(\epsilon,w_{n-1})\partial_{\omega}w_{n-1}\|_{s+\beta})\\
&+\epsilon\frac{\widehat{E}}{\gamma\omega}N^{\tau-1}_{n}(\|(\Pi_{W}\mathrm{D}_{w}\mathcal{F}(\epsilon,w_{n})-\Pi_{W}\mathrm{D}_{w}\mathcal{F}(\epsilon,w_{n-1}))\partial_{\omega}w_{n-1}\|_{s+\beta}),\\
\Sigma_2:=&\frac{\widehat{E}}{\gamma\omega}N^{2(\tau-1)}_{n}(\|w_{n-1}\|_{s+\beta+\sigma}+\|h_{n}\|_{s+\beta})(2\omega\|\rho\|_{H^1}N^2_{n}
\|h_{n}\|_{s})\\
&+\epsilon\frac{\widehat{E}}{\gamma\omega}N^{2(\tau-1)}_{n}(\|w_{n-1}\|_{s+\beta+\sigma}+\|h_{n}\|_{s+\beta})(\|\partial_{\omega} w_{n-1}\|_{s}
+\|h_n\|_{s}\|\partial_{\omega}w_{n-1}\|_{s}).
\end{align*}
Applying \eqref{E4.17}, \eqref{E2.1}, Remark \ref{remark8},  property $(P2)$ and definition \eqref{E4.33} on $\rho_{n}$,  for $\frac{\epsilon}{\gamma\omega}$ small enough,  we can obtain
\begin{align*}
&\Sigma_1\leq\frac{E_1}{2\gamma}N^{2\tau+\sigma}_{n}S_{n-1}+\frac{\epsilon E_1}{\gamma\omega}N^{\tau-1}_{n}S'_{n-1},\quad \Sigma_2\leq\frac{E_1}{2\gamma}N^{2\tau+\sigma}_{n}S_{n-1}.
\end{align*}
The  proof of the relationship between $S''_n$ and $S''_{n-1}$  can apply  the similar step as above. For the sake of convenience, we omit the process.

Denote  $\alpha_1=\tau-1$, $\alpha_2=2\tau+1$, $\alpha_3=\tau-1+\sigma$. Formula \eqref{E4.40} leads to
\[
S'_{n}\leq \Sigma'_1+\Sigma'_2,
\]
 where
\begin{align*}
\Sigma'_1=S'_{0}\prod^{n}_{i=1}\left(1+\frac{\epsilon E_1}{\gamma\omega}N^{\alpha_1}_{i}\right),\quad
\Sigma'_2=\sum\limits_{i=1}^{n}\left(\prod_{j=2}^i\left(1+\frac{\epsilon E_1}{\gamma\omega}N^{\alpha_1}_{n-(j-2)}\right)\right)\frac{E_1}{\gamma}N^{\alpha_2}_{n-(i-1)}S_{n-i}.
\end{align*}
Since the upper bound on $\Sigma'_1$ is proved in the same way as shown in Lemma \ref{lemma4.3},  the detail is omitted.
As a consequence,
 \[\Sigma'_1\leq C'B'_{0}N^{\frac{1}{\chi-1}\alpha_1}_{n+1}.\]
We write $\Sigma'_2=\sum^n_{j=1}\Sigma'_{2,j}$, where
\begin{align*}
\Sigma'_{2,1}=\frac{E_1}{\gamma}N^{\alpha_2}_{n}S_{n-1},\quad
\Sigma'_{2,i}=\left(\prod_{j=2}^i\left(1+\frac{\epsilon E_1}{\gamma\omega}N^{\alpha_1}_{n-(j-2)}\right)\right)\frac{E_1}{\gamma}N^{\alpha_2}_{n-(i-1)}S_{n-i}\quad\forall2\leq i\leq n.
\end{align*}
On the one hand, formula \eqref{E4.35} shows
\begin{equation*}
\Sigma'_{2,1}\leq\frac{E_2S_0}{\gamma}e^{\alpha_2\mathfrak{d}\chi^n}e^{\alpha_3\mathfrak{d}\frac{\chi^n}{\chi-1}}
=\frac{E_2S_0}{\gamma}e^{\mathfrak{d}\chi^n(\alpha_2+\frac{1}{\chi-1}\alpha_3)}.
\end{equation*}
And on the other hand, a simple computation yields
\begin{align*}
\sum\limits_{i=2}^{n} \Sigma'_{2,i}&\leq\frac{E_3S_0}{\gamma}\sum\limits_{i=2}^{n}e^{\alpha_1\mathfrak{d}\frac{\chi^{n+1}-\chi^{n+2-i}}{\chi-1}}
e^{\alpha_2\mathfrak{d}\chi^{n+1-i}}e^{{\alpha_3}\mathfrak{d}\frac{\chi^{n+1-i}}{\chi-1}}\\
&\leq\frac{E_3S_0}{\gamma}e^{\alpha_1\mathfrak{d}\frac{\chi^{n+1}}{\chi-1}}
\sum\limits_{i=2}^{n}e^{{(-\alpha_1+(\chi-1)\alpha_2+\alpha_3)}\frac{\chi^{n+2-i}}{\chi-1}}\\
&\leq\frac{E_3S_0}{\gamma}e^{((\chi-1)\alpha_2+\alpha_3)\mathfrak{d}\frac{\chi^{n+1}}{\chi-1}}\leq\frac{E_3S_0}{\gamma}N^{\alpha_2+\frac{1}{\chi-1}\alpha_3}_{n+1}.
\end{align*}
Thus, for $S'_0\leq \frac{{C}_0}{\gamma}S_0$, we obtain $S'_{n}\leq C'_1\frac{S_{0}}{\gamma}N^{\mathfrak{r}}$, where $\mathfrak{r}=2\tau+\sigma+{\frac{1}{\chi-1}{(\tau-1+\sigma)}}$.
The upper bound of $S''_n$ can be proved by the similar method as employed on $S'_n$.
\end{proof}
\subsection{Whitney extension}\label{sec:6.3}
Define
\begin{align}
&\widehat{A}_n:=\left\{(\epsilon,\omega)\in A_n,\quad\mathrm{dist}((\epsilon,\omega),\partial A_n)>\frac{\gamma_{0}\gamma^4}{N^{\tau+1}_n}\right\},\label{E4.45}\\
&\widetilde{A}_n:=\left\{(\epsilon,\omega)\in A_n,\quad\mathrm{dist}((\epsilon,\omega),\partial A_n )>\frac{2\gamma_{0}\gamma^4}{N^{\tau+1}_n}\right\}\subset \widehat{A}_n.\label{E4.46}
\end{align}
Remark that $\gamma_{0}$ will be given in Lemma \ref{lemma4.10}. Define a $C^{\infty}$ cut-off function $\psi_n:~A_0\rightarrow[0,1]$ as
\begin{align*}
\psi_n:=
\begin{cases}
1\quad\quad\text{if}\quad(\omega,\epsilon)\in\widetilde{A}_n,\\
0\quad\quad\text{if}\quad(\omega,\epsilon)\in\widehat{A}_n
\end{cases}
\end{align*}
with
\begin{align}\label{E1.19}
|\partial_{\omega,\epsilon}\psi_n|\leq C{N^{\tau+1}_{n}}/{(\gamma_0\gamma^4)},
\end{align}
where $A_0$ is defined by \eqref{E4.30}. Then $\tilde{h}_n:=\psi_nh_n\in C^1(A_0;W_{N_n})$. From the definition of $\psi_n$, \eqref{E4.11}, \eqref{E1.19} and Lemma \ref{lemma4.5}, for $\frac{\epsilon}{\gamma^2\omega}$ small enough, it yields that
\begin{align}
&\|\tilde{h}_{n}\|_{s}\leq\|{h}_{n}\|_{s}\leq\frac{\widetilde{C}\epsilon}{\gamma\omega}N^{-\sigma-3}_{n},\label{E4.42}\\
\|\partial_{\omega}\tilde{h}_{n}\|_{s}&\leq|\partial_{\omega}\psi_{n}|\|h_n\|_s+|\psi_{n}|\|\partial_{\omega}h_n\|_{s}
\leq\frac{\widetilde{C}(\gamma_0)\epsilon}{\gamma^{5}\omega}N^{-1}_{n},\label{E4.43}\\
\|\partial_{\epsilon}\tilde{h}_{n}\|_{s}&\leq|\partial_{\epsilon}\psi_{n}|\|h_n\|_s+|\psi_{n}|\|\partial_{\epsilon}h_n\|_{s}
\leq\frac{\widetilde{C}(\gamma_0)}{\gamma^{5}\omega}N^{-1}_{n}.\label{E4.44}
\end{align}
Formulae \eqref{E4.42}-\eqref{E4.44} show that $\tilde{w}_n=\sum_{i=0}^{n}\tilde{h}_i$ is an extention of ${w}_n$ with $\tilde{w}_n(\epsilon,\omega)={w}_n(\epsilon,\omega)$ for all $(\epsilon,\omega)\in\widetilde{A}_n$.
Then $\tilde{w}(\epsilon,\omega)$ belongs to $C^1(A_0;W)$ with
\begin{equation}\label{E4.22}
\|\tilde{w}\|_{s}\leq\frac{K\epsilon}{\gamma\omega},
\quad\|\partial_{\omega}\tilde{w}\|_s\leq\frac{K(\gamma_0)\epsilon}{\gamma^5\omega},
\quad\|\partial_{\epsilon}\tilde{w}\|_s\leq\frac{K(\gamma_0)}{\gamma^5\omega}.
\end{equation}
Furthermore, for $n\geq1$, \eqref{E4.42} gives rise to
\begin{align}
\|\tilde{w}-\tilde{w}_{n}\|_{s}&\leq\sum\limits_{i\geq n+1}\frac{\widetilde{C}\epsilon}{\gamma\omega}N^{-\sigma-3}_{i}\stackrel{\widetilde{C}'=\widetilde{C}2^{\sigma+3}}{\leq}\sum\limits_{i\geq n+1}\frac{\widetilde{C}'\epsilon}{\gamma\omega}e^{-(\sigma+3)\mathfrak{d}\chi^{i}}\leq\frac{\widetilde{C}''\epsilon}{\gamma\omega}e^{-(\sigma+3)\mathfrak{d}\chi^{n}}\nonumber\\
&\leq\frac{\widetilde{C}''\epsilon}{\gamma\omega}N^{-\sigma-3}_{n}\leq\frac{\widetilde{C}'''\epsilon}{\gamma\omega}N^{-(\sigma+3)/\chi}_{n+1}.\label{E4.21}
\end{align}
Let ${{\lambda}}_j(\epsilon,\tilde{w}),j\geq0$ denote the eigenvalues of the Sturm-Liouville problem
\begin{align*}
\begin{cases}
-(py')'+my-\epsilon\Pi_{V}f'(v(\epsilon,\tilde{w})+\tilde{w})y={\lambda}\rho y,\\
\alpha_1y(0)-\beta_1y_{x}(0)=0,\\
\alpha_2y(\pi)+\beta_2y_{x}(\pi)=0,
\end{cases}
\end{align*}
where $c$ is defined by \eqref{E6.23}. Define
\begin{equation}\label{E4.31}
B_{\gamma}:=\left\{(\epsilon,\omega)\in\Delta^{2\gamma,\tau}_{\infty}(\tilde{w}):~\left|\omega l-\sqrt{\lambda_{j}}\right|>\frac{2\gamma}{l^{\tau}},
\quad\forall l=1,\cdots,N_0,j\geq0\right\},
\end{equation}
where
\begin{align*}
\Delta^{2\gamma,\tau}_{\infty}(\tilde{w}):=\bigcap_{n\geq0}\Delta^{2\gamma,\tau}_{N_n}(\tilde{w}_{n-1})
:=\bigg\{&(\epsilon,\omega)\in(\epsilon_1,\epsilon_2)\times(2\omega_0,+\infty):~\left|\omega l-{j}/{c}\right|>\frac{2\gamma}{l^{\tau}},\\
&\left|\omega l-\sqrt{\lambda_j(\epsilon,\tilde{w})}\right|>\frac{2\gamma}{l^{\tau}},~\forall l\geq1,j\geq0\bigg\}.\\
\end{align*}

\begin{lemm}\label{lemma4.10}
If $\frac{\epsilon}{\gamma^3\omega}$ is  small enough, then we have for some $\gamma_0>0$
\begin{equation*}
B_{\gamma}\subseteq\widetilde{A}_n\subset A_n,\quad n\geq0.
\end{equation*}
\end{lemm}
Before proving Lemma \ref{lemma4.10}, we have to introduce the following ``perturbation of self-adjoint operators'' result developed  by T.Kato \cite{kato1995perturbation}.
 Denote by $H$ and $\mathcal{B}(H)$ a Hilbert space and the space of bounded operators from $H$ to $H$ respectively.
\begin{theo}{\cite[Theorem 4.10]{kato1995perturbation}}\label{Theorem5.7}
Define $T_1=T_2+\mathcal{S}$ with $T_2$ self-adjoint in $H$ and $\mathcal{S}\in\mathcal{B}(H)$ symmetric. Then $T_1$ is a self-adjoint operator with $\mathrm{dist}(\Sigma(T_1),\Sigma(T_2))\leq\|\mathcal{S}\|$, namely
\begin{equation*}
\sup_{\xi\in\Sigma(T_1)}\mathrm{dist}(\xi,\Sigma(T_2))\leq\|\mathcal{S}\|,\quad\sup_{\xi\in\Sigma(T_2)}\mathrm{dist}(\xi,\Sigma(T_1))\leq\|\mathcal{S}\|,
\end{equation*}
where $\Sigma(T_1)$ and $\Sigma(T_2)$ are spectrums of $T_1$ and $T_2$ respectively.
\end{theo}
This implies the following lemma.
\begin{lemm}\label{Lemma6.6}
The eigenvalues of \eqref{E6.1} satisfy $\forall(\epsilon,w)\in[\epsilon_1,\epsilon_2]\times\left\{W\cap H^s:\|w\|_{s}\leq r \right\}$, $\forall n\geq0$,
\begin{equation}\label{E3.4}
|\lambda_{n}(\epsilon,w)-\lambda_{n}(\epsilon',w')|\leq\kappa(|\epsilon-\epsilon'|+\|w-w'\|_s)\
\end{equation}
 for some constant $\kappa>0$.
\end{lemm}
The proof can be seen in the Appendix.

\begin{proof}
(Lemma \ref{lemma4.10})
It is clear to read that $\widetilde{A}_n\subset A_n,n\geq0$.
 Moreover we claim that\\
 {\bf (F5):}~ There exists $\gamma_0>0$,  for $\frac{\epsilon}{\gamma^3\omega}$ small enough,  such that
\begin{equation*}
{\mathfrak{B}}\left((\epsilon,\omega),\frac{2\gamma_0\gamma^4}{N^{\tau+1}_n}\right)\subseteq A_n\quad\forall(\epsilon,\omega)\in B_{\gamma},\forall n\in\mathbf{N}.
\end{equation*}
The claim {\bf (F5)} shows that $(\epsilon,\omega)$ may belong to $\widetilde{A}_n$ for all $n\geq0$.

 Now we verify {\bf (F5)} by induction.
If $\gamma_0\leq\frac{1}{2}$,  $\forall(\epsilon_1,\omega_1)\in{\mathfrak{B}}\left((\epsilon,\omega),\frac{2\gamma_0\gamma^4}{N^{\tau+1}_0}\right)$, $\forall l=1,\cdots,N_0$,  then we can obtain
\begin{equation*}
\left|\omega_1l-\sqrt{\lambda_j}\right|\geq\left|\omega l-\sqrt{\lambda_j}\right|-\left|\omega-\omega_1\right|l>\frac{2\gamma}{l^{\tau}}-\frac{2\gamma_0\gamma^4}{N^{\tau+1}_0}l
\geq\frac{\gamma}{l^\tau}+\frac{\gamma}{N^\tau_0}-\frac{2\gamma_0\gamma^4}{N^{\tau}_0}\geq\frac{\gamma}{l^{\tau}}.
\end{equation*}

Another step is to suppose  that
\[{\mathfrak{B}}\left((\epsilon,\omega),\frac{2\gamma_0\gamma^4}{N^{\tau+1}_n}\right)\subseteq A_n,\]
 which  implies that $(\epsilon,\omega)\in\widetilde{A}_n$. As a result $\tilde{w}_n(\epsilon,\omega)={w}_n(\epsilon,\omega)$.

Finally, let us check that the claim {\bf (F5)} holds at $(n+1)$-th step.
A similar argument yields $\forall(\epsilon_1,\omega_1)\in\mathfrak{B}((\epsilon,\omega),\frac{2\gamma_0\gamma^4}{N^{\tau+1}_{n+1}})$, $\forall l=1,\cdots,N_{n+1},$
\begin{equation*}
\left|\omega_1l-\frac{j}{c}\right|\geq\left|\omega l-\frac{j}{c}\right|-\left|\omega-\omega_1\right|l>\frac{2\gamma}{l^{\tau}}-\frac{2\gamma_0\gamma^4}{N^{\tau+1}_{N_{n+1}}}l\geq \frac{\gamma}{l^\tau}
+\frac{\gamma}{N^\tau_{n+1}}-\frac{2\gamma_0\gamma^4}{N^{\tau}_{n+1}}\geq\frac{\gamma}{l^{\tau}}
\end{equation*}
if $\gamma_0\leq\frac{1}{2}$. For brevity, denote $\lambda_{j,n}(\epsilon_1,\omega_1):=\lambda_{j}(\epsilon_1,w_n(\epsilon_1,\omega_1))$, $\tilde{\lambda}_j(\epsilon,\omega):=\lambda_{j}(\epsilon,\tilde{w}(\epsilon,\omega))$. Moreover, let
\begin{equation*}
\delta_0:=\inf\left\{~|\lambda_j(\epsilon,\omega)|:~j\geq0,\epsilon\in[\epsilon_1,\epsilon_2],\|w\|_s\leq r~\right\}.
\end{equation*}
Lemma \ref{lemma6.1} together with formula \eqref{E3.4} can show that $\delta_0>0$ is a constant. It follows from \eqref{E3.4}, Remark \ref{remark8}, \eqref{E4.21} and $\omega\geq\omega_0>\gamma$ that
\begin{align*}
\left|\sqrt{\lambda_{j,n}(\epsilon_1,\omega_1)}-\sqrt{\tilde{\lambda}_{j}(\epsilon,\omega)}\right|=&\frac{|\lambda_{j,n}(\epsilon_1,\omega_1)
-\tilde{\lambda}_{j}(\epsilon,\omega)|}{\left|\sqrt{\lambda_{j,n}(\epsilon_1,\omega_1)}\right|+\left|\sqrt{\tilde{\lambda}_{j}(\epsilon,\omega)}\right|}\leq
\frac{1}{\sqrt{\delta_0}}{|\lambda_{j,n}(\epsilon_1,\omega_1)
-\tilde{\lambda}_{j}(\epsilon,\omega)|}\\
\leq&\frac{\kappa}{\sqrt{\delta_0}}|\epsilon-\epsilon_1|+\frac{\kappa}{\sqrt{\delta_0}}\|w_n(\epsilon_1,\omega_1)-\tilde{w}(\epsilon,\omega)\|_s\\
\leq&\frac{\kappa}{\sqrt{\delta_0}}|\epsilon-\epsilon_1|+\frac{\kappa}{\sqrt{\delta_0}}\|w_n(\epsilon_1,\omega_1)-{w}_n(\epsilon_1,\omega)\|_s\\
&+\frac{\kappa}{\sqrt{\delta_0}}\|w_n(\epsilon_1,\omega)-{w}_n(\epsilon,\omega)\|_s
+\frac{\kappa}{\sqrt{\delta_0}}\|\tilde{w}_n(\epsilon,\omega)-\tilde{w}(\epsilon,\omega)\|_s\\
\leq&\frac{\kappa}{\sqrt{\delta_0}}\left(\frac{2\gamma_0\gamma^4}{N^{\tau+1}_{n+1}}+\frac{2E}{\gamma^2\omega}\frac{2\gamma_0\gamma^4}{N^{\tau+1}_{n+1}}+
\frac{\widetilde{C}'''\epsilon}{\gamma\omega }{N^{-(\sigma+3)/\chi}_{n+1}}\right).
\end{align*}
If the fact $-(\sigma+3)/\chi\leq{-\tau}$ holds, then, for $\gamma_0,\frac{\epsilon}{\gamma^2\omega}$ small enough, we infer
\begin{align}\label{E8.8}
\left|\sqrt{\lambda_{j,n}(\epsilon_1,\omega_1)}-\sqrt{\tilde{\lambda}_{j}(\epsilon,\omega)}\right|\leq\frac{\gamma}{2l^{\tau}}.
\end{align}
In fact, define a function $\mathfrak{g}$ on $(1,2)$ as
\begin{align*}
\mathfrak{g}(\tau):=\frac{\sigma+3}{\tau}\quad\text{with}~\sigma=\frac{\tau(\tau-1)}{2-\tau}.
\end{align*}
It is evident that $\min\limits_{1<\tau<2}\mathfrak{g}(\tau)=\mathfrak{g}(3-\sqrt{3})=1+\sqrt{3}$. This shows that $\chi\leq\frac{\sigma+3}{\tau}$ owing to $1<\chi\leq2$.
Consequently, by means of \eqref{E8.8}, $\forall(\epsilon_1,\omega_1)\in\mathfrak{B}((\epsilon,\omega),\frac{2\gamma_0\gamma^4}{N^{\tau+1}_{n+1}})$, $\forall l=1,\cdots,N_{n+1}$,
we can obtain that   for $\gamma_0,\frac{\epsilon}{\gamma^2\omega}$ small enough
\begin{align*}
\left|\omega_1l-\sqrt{\lambda_{j,n}(\epsilon_1,\omega_1)}\right|&\geq\left|\omega l-\sqrt{\tilde{\lambda}_{j}(\epsilon,\omega)}\right|-|\omega-\omega_1|l-\left|\sqrt{\lambda_{j,n}(\epsilon_1,\omega_1)}-\sqrt{\tilde{\lambda}_{j}(\epsilon,\omega)}\right|\\
&>\frac{2\gamma}{l^{\tau}}-\frac{2\gamma_0\gamma^4}{N^{\tau+1}_{n+1}}l-\frac{\gamma}{2l^{\tau}}>\frac{\gamma}{l^{\tau}}.
\end{align*}
The proof  is completed.
\end{proof}
Let $\Omega:=(\epsilon',\epsilon'')\times(\omega',\omega'')$ stand for a rectangle contained in $(\epsilon_1,\epsilon_2)\times(2\omega_0,+\infty)$. Denote by $B_{\gamma}(\epsilon)$ ($\mathrm{resp}$. $B_{\gamma}(\omega)$) the $\epsilon$ ($\mathrm{resp}$. $\omega$)-section as follows:
\begin{align*}
&B_{\gamma}(\epsilon):=\left\{\omega:(\epsilon,\omega)\in B_\gamma\right\}~
 \forall \epsilon\in (\epsilon',\epsilon'');\quad B_{\gamma}(\omega):=\left\{\epsilon:(\epsilon,\omega)\in B_\gamma\right\}~ \forall \omega\in (\omega',\omega'').
\end{align*}
We have to fix $\omega''-\omega'=\mathbf{\textbf{constant}}$.
\begin{lemm}
For $\frac{\epsilon}{\gamma^5\omega}$ small enough, the measure estimate on $B_{\gamma}(\epsilon)$ satisfies
\begin{equation}\label{E4.24}
|B_{\gamma}(\epsilon)\cap(\omega',\omega'')|\geq(1-Q\gamma)(\omega''-\omega')
\end{equation}
for some constant $Q>0$. Furthermore
\begin{equation}\label{E4.25}
|B_{\gamma}\cap\Omega|\geq(1-Q\gamma)|\Omega|=(1-Q\gamma)(\omega''-\omega')(\epsilon''-\epsilon').
\end{equation}
\end{lemm}
\begin{proof}
The completementary set $\mathfrak{R}$ of $B_{\gamma}$ is
\begin{equation*}
\mathfrak{R}\subseteq\mathfrak{R}^1\cup\mathfrak{R}^2\cup\mathfrak{R}^3,
\end{equation*}
where $\mathfrak{R}^1=\bigcup\limits_{l\geq1,j\geq0}\mathfrak{R}^1_{l,j}$, $\mathfrak{R}^2=\bigcup\limits_{l\geq1,j\geq0}\mathfrak{R}^2_{l,j}$, $\mathfrak{R}^3=\bigcup\limits_{l\geq1,j\geq0}\mathfrak{R}^3_{l,j}$, with
\begin{align*}
\mathfrak{R}^1_{l,j}&:=\left\{\omega\in(\omega',\omega''):\left|\omega l-\sqrt{\tilde{\lambda}_{j}(\epsilon,\omega)}\right|\leq\frac{2\gamma}{l^{\tau}}\right\},\\
\mathfrak{R}^2_{l,j}&:=\left\{\omega\in(\omega',\omega''):\left|\omega l-\sqrt{{\lambda}_{j}}\right|\leq\frac{2\gamma}{l^{\tau}}\right\},\quad\mathfrak{R}^3_{l,j}:=\left\{\omega\in(\omega',\omega''):\left|\omega l-\frac{j}{c}\right|\leq\frac{2\gamma}{l^{\tau}}\right\}.
\end{align*}
Let us give  the upper bound of $|\mathfrak{R}^1|$. It follows from \eqref{E4.22}, \eqref{E3.4} and the definition of $\delta_0$ that
\begin{align*}
\left|\sqrt{\tilde{\lambda}_{j}(\epsilon,\omega_1)}-\sqrt{\tilde{\lambda}_{j}(\epsilon,\omega)}\right|&=\frac{\left|\tilde{\lambda}_{j}(\epsilon,\omega_1)
-\tilde{\lambda}_{j}(\epsilon,\omega)\right|}{\left|\sqrt{\tilde{\lambda}_{j}(\epsilon,\omega_1)}\right|+\left|\sqrt{\tilde{\lambda}_{j}(\epsilon,\omega)}\right|}\leq
\frac{1}{\sqrt{\delta_0}}{\left|\tilde{\lambda}_{j}(\epsilon,\omega_1)
-\tilde{\lambda}_{j}(\epsilon,\omega)\right|}\\
&\leq\frac{\kappa}{\sqrt{\delta_0}}\|\tilde{\omega}(\epsilon,\omega_1)-\tilde{\omega}(\epsilon,\omega)\|_s\leq\frac{\kappa \epsilon K(\gamma_0)}{\sqrt{\delta_0}\gamma^5\omega}|\omega_1-\omega|,
\end{align*}
which  implies  $\left|\partial_{\omega}\sqrt{\tilde{\lambda}_j(\epsilon,\omega)}\right|\leq\frac{\kappa \epsilon K(\gamma_0)}{\sqrt{\delta_0}\gamma^5\omega}$.

 Define the function  $\mathfrak{f}(\omega):=\omega l-\sqrt{\tilde{\lambda}_j(\epsilon,\omega)}$. For $\frac{\epsilon}{\gamma^5\omega}$  small enough, it is obvious that
\begin{equation*}
\partial_{\omega}\mathfrak{f}(\omega)=l-\partial_{\omega}\sqrt{\tilde{\lambda}_j(\epsilon,\omega)}\geq\frac{l}{2},\quad\forall l\geq1.
\end{equation*}
 A simple computation yields
\begin{equation*}
|\mathfrak{R}^1_{l,j}|\leq\frac{|\mathfrak{f}(\omega_1)-\mathfrak{f}(\omega_2)|}{|\partial_{\omega}\mathfrak{f}(\omega)|}\leq\frac{8\gamma}{l^{\tau+1}}.
\end{equation*}
If $\mathfrak{R}^1_{l,j}\neq\emptyset$, for fixed $l$, then it is easy to show that
\begin{equation*}
\omega'l-\frac{2\gamma}{l^\tau}\leq\sqrt{\tilde{\lambda}_j(\epsilon,\omega)}\leq\omega''l+\frac{2\gamma}{l^\tau}.
\end{equation*}
Set
\begin{equation}\label{E4.23}
\delta_1:=\inf\left\{\left|{\sqrt{\lambda_{j+1}(\epsilon,\omega)}}-\sqrt{{\lambda}_j(\epsilon,\omega)}\right|:j\geq0,\epsilon\in[\epsilon_1,\epsilon_2],\|w\|_s\leq r\right\}>0,
\end{equation}
and
\begin{equation*}
\delta_2:=\inf\left\{\left|{\sqrt{\lambda_{j+1}(\epsilon,\omega)}}-\sqrt{{\lambda}_j(\epsilon,\omega)}\right|:j\geq0,(\epsilon,\omega)\in B_\gamma\right\},
\end{equation*}
while \eqref{E4.23} will be verified in the appendix.
Recall that ${{\lambda_{j}(\epsilon,\omega)}}={{\tilde{\lambda}_{j}(\epsilon,\omega)}}$ on $B_\gamma$.
It follows from  \eqref{E4.22} that  for $\frac{\epsilon}{\gamma\omega}$ small enough
 \[\|\tilde{w}\|_{s}\leq\frac{K\epsilon}{\gamma\omega}<r,\]
 which  implies that  $\delta_2\geq\delta_1>0$. Thus
\begin{equation*}
\sharp j\leq\frac{1}{\delta_1}(l(\omega''-\omega')+\frac{4\gamma}{l^{\tau}})+1,
\end{equation*}
where $\sharp j$ denotes the number of $j$. As a consequence
\begin{equation*}
|\mathfrak{R}^1|\leq\sum\limits_{l=1}^{+\infty}\frac{8\gamma}{l^{\tau+1}}\left(\frac{1}{\delta_1}(l(\omega''-\omega')
+\frac{4\gamma}{l^{\tau}})+1\right)\leq\sum\limits_{l=1}^{+\infty}\frac{8\gamma}{l^{\tau+1}}Q''l(\omega''-\omega')
\leq Q'\gamma(\omega''-\omega').
\end{equation*}
The remainder of the argument on the upper bounds of $|\mathfrak{R}^2|,~|\mathfrak{R}^3|$ is analogous to the one used as above and so is omitted. Finally, we get
\[|\mathfrak{R}^2|\leq Q'\gamma(\omega''-\omega'), \quad|\mathfrak{R}^3|\leq Q'\gamma(\omega''-\omega').\]
Formula \eqref{E4.24} is obtained. In addition
\begin{equation*}
|B_{\gamma}\cap\Omega|=\int_{\epsilon'}^{\epsilon''}|B_{\gamma}(\epsilon)\cap(\omega',\omega'')|
~\mathrm{d}\epsilon\geq(1-Q\gamma)|\Omega|.
\end{equation*}
\end{proof}

\begin{lemm}
For $\frac{\epsilon}{\gamma^5\omega}$ small enough, for every $\gamma_1\in(0,1)$, the measure estimate on $B_{\gamma}(\omega)$ satisfies
\begin{equation*}
\left|\left\{\omega\in(\omega',\omega''):\frac{|B_{\gamma}(\omega)\cap(\epsilon',\epsilon'')|}
{\epsilon''-\epsilon'}\geq1-\gamma_1\right\}\right|\geq(\omega''-\omega')\left(1-Q\frac{\gamma}{\gamma_1}\right).
\end{equation*}
\end{lemm}
\begin{proof}
Define
\begin{equation*}
\begin{aligned}
&\Omega^{+}_{\epsilon}:=\{\omega\in(\omega',\omega''):
|B_{\gamma}(\omega)\cap(\epsilon',\epsilon'')|\geq(\epsilon''-\epsilon')(1-\gamma_1)\},\\
&\Omega^{-}_{\epsilon}:=\{\omega\in(\omega',\omega''):
|B_{\gamma}(\omega)\cap(\epsilon',\epsilon'')|<(\epsilon''-\epsilon')(1-\gamma_1)\}.
\end{aligned}
\end{equation*}
It follows from the Fubini`s theorem that
\begin{align}
|B_{\gamma}\cap\Omega|&=\int^{\omega''}_{\omega'}|B_{\gamma}(\omega)\cap(\epsilon',\epsilon'')|\mathrm{d}\omega=\int_{\Omega^{+}_{\epsilon}}|B_{\gamma}(\omega)\cap(\epsilon',\epsilon'')|\mathrm{d}\omega+
\int_{\Omega^{-}_{\epsilon}}|B_{\gamma}(\omega)\cap(\epsilon',\epsilon'')|\mathrm{d}\omega\nonumber\\
&\leq(\epsilon''-\epsilon')|\Omega^{+}_{\epsilon}|+(\epsilon''-\epsilon')(1-\gamma_1)|\Omega^{-}_{\epsilon}|. \label{E4.26}
\end{align}
Formulae \eqref{E4.25}, \eqref{E4.26} and the quality $|\Omega^{+}_{\epsilon}|+|\Omega^{-}_{\epsilon}|=\omega''-\omega'$ give
\begin{equation*}
(\omega''-\omega')(1-Q\gamma)\leq(\omega''-\omega')-\gamma_1|\Omega^{-}_{\epsilon}|\Longrightarrow
|\Omega^{-}_{\epsilon}|\leq(\omega''-\omega')Q\frac{\gamma}{\gamma_1}.
\end{equation*}
Combining this with $(\omega''-\omega')(1-Q\gamma)\leq|\Omega^{+}_{\epsilon}|+(1-\gamma_1)|\Omega^{-}_{\epsilon}|$, we derive
\begin{equation*}
|\Omega^{+}_{\epsilon}|\geq(\omega''-\omega')(1-Q\frac{\gamma}{\gamma_1}).
\end{equation*}
This completes the proof of the lemma.
\end{proof}

The above discussion in subsections \ref{sec:6.0}-\ref{sec:6.3} gives that Theorem \ref{Th2} holds. Lemma \ref{lemma3.1} together with Theorem \ref{Th2} may show that
\begin{equation*}
\bar{u}:=v(\epsilon,\tilde{w}(\epsilon,\omega))+\tilde{w}(\epsilon,\omega)\in \mathcal{{H}}^1_g(0,\pi)\oplus(W\cap H^s)
\end{equation*}
is a solution of Eq. \eqref{E1.4}. Thus the proof of Theorem \ref{Th1} is completed .

\section{Appendix}\label{sec:8}
\subsection{The proof of Remark \ref{remark1}}\label{prm1}
{\rm{(i)}} We decompose $uv$ as $\sum_{l\in\mathbf{Z}}(\sum_{k\in\mathbf{Z}}u_{l-k}v_{k})e^{\mathrm{i}lt}$ for all $u,v\in H^s$ with $s>1/2$. Using the
Cauchy inequality, we can get
\begin{align*}
\|uv\|^2_{s}&=\sum\limits_{l}(1+l^{2s})\left\|\sum\limits_{k}u_{l-k}v_{k}\right\|^2_{H^1}\leq\sum\limits_{l}\left(\sum\limits_{k}(1+l^{2s})^{\frac{1}{2}}c_{lk}\|u_{l-k}v_{k}\|_{H^1}\frac{1}{c_{lk}}\right)^2\\
&{\leq}\sum\limits_{l}\left(\sum\limits_{k}\frac{1}{c^2_{lk}}
\sum\limits_{k}\|u_{l-k}\|^2_{H^1}(1+|l-k|^{2s})\|v_{k}\|^2_{H^1}(1+k^{2s})\right),
\end{align*}
where
\begin{equation*}
c_{lk}:=\left(\frac{(1+k^{2s})(1+|l-k|^{2s})}{1+l^{2s}}\right)^{\frac{1}{2}}.
\end{equation*}
A simple process yields for $s>1/2$
\begin{equation*}
\begin{aligned}
1+l^{2s}&\leq1+(|k|+|l-k|)^{2s}\leq1+2^{2s-1}(k^{2s}+|l-k|^{2s})<2^{2s-1}(1+k^{2s}+1+|l-k|^{2s}).
\end{aligned}
\end{equation*}
Then, for $s>1/2$, we have
\begin{align*}
\sum\limits_{k\in\mathbf{Z}}\frac{1}{c^2_{lk}}<2^{2s-1}\left(\sum\limits_{k\in\mathbf{Z}}\frac{1}{1+k^{2s}}+\sum\limits_{k\in\mathbf{Z}}\frac{1}{1+|l-k|^{2s}}\right)
=2^{2s}\sum\limits_{k\in\mathbf{Z}}\frac{1}{1+k^{2s}}=:C(s)^2.
\end{align*}
Hence $\|uv\|_{s}$ may be bounded from above by $C(s)\|u\|_s\|v\|_s$.
\\
{\rm{(ii)}} It also follows from the  Cauchy inequality  that
\begin{equation*}
\sum\limits_{l\in\mathbf{Z}}\|u_l\|_{L^{\infty}}\leq C\sum\limits_{l\in
\mathbf{Z}}\|u_l\|_{H^1}\leq C\left(\sum\limits_{l\in
\mathbf{Z}}\|u_l\|_{H^1}(1+l^{2s})\right)^{\frac{1}{2}}\left(\sum\limits_{l\in\mathbf{Z}}\frac{1}{1+l^{2s}}\right)^{\frac{1}{2}}\leq C(s)\|u\|_{s}.
\end{equation*}

\subsection{Preliminaries}\label{sec:5.2}
By the definitions of $H^{s}$,  for completeness,   we list  Lemmas \ref{lemma2.1}-\ref{lemma2.5} and the proof can be found in   \cite{berti2008cantor}.
\begin{lemm}[Moser-Nirenberg]\label{lemma2.1}
For all $u_1,u_2\in H^{s'}\cap H^s$ with $s'\geq0$ and $s>\frac{1}{2}$, we have
\begin{align}
\|u_1u_2\|_{{s'}}
&\leq C(s')\left(\|u_1\|_{L^{\infty}(\mathbf{T},H^1(0,\pi))}\|u_2\|_{{s'}}+\|u_1\|_{{s'}}\|u_2\|_{L^{\infty}(\mathbf{T},H^1(0,\pi))}\right)\label{E2.10}\\
&\leq  C(s')\left(\|u_1\|_{s}\|u_2\|_{{s'}}+\|u_1\|_{{s'}}\|u_2\|_{s}\right).\label{E2.1}
\end{align}
\end{lemm}
\begin{lemm}[Logarithmic convexity]\label{lemma2.2}
Let $0\leq\mathfrak{a}\leq a\leq b\leq\mathfrak{b}$ satisfy $a+b=\mathfrak{a}+\mathfrak{b}$. Taking $\mathfrak{l}:=\frac{\mathfrak{b}-a}{\mathfrak{b}-\mathfrak{a}}$, it holds
\begin{equation*}
\|u_1\|_{a}\|u\|_{b}\leq\mathfrak{l}\|u_1\|_{\mathfrak{l}}\|u_2\|_{\mathfrak{b}}+(1-\mathfrak{l})\|u_2\|_{\mathfrak{a}}\|u_1\|_{\mathfrak{b}}
\end{equation*}
for all $u_1,u_2\in H^{\mathfrak{b}}$. In particular
\begin{equation}\label{E2.3}
\|u\|_{a}\|u\|_{b}\leq\|u\|_{\mathfrak{a}}\|u\|_{\mathfrak{b}}, \quad \forall u\in H^{\mathfrak{b}}.
\end{equation}
\end{lemm}
Define
\begin{align*}
\mathscr {C}_k:=\left\{f\in C([0,\pi]\times\mathbf{R};\mathbf{R}):~u\mapsto f(\cdot,u)~{\mathrm{belongs}}~{\mathrm{ to}}~C^{k}(\mathbf{R};H^{1}(0,\pi))\right\}.
\end{align*}
\begin{lemm}\label{lemma2.5}
Let $f\in\mathscr {C}_1$. Then the composition operator $u(x)\mapsto f(x,u(x))$ belongs to $C(H^1(0,\pi);H^1(0,\pi))$ with
\begin{equation*}
\|f(x,u(x))\|_{H^1}\leq C\left(\max_{u\in[-\mathfrak{C},\mathfrak{C}]}\|f(\cdot,u)\|_{H^1}+\max_{u\in[-\mathfrak{C},\mathfrak{C}]}\|\partial_uf(\cdot,u)\|_{H^1}\|u\|_{H^1}\right),
\end{equation*}
where $\mathfrak{C}:=\|u\|_{L^{\infty}(0,\pi)}$. In particular, we have
\begin{equation*}
\|f(x,0)\|_{H^1}\leq C.
\end{equation*}
\end{lemm}
With the help of Lemmas \ref{lemma2.1}-\ref{lemma2.5}, the following lemma can be obtained.
\begin{lemm}
Let $f\in\mathcal{C}_{k}$ with $k\geq1$. Then the composition operator
\begin{equation*}
u(t,x)\mapsto f(t,x,u(t,x))
\end{equation*}
is a continuous map from $H^{s}\cap {H}^{s'}$ to ${H}^{s'}$ for all $s>\frac{1}{2},0\leq s'\leq k-1$. Furthermore
\begin{equation}\label{E2.4}
\|f(t,x,u)\|_{{s'}}\leq C(s',\|u\|_{s})(1+\|u\|_{{s'}}).
\end{equation}
\end{lemm}
\begin{proof}
If $s'=l$ is an integer, for all $l\in\mathbf{N}$ with $l\leq k-1$, $u\in H^{s}\cap H^{l}$,  we have to prove that
\begin{equation}\label{E3.0}
\|f(t,x,u)\|_{l}\leq C(l,\|u\|_{s})(1+\|u\|_{l})
\end{equation}
 and that
 \begin{equation}\label{E3.1}
 f(t,x,u_n)\rightarrow f(t,x,u)
 \end{equation}
when $u_n\rightarrow u$ in $ H^{s}\cap H^{l}$.
 Let us verify formula \eqref{E3.0} and  the continuous property of $f$ with respect to $u$ by a recursive argument. Obviously, Lemma \ref{lemma2.5} indicates for all $g\in\mathscr {C}_1$,
\begin{equation}\label{E2.6}
\|g(x,u(x))\|_{H^1}\leq C(1+\|u\|_{H^1}).
\end{equation}
First, for $l=0$, we can derive
\begin{align*}
\|f(t,x,u)\|_{0}&\leq C\max_{t\in\mathbf{T}}\|f(t,\cdot,u(t,\cdot))\|_{H^1(0,\pi)}\stackrel{\mathrm{Fix}~\mathrm{t},~\eqref{E2.6}}{\leq}C(1+\max_{t\in\mathbf{T}}\|u(t,\cdot)\|_{H^1(0,\pi)})\\
&\stackrel{\mathrm{Remark}\ref{remark1}\mathrm{(ii)}}{\leq}C(1+\|u\|_{s})=:C(\|u\|_{s}).
\end{align*}
A similar argument as above can yield
\begin{equation}\label{E2.7}
\|\partial_{t}f(t,x,u)\|_{0}\leq C(\|u\|_{s}),\quad\max_{t\in\mathbf{T}}\|\partial_{u}f(t,\cdot,u(t,\cdot))\|_{H^1(0,\pi)}\leq C(\|u\|_{s}).
\end{equation}
By Remark \ref{remark1} $\mathrm{(ii)}$, it leads to
\[\max_{t\in\mathbf{T}}\|u_{n}(t,\cdot)-u(t,\cdot)\|_{H^1(0,\pi)}\rightarrow0 \quad \text { as}~ u_n\rightarrow u ~ \text { in}~  H^{s}\cap H^{l}.
\] Then, it follows from  the continuity property in Lemma \ref{lemma2.5} and the compactness of $\mathbf{T}$ that
\begin{equation*}
\|f(t,x,u_{n})-f(t,x,u)\|_{0}\leq C\max_{t\in\mathbf{T}}\|f(t,\cdot,u_{n}(t,\cdot))-f(t,\cdot,u(t,\cdot))\|_{H^1(0,\pi)}\rightarrow0
\end{equation*}
as $ u_n\rightarrow u$
in $ H^{s}\cap H^{l}$.

Assume that \eqref{E3.0} holds for $l=\mathfrak{k}$ with $\mathfrak{k}\in \mathbf N^+$, then we have to  verify  that  it holds for $l=\mathfrak{k}+1$ with $\mathfrak{k}+1 \leq k-1$.

Since $\partial_{t}f,\partial_{u}f\in\mathcal{C}_{k-1}$, by the above assumption for $l=\mathfrak{k}$, we get for $u\in H^s\cap H^l$
\begin{equation*}
\|\partial_{t}f(t,x,u)\|_{\mathfrak{k}}\leq C(\mathfrak{k},\|u\|_{s})(1+\|u\|_{\mathfrak{k}}),\quad\|\partial_{u}f(t,x,u)\|_{\mathfrak{k}}\leq C(\mathfrak{k},\|u\|_{s})(1+\|u\|_{\mathfrak{k}}).
\end{equation*}
Let $\mathfrak{q}(t,x):=f(t,x,u(t,x))$. We write $\mathfrak{q}$ as the form
 \[\mathfrak{q}(t,x)=\sum_{j\in\mathbf{Z}}\mathfrak{q}_{j}(x)e^{{\rm i}jt}.\]
 It is obvious that  $\mathfrak{q}_{t}(t,x)=\sum_{j\in\mathbf{Z}}{\rm i}j\mathfrak{q}_{j}(x)e^{{\rm i}jt}$. By the definition of $\|\cdot\|_{s}$, we obtain
\begin{align*}
\|\mathfrak{q}(t,x)\|^2_{{\mathfrak{k}+1}}&=\sum\limits_{j\in\mathbf{Z}}(1+j^{2(\mathfrak{k}+1)})\|\mathfrak{q}_{j}\|^2_{H^1}=
\sum_{j\in\mathbf{Z}}\|\mathfrak{q}_{j}\|^2_{H^1}+\sum_{j\in\mathbf{Z}}j^{2\mathfrak{k}}\|{\rm i}j\mathfrak{q}_j\|^2_{H^1}\\
&\leq\|\mathfrak{q}(t,x)\|^2_{0}+\|\mathfrak{q}_{t}(t,x)\|^2_{\mathfrak{k}}
\leq\left(\|\mathfrak{q}(t,x)\|_{0}+\|\mathfrak{q}_{t}(t,x)\|_{\mathfrak{k}}\right)^2.
\end{align*}
As a consequence
\begin{equation}\label{E2.8}
\|f(t,x,u)\|_{{\mathfrak{k}+1}}\leq \|f(t,x,u)\|_{0}+\|\partial_{t}f(t,x,u)\|_{\mathfrak{k}}+\|\partial_{u}f(t,x,u)\partial_{t}u\|_{\mathfrak{k}}.
\end{equation}
For $l=1$, formulae \eqref{E2.7}-\eqref{E2.8} carry out
\begin{align*}
\|f(t,x,u)\|_{1}&\leq \|f(t,x,u)\|_{0}+\|\partial_{t}f(t,x,u)\|_{0}+\max_{t\in\mathbf{T}}\|\partial_{u}f(t,\cdot,u(t,\cdot))\|_{H^1(0,\pi)}\|u\|_{1}\\
&{\leq}~2C(\|u\|_{s})+C(\|u\|_{s})\|u\|_{1}\\
&\leq C(1,\|u\|_{s})(1+\|u\|_{1}),
\end{align*}
where $C(1,\|u\|_{s}):=2C(\|u\|_{s})$. Letting $s_{1}\in(1/2,\min(1,s))$, we establishes ${s}_{1}<s_1+1\leq l<l+1$ for $l\geq2$. Therefore, it follows from \eqref{E2.3} that
\begin{equation}\label{E2.9}
\|u\|_{{l}}\|u\|_{{s_1+1}}\leq\|u\|_{{l+1}}\|u\|_{{s_1}}\leq\|u\|_{{l+1}}\|u\|_{s}.
\end{equation}
Thus, by combining \eqref{E2.10}, \eqref{E2.7}, \eqref{E2.9}, Remark \ref{remark1} $\mathrm{(ii)}$ with the above assumption for $l=\mathfrak{k}$, we get
\begin{align*}
\|f(t,x,u)\|_{{\mathfrak{k}+1}}
\leq & C(\|u\|_{s})+C(\mathfrak{k},\|u\|_{s})(1+\|u\|_{\mathfrak{k}})+C(\mathfrak{k})\|\partial_{u}f(t,x,u)\|_{\mathfrak{k}}\|\partial_{t}u\|_{L^{\infty}(\mathbf{T},H^1(0,\pi))}\\
&+C(\mathfrak{k})
\|\partial_{u}f(t,x,u)\|_{L^{\infty}(\mathbf{T},H^1(0,\pi))}\|u\|_{{\mathfrak{k}+1}}\\
\leq &C(\|u\|_{s})+C(\mathfrak{k},\|u\|_{s})(1+\|u\|_{\mathfrak{k}})+C(\mathfrak{k})C(\mathfrak{k},\|u\|_{s})(1+\|u\|_{\mathfrak{k}})\|u\|_{{s_1+1}}\\
&+C(\mathfrak{k})C(\|u\|_{s})\|u\|_{{\mathfrak{k}+1}}\\
\leq&C(\mathfrak{k}+1,\|u\|_{s})(1+\|u\|_{{\mathfrak{k}+1}}),
\end{align*}
where  $C(\mathfrak{k}+1,\|u\|_{s})=4\max{\left\{C(\|u\|_{s}),C(\mathfrak{k},\|u\|_{s}),C(\mathfrak{k})C(\mathfrak{k},\|u\|_{s})(1+\|u\|_{s}),C(\mathfrak{k})C(\|u\|_{s})\right\}}.$
This implies that \eqref{E3.0} is  satisfied for $l=\mathfrak{k}+1$.

Finally,  we  assume that \eqref{E3.1} holds for $l=\mathfrak{k}$.
 Using the inequality \eqref{E2.8}, we may obtain that the continuity property of $f$ with respect to $u$ also holds for  $l=\mathfrak{k}+1$ with $\mathfrak{k}+1\leq k-1$.

When  $s'$ is not an integer, we can obtain the result by the  Fourier dyadic decomposition. The argument is similar to the proof of the Lemma A.1 in \cite{Delort2011}.
\end{proof}
\begin{lemm}\label{lemma2.4}
For all $0\leq s'\leq k-3$, define a map $F$ as
\begin{align*}
F:\quad H^{s}\cap H^{s'}&\rightarrow {H}^{s'}\\
u&\mapsto f(t,x,u),
\end{align*}
where $f\in\mathcal{C}_{k}$ with $k\geq3$. Then $F$ is a  $C^2$ map with respect to $u$. Furthermore for all $h\in H^s\cap H^{s'}$, we have
\begin{equation*}
{\rm D}_{u}F(u)[h]=\partial_{u}f(t,x,u)h, \quad {\rm D}^2_{u}G(u)[h,h]=\partial^2_{u}f(t,x,u)h^2,\quad \text{with}
\end{equation*}
\begin{align}\label{E2.5}
&\|\partial_{u}f(t,x,u)\|_{{s'}}\leq C(s',\|u\|_{s})(1+\|u\|_{{s'}}),\quad\|\partial^2_{u}f(t,x,u)\|_{{s'}}\leq C(s',\|u\|_{s})(1+\|u\|_{{s'}}).
\end{align}
\end{lemm}
\begin{proof}
Since $\partial_{u}f,\partial^2_{u}f$ are in $\mathcal{C}_{k-1},\mathcal{C}_{k-2}$ respectively, by the inequality \eqref{E2.4}, we verify that
\eqref{E2.5} holds and that the maps $u\mapsto\partial_{u}f(t,x,u)$, $u\mapsto\partial^2_{u}f(t,x,u)$ are continuous.
Let us check that $F$  is $C^2$ respect to $u$.
 It follows form the continuity property of $u\mapsto\partial_{u}f(t,x,u)$ that
\begin{align*}
\|f(t,x,u+h)-f(t,x,u)-&\partial_{u}f(t,x,u)h\|_{{s'}}=\left\|h\int_0^1(\partial_{u}f(t,x,u+\mathfrak{t} h)-\partial_{u}f(t,x,u))~\mathrm{d}\mathfrak{t}\right\|_{{s'}}\\
&\leq C(s')\|h\|_{{\max{\{s,s'\}}}}\max_{\sigma\in[0,1]}\|\partial_{u}f(t,x,u+\mathfrak{t} h)-\partial_{u}f(t,x,u)\|_{{\max{\{s,s'\}}}}\\
&=o(\|h\|_{{\max{\{s,s'\}}}}).
\end{align*}
Hence, for all $h\in H^s\cap H^{s'}$, we obtain that
\begin{equation*}
{\rm D}_{u}F(u)[h]=\partial_{u}f(t,x,u)h
\end{equation*}
and that $u\mapsto {\rm D}_uF(u)$ is continuous. In addition
\begin{align*}
&\partial_{u}f(t,x,u+\sigma h)h-\partial_{u}f(t,x,u)h-\partial^2_{u}f(t,x,u)h^2=h^2\int_0^1(\partial^2_{u}f(t,x,u+\mathfrak{t} h)-\partial^2_{u}f(t,x,u))~\mathrm{d}\mathfrak{t}.
\end{align*}
The same discussion as above yields that $F$ is twice differentiable with respect to $u$ and that $u\mapsto {\rm D}^2_uF(u)$ is continuous.
\end{proof}

\subsection{The proof of Lemma \ref{Lemma6.6}}
\begin{proof}
By \eqref{E6.6}, let $\vartheta_i(\xi)=c^2 \frac{d_i(\psi)}{\rho(\psi)}(\xi),i=1,2$. Define
\[T_1 u:=\frac{\mathrm{d}^{2}}{\mathrm{d}{\xi^2}}u+\vartheta_1(\xi)u, \quad T_2 u:=\frac{\mathrm{d}^{2}}{\mathrm{d}{\xi^2}}u+\vartheta_2(\xi)u.\]
It follows from \eqref{E6.36}, Lemma \ref{lemma2.5}, $m\in H^1$ and $\rho\in H^3$ that
\[\|\vartheta_1u\|_{L^2}\leq\|\vartheta_1\|_{L^{\infty}}\|u\|_{L^2}\leq\left\|c^2 \frac{d_i(\psi)}{\rho(\psi)}\right\|_{H^1}\|u\|_{L^2}\leq \widetilde{\mathfrak{C}}\|u\|_{L^2}.\]
This indicates that $\vartheta\in\mathcal{L}(L^2,L^2)$. It is obvious that $T_1,T_2$ are self-adjoint using Theorem \ref{Theorem5.7}. By means of Theorem \ref{Theorem5.7}, Lemma \ref{lemma2.4} and the inverse Liouville substitution of \eqref{E6.26}, for all $(\epsilon,w)\in[\epsilon_1,\epsilon_2]\times\{W\cap H^s:\|w\|_{s}\leq r\}$, we derive
\begin{align*}
|\lambda_n(d_2)-\lambda_n(d_1)|&\leq\frac{1}{c^2}|\mu_{n}(\vartheta_2)-\mu_n(\vartheta_1)|
\leq\frac{1}{c^2}\|\vartheta_2-\vartheta_1\|_{\mathcal{L}(L^2,L^2)}
\leq\frac{1}{c^2}\|\vartheta_2-\vartheta_1\|_{L^{\infty}}\\&\leq\frac{1}{c^2}\|\vartheta_2-\vartheta_1\|_{H^1}
\leq\kappa_0\|d_2-d_1\|_{H^1}\\
&\leq\kappa(|\epsilon-\epsilon'|+\|w-w'\|_s).
\end{align*}
\end{proof}

\subsection{The proof of formula \eqref{E4.23}}
\begin{proof}
If $j>\max{\{{L}_0,{4 c\mathfrak{M}}\}}$, $\forall\epsilon\in[\epsilon_1,\epsilon_2]$ ,$\forall\|w\|_s\leq r$, then formula \eqref{E6.28} shows that either
\begin{align*}
\inf\left|\sqrt{\lambda_{j+1}(\epsilon,w)}-\sqrt{\lambda_j(\epsilon,w)}\right|&\geq\frac{1}{c}-\left|\sqrt{\lambda_{j+1}(\epsilon,w})-\frac{j+1}{c}\right|
-\left|\sqrt{\lambda_{j}(\epsilon,w)}-\frac{j}{c}\right|\\
&\geq\frac{1}{c}-\frac{2\mathfrak{M}}{j}>\frac{1}{2c}
\end{align*}
or
\begin{align*}
\inf\left|\sqrt{\lambda_{j+1}(\epsilon,w)}-\sqrt{\lambda_j(\epsilon,w)}\right|&\geq\frac{1}{c}-\left|\sqrt{\lambda_{j+1}(\epsilon,w)}-\frac{j+1+1/2}{c}\right|
-\left|\sqrt{\lambda_{j}(\epsilon,w)}-\frac{j+1/2}{c}\right|\\
&>\frac{1}{2c}
\end{align*}
holds. For $0\leq j\leq\max{\{L_0,{4c\mathfrak{M}}\}}$, the minimum
\begin{equation*}
\mathfrak{w}_j:=\min_{\stackrel{\epsilon\in[\epsilon_1,\epsilon_2]}{w\in\{W\cap H^s:\|w\|_{s}\leq r\}}}\left|\sqrt{\lambda_{j+1}(\epsilon,w)}-\sqrt{\lambda_j(\epsilon,w)}\right|
\end{equation*}
can be  obtained.
\end{proof}



\begin{thebibliography}{10}

\bibitem{bahri1980periodic}
A.~Bahri and H.~Br{\'e}zis.
\newblock Periodic solution of a nonlinear wave equation.
\newblock {\em Proc. Roy. Soc. Edinburgh Sect. A}, 85(3-4):313--320, 1980.

\bibitem{baldi2008forced}
P.~Baldi and M.~Berti.
\newblock Forced vibrations of a nonhomogeneous string.
\newblock {\em SIAM J. Math. Anal.}, 40(1):382--412, 2008.

\bibitem{Bamberger1979about}
A.~Bamberger, G.~Chavent, and P.~Lailly.
\newblock About the stability of the inverse problem in 1-d wave
  equations―applications to the interpretation of seismic profiles.
\newblock {\em Appl. {M}ath. {O}ptim.}, 5(1):1--47, 1979.

\bibitem{Barbu1996Periodic}
V.~Barbu and N.~H. Pavel.
\newblock Periodic solutions to one-dimensional wave equation with piece-wise
  constant coefficients.
\newblock {\em J. Differential Equations}, 132(2):319--337, 1996.

\bibitem{Barbu1997determining}
V.~Barbu and N.~H. Pavel.
\newblock Determining the acoustic impedance in the 1-d wave equation via an
  optimal control problem.
\newblock {\em SIAM J. Control Optim}, 35(5):2035--2048, 1997.

\bibitem{Barbu1997Periodic}
V.~Barbu and N.~H. Pavel.
\newblock Periodic solutions to nonlinear one-dimensional wave equation with
  x-dependent coefficients.
\newblock {\em Trans. Amer. Math. Soc.}, 349(5):2035--2048, 1997.

\bibitem{berti2006cantor}
M.~Berti and P.~Bolle.
\newblock Cantor families of periodic solutions for completely resonant
  nonlinear wave equations.
\newblock {\em Duke Math. J.}, 134(2):359--419, 2006.

\bibitem{berti2008cantor}
M.~Berti and P.~Bolle.
\newblock Cantor families of periodic solutions of wave equations with ${C}^k$
  nonlinearities.
\newblock {\em No{DEA} {N}onlinear differ. equ. appl.}, 15(1-2):247--276, 2008.

\bibitem{berti2010sobolev}
M.~Berti and P.~Bolle.
\newblock Sobolev periodic solutions of nonlinear wave equations in higher
  spatial dimensions.
\newblock {\em Arch. Ration. Mech. Anal.}, 195(2):609--642, 2010.

\bibitem{Berti2012nonlinearity}
M.~Berti and P.~Bolle.
\newblock Sobolev quasi-periodic solutions of multidimensional wave equations
  with a multiplicative potential.
\newblock {\em Nonlinearity}, 25(9):2579--2613, 2012.

\bibitem{berti2015abstract}
M.~Berti, L.~Corsi, and M.~Procesi.
\newblock An abstract {N}ash-{M}oser theorem and quasi-periodic solutions for
  {NLW} and {NLS} on compact {L}ie groups and homogeneous manifolds.
\newblock {\em Comm. Math. Phys.}, 334(3):1413--1454, 2015.

\bibitem{bourgain1994construction}
J.~Bourgain.
\newblock Construction of quasi-periodic solutions for {H}amiltonian
  perturbations of linear equations and applications to nonlinear {PDE}.
\newblock {\em Internat. Math. Res. Notices}, (11):475ff., approx.\ 21 pp.\
  (electronic), 1994.

\bibitem{bourgain1995construction}
J.~Bourgain.
\newblock Construction of periodic solutions of nonlinear wave equations in
  higher dimension.
\newblock {\em Geom. Funct. Anal.}, 5(4):629--639, 1995.

\bibitem{brezis1983periodic}
H.~Br{\'e}zis.
\newblock Periodic solutions of nonlinear vibrating strings and duality
  principles.
\newblock {\em Bull. Amer. Math. Soc. (N.S.)}, 8(3):409--426, 1983.

\bibitem{brezis1981periodic}
H.~Br{\'e}zis and J.-M. Coron.
\newblock Periodic solutions of nonlinear wave equations and {H}amiltonian
  systems.
\newblock {\em Amer. J. Math.}, 103(3):559--570, 1981.

\bibitem{brezis1978forced}
H.~Br{\'e}zis and L.~Nirenberg.
\newblock Forced vibrations for a nonlinear wave equation.
\newblock {\em Comm. Pure Appl. Math.}, 31(1):1--30, 1978.

\bibitem{chierchia2000kam}
L.~Chierchia and J.~You.
\newblock K{AM} tori for 1{D} nonlinear wave equations with periodic boundary
  conditions.
\newblock {\em Comm. Math. Phys.}, 211(2):497--525, 2000.

\bibitem{coddington1972theory}
E.~A. Coddington and N.~Levinson.
\newblock {\em Theory of Ordinary Differential Equations}.
\newblock Mc{G}raw-{H}ill, {N}ew {Y}ork, 1955.

\bibitem{craig1993newton}
W.~Craig and C.~E. Wayne.
\newblock Newton's method and periodic solutions of nonlinear wave equations.
\newblock {\em Comm. Pure Appl. Math.}, 46(11):1409--1498, 1993.

\bibitem{Delort2011}
J.-M. Delort.
\newblock Periodic solutions of nonlinear {S}chr\"odinger equations: a
  paradifferential approach.
\newblock {\em Anal. PDE}, 4(5):639--676, 2011.

\bibitem{eliasson1988perturbations}
L.~H. Eliasson.
\newblock Perturbations of stable invariant tori for {H}amiltonian systems.
\newblock {\em Ann. Scuola Norm. Sup. Pisa Cl. Sci. (4)}, 15(1):115--147
  (1989), 1988.

\bibitem{mickan1995periodic}
M.~Fe{\v{c}}kan.
\newblock Periodic solutions of certain abstract wave equations.
\newblock {\em Proc. Amer. Math. Soc.}, 123(2):465--470, 1995.

\bibitem{Gao2009}
Y.~Gao, Y.~Li, and J.~Zhang.
\newblock Invariant tori of nonlinear {S}chr\"odinger equation.
\newblock {\em J. Differential Equations}, 246(8):3296--3331, 2009.

\bibitem{gengyou2006kam}
J.~Geng and J.~You.
\newblock A {KAM} theorem for {H}amiltonian partial differential equations in
  higher dimensional spaces.
\newblock {\em Comm. Math. Phys.}, 262(2):343--372, 2006.

\bibitem{Geng2013}
J.~Geng and Z.~Zhao.
\newblock Quasi-periodic solutions for one-dimensional nonlinear lattice
  {S}chr\"odinger equation with tangent potential.
\newblock {\em SIAM J. Math. Anal.}, 45(6):3651--3689, 2013.

\bibitem{ji2008time}
S.~Ji.
\newblock Time periodic solutions to a nonlinear wave equation with x-dependent
  coefficients.
\newblock {\em Calc. {V}ar. {P}artial {D}ifferential {E}quations},
  32(2):137--153, 2008.

\bibitem{ji2009peridic}
S.~Ji.
\newblock Time-periodic solutions to a nonlinear wave equation with periodic or
  anti-periodic boundary conditions.
\newblock {\em Proc. R. Soc. Lond. Ser. A Math. Phys. Eng. Sci.},
  465(2103):895--913, 2009.

\bibitem{ji2006periodic}
S.~Ji and Y.~Li.
\newblock Periodic solutions to one-dimensional wave equation with
  {$x$}-dependent coefficients.
\newblock {\em J. Differential Equations}, 229(2):466--493, 2006.

\bibitem{ji2007periodic}
S.~Ji and Y.~Li.
\newblock Time-periodic solutions to the one-dimensional wave equation with
  periodic or anti-periodic boundary conditions.
\newblock {\em Proc. Roy. Soc. Edinburgh Sect. A}, 137(2):349--371, 2007.

\bibitem{ji2011time}
S.~Ji and Y.~Li.
\newblock Time periodic solutions to the one-dimensional nonlinear wave
  equation.
\newblock {\em Arch. Ration. Mech. Anal.}, 199(2):435--451, 2011.

\bibitem{kato1995perturbation}
T.~Kato.
\newblock {\em Perturbation theory for linear operators}.
\newblock Classics in Mathematics. Springer-Verlag, Berlin, 1995.
\newblock Reprint of the 1980 edition.

\bibitem{kuksin1987hamiltonian}
S.~B. Kuksin.
\newblock Hamiltonian perturbations of infinite-dimensional linear systems with
  imaginary spectrum.
\newblock {\em Funktsional. Anal. i Prilozhen.}, 21(3):22--37, 95, 1987.

\bibitem{liu2010spectrum}
J.~Liu and X.~Yuan.
\newblock Spectrum for quantum {D}uffing oscillator and small-divisor equation
  with large-variable coefficient.
\newblock {\em Comm. Pure Appl. Math.}, 63(9):1145--1172, 2010.

\bibitem{mckenna1985Osolutions}
P.~J. McKenna.
\newblock On solutions of a nonlinear wave question when the ratio of the
  period to the length of the interval is irrational.
\newblock {\em Proc. Amer. Math. Soc.}, 93(1):59--64, 1985.

\bibitem{Rabinowitz1967periodic}
P.~H. Rabinowitz.
\newblock Periodic solutions of nonlinear hyperbolic partial differential
  equations.
\newblock {\em Comm. Pure Appl. Math.}, 20:145--205, 1967.

\bibitem{rabinowitz1968periodic}
P.~H. Rabinowitz.
\newblock Periodic solutions of nonlinear hyperbolic partial differential
  equations. {II}.
\newblock {\em Comm. Pure Appl. Math.}, 22:15--39, 1968.

\bibitem{Rabinowitz1971}
P.~H. Rabinowitz.
\newblock Time periodic solutions of nonlinear wave equations.
\newblock {\em Manuscripta Math.}, 5:165--194, 1971.

\bibitem{Rabinowitz1978}
P.~H. Rabinowitz.
\newblock Free vibrations for a semilinear wave equation.
\newblock {\em Comm. Pure Appl. Math.}, 31(1):31--68, 1978.

\bibitem{wayne1990periodic}
C.~E. Wayne.
\newblock Periodic and quasi-periodic solutions of nonlinear wave equations via
  {KAM} theory.
\newblock {\em Comm. Math. Phys.}, 127(3):479--528, 1990.

\bibitem{yuan2006quasi-periodic}
X.~Yuan.
\newblock Quasi-periodic solutions of completely resonant nonlinear wave
  equations.
\newblock {\em J. Differential Equations}, 230(1):213--274, 2006.

\end{thebibliography}


\end{document}